\documentclass[times,review]{elsarticle}
\journal{Journal of Computational Physics}
\usepackage{framed,multirow}

\usepackage[utf8]{inputenc}

\usepackage{graphicx}
\usepackage[version=4]{mhchem}
\usepackage{longtable,tabularx}
\usepackage{amsmath,amssymb,amsthm}

\usepackage{csquotes}

\usepackage{geometry}
\usepackage{float}

\usepackage{url}
\usepackage{pifont}
\usepackage{empheq}
\usepackage{bm}

\usepackage{scalerel}

\usepackage{amsthm}

\usepackage{csquotes}

\usepackage{booktabs}
\usepackage{geometry}
\usepackage{float}

\usepackage{url}
\usepackage{pifont}

\usepackage{empheq}
\usepackage{bm}
\usepackage[english]{babel} 
\addto\captionsenglish{\renewcommand{\appendixname}{Appendix }}
\usepackage{amsmath,amssymb}
\usepackage{epstopdf}
\usepackage{relsize}
\usepackage{stmaryrd} 
\usepackage{subcaption}
\usepackage{multirow}
\usepackage{array,makecell}
\usepackage[squaren, Gray, cdot]{SIunits}

\usepackage{import}
\usepackage{diagbox}

\usepackage{csquotes}

\usepackage{xspace}

\usepackage{scalerel}

\addto\captionsenglish{\renewcommand{\appendixname}{}}

\makeatletter
\def\ps@pprintTitle{%
 \let\@oddhead\@empty
 \let\@evenhead\@empty
 \def\@oddfoot{}%
 \let\@evenfoot\@oddfoot}

\usepackage{hyperref}

\usepackage{xcolor}

\DeclareMathOperator{\diag}{diag}

\newtheorem{thrm}{Theorem}

\newtheorem{remark}{Remark}
\newtheorem{lemma}{Lemma}

\newcommand{\Tr}{\ensuremath{^{\mr{T}}}}

\newcommand{\mr}[1]{\ensuremath{\mathrm{#1}}}
\newcommand{\fnc}[1]{\ensuremath{\mathcal{#1}}}

\newcommand{\etal}[0]{{\em et~al.\@}\xspace}

\newcommand{\ie}[0]{{i.e.\@}\xspace}

\newcommand{\xil}[0]{\ensuremath{\xi_{l}}}

\newcommand{\matAlmtwol}[1]{\ensuremath{\left[\fnc{J}\frac{\partial\xil}{\partial x_{m}}\right]_{2l}}}
\newcommand{\matAlmtwolmone}[1]{\ensuremath{\left[\fnc{J}\frac{\partial\xil}{\partial x_{m}}\right]_{(2l-1)}}}
\newcommand{\matAlmtwoltilde}[1]{\ensuremath{\widetilde{\left[\fnc{J}\frac{\partial\xil}{\partial x_{m}}\right]}_{2l}}}
\newcommand{\matAlmtwolmonetilde}[1]{\ensuremath{\widetilde{\left[\fnc{J}\frac{\partial\xil}{\partial x_{m}}\right]}_{(2l-1)}}}

\makeatletter
\def\ps@pprintTitle{%
 \let\@oddhead\@empty
 \let\@evenhead\@empty
 \def\@oddfoot{}%
 \let\@evenfoot\@oddfoot}

\definecolor{newcolor}{rgb}{.8,.349,.1}

\begin{document}


\begin{frontmatter}

\title{Provably Stable Flux Reconstruction High-Order Methods on Curvilinear Elements}



\author[1]{Alexander {Cicchino}\corref{cor1}\fnref{CICC}}
\cortext[cor1]{Corresponding author. 
}
  \ead{alexander.cicchino@mail.mcgill.ca}
  \fntext[CICC]{Ph.D. Student}
  
\author[2]{David C. {Del Rey Fern\'andez}\fnref{DCDRF}} 
\ead{dcdelrey@gmail.com}
\fntext[DCDRF]{Assistant Professor}
  
\author[1]{Siva {Nadarajah}\fnref{Nadarajah}}
\fntext[Nadarajah]{Associate Professor}

\ead{siva.nadarajah@mcgill.ca}

\author[3]{Jesse {Chan}\fnref{Chan}}
\fntext[Chan]{Assistant Professor}
\author[4]{Mark H. {Carpenter}\fnref{MHC}}
\fntext[MHC]{Senior Research Scientist}

\address[1]{Department of Mechanical Engineering, McGill University, 
Montreal, QC, H3A 0C3, Canada}
\address[2]{Department of Applied Mathematics, University of Waterloo, 
Waterloo, ON, N2L, Canada}
\address[3]{Department of Computational and Applied Mathematics, Rice University, 
Houston, TX, 77005, USA}
\address[4]{Computational AeroSciences Branch, NASA Langley Research Center (LaRC), 
Hampton, VA, 23666, USA}


\tnotetext[tnote1]{Preprint Submitted to the Journal of Computational Physics}




\begin{abstract}
Provably stable flux reconstruction (FR) schemes are derived for partial differential equations cast in curvilinear coordinates. Specifically, energy stable flux reconstruction (ESFR) schemes are considered as they allow for design flexibility as well as stability proofs for the linear advection problem on affine elements. Additionally, split forms are examined as they enable the development of energy stability proofs. The first critical step proves, that in curvilinear coordinates, the discontinuous Galerkin (DG) conservative and non-conservative forms are inherently different--even under exact integration and analytically exact metric terms. This analysis demonstrates that the split form is essential to developing provably stable DG schemes on curvilinear coordinates and motivates the construction of metric dependent ESFR correction functions in each element. Furthermore, the provably stable FR schemes differ from schemes in the literature that only apply the ESFR correction functions to surface terms or on the conservative form, and instead incorporate the ESFR correction functions on the full split form of the equations. It is demonstrated that the scheme is divergent when the correction functions are only used for surface reconstruction in curvilinear coordinates. We numerically verify the stability claims for our proposed FR split forms and compare them to ESFR schemes in the literature. Lastly, the newly proposed provably stable FR schemes are shown to obtain optimal orders of convergence. The scheme loses the orders of accuracy at the equivalent correction parameter value $c$ as that of the one-dimensional ESFR scheme. 
\end{abstract}
\end{frontmatter}

\section{Introduction}

The Flux Reconstruction (FR) framework, originally proposed by Huynh~\cite{huynh_flux_2007} 
(also referred to as lifting collocation penalty~\cite{wang2009unifying} or correction procedure via reconstruction~\cite{huynh2014high}), 
has emerged as a popular FEM approach that is both simple as it can be cast in a differential collocated form and affords design flexibility, where through a choice of the correction functions, the properties of the scheme can be altered. Importantly, subsets of FR schemes 
have been identified as provably linearly stable (see Refs.~\cite{vincent_insights_2011,williams_energy_2014,vincent_extended_2015,castonguay_phd,castonguay_new_2012}) also known as Vincent-Castonguay-Jameson-Huynh (VCJH) schemes or Energy Stable Flux Reconstruction (ESFR). Unfortunately, these proofs are limited to affine elements and hence do not apply to general curvilinear meshes.

A discretization agnostic approach for the design and analysis of arbitrarily high-order and provably stable numerical methods for linear variable coefficient problems is provided by the summation-by-parts (SBP)
framework~\cite{fernandez2014review,svard2014review,fernandez2014generalized}. SBP operators are matrix difference operators {\color{black}that are mimetic} 
to high-order 
integration by parts and when combined with appropriate interface coupling procedures, for example simultaneous approximation terms (SATs)~\cite{carpenter1994time,carpenter1999stable,nordstrom1999boundary,nordstrom2001high,carpenter2010revisiting,svard2014entropy,parsani2015entropy,parsani2015entropyWall}, lead to provably stable and conservative methods. 
FR has been cast in SBP form~\cite{ranocha2016summation,ranocha2017extended,montoya2021unifying} as well in residual distribution schemes~\cite{abgrall2019reinterpretation,abgrall2018general,abgrall2018connection,abgrall2021analysis} paving the way for a common framework to analyze high-order schemes. 
{\color{black} Moreover, discretizations having the SBP property form the foundations for nonlinearly stable schemes for nonlinear conservation laws}
~\cite{fisher2012high,fisher2013discretely,fisher2013high,carpenter2014entropy,parsani2015entropy,parsani2015entropyWall,parsani2016entropy,carpenter2016towards,yamaleev2017family,crean2018entropy,chen2017entropy,Crean2019Staggered,chan2018discretely,FriedrichEntropy2020}.

The focus of this article is on the construction of provably stable flux reconstruction schemes in curvilinear coordinates. Since the publication by Sv\"ard~\cite{svard2004coordinate}, the extension of stability proofs for dense-norm SBP operators, to variable coefficient problems---particularly curvilinear coordinate transformations, has received little attention in the SBP literature.
Sv\"ard~\cite{svard2004coordinate} proved that 
when dense-norms, $\bm{M}$, are multiplied against a diagonal matrix containing the metric Jacobian on the mesh nodes, $\bm{J}$, the result is not a norm, \ie, $\bm{M}\bm{J}$ is not in general a norm, and therefore provable stability is lost.
However, by recasting dense-norm SBP operators in staggered form and constructing metrics on the staggered grid, stability can be recovered for partial differential equations (PDE) in curvilinear coordinates discretized using dense-norm SBP operators~\cite{Fernandez2019curvitensor}. Alternatively, Ranocha~\etal~\cite{ranocha2017extended} demonstrated in one-dimension, that for modal based operators 
the issue with dense-norms can be overcome by using a dense matrix, $\Tilde{\bm{J}}$, such that 
$\bm{M}\Tilde{\bm{J}}=\left(\bm{M}\Tilde{\bm{J}}\right)\Tr$.
In a somewhat analogous way, the extension of stability proofs of ESFR schemes to curvilinear coordinates has been unclear since the ESFR norm is dense.
In this paper, taking inspiration from the developments in the SBP literature and starting from the variational form, we demonstrate how to incorporate metric Jacobian dependence in dense norms, specifically those arising in ESFR schemes.
In variational form, it is immediately seen that including metric Jacobian dependence does not merely correspond to right multiplying the norm matrix, but instead having the determinant of the Jacobian embedded within the integral; since the metric Jacobian is always built on the quadrature nodes and arises in the integral by transforming from the physical to the reference domain. This allows us to formulate the metric Jacobian dependent ESFR filter and the metric dependent ESFR 
correction functions. 

The overarching objective of this paper is to develop provably stable FR discretizations on curvilinear coordinates {\color{black}for systems of partial differential equations}.
As highlighted by the SBP community~\cite{aalund2019encapsulated,wintermeyer2017entropy,Fernandez2019curvitensor}, discrete integration by parts is not satisfied in the physical space for curvilinear coordinates. 
This is due to the physical flux never explicitly being represented by an interpolating polynomial in the physical space~\cite{moxey2019interpolation}.
This distinction, to the authors' knowledge, has not been investigated within the ESFR and DG communities~\cite{mengaldo2016connections,NDG,karniadakis2013spectral,zwanenburg_equivalence_2016}. 
The DG strong form in reference space can be derived by either an application of integration by parts on the DG weak form in reference space or in physical space. Since the two strong DG discretizations are not  equivalent~\cite{teukolsky2016formulation}, we present the split form in order to mimic integration by parts in the physical space. A critical result that has not been shown in the SBP literature~\cite{Fernandez2019curvitensor,chan2019discretely,aalund2019encapsulated}, is that the two DG strong forms are not equivalent even under exact integration and analytically exact metric terms, {\color{black}making} the split form essential for curvilinear high-order schemes.

{\color{black}In this article, we derive provably stable FR schemes on curvilinear coordinates and consider various design decisions: modal or nodal basis, uncollocated integration, different ESFR correction functions, and different volume and surface quadrature nodes.} 
The first main insight is that the ESFR stability condition~\cite{vincent_new_2011,castonguay2012newTRI,williams_energy_2014,Cicchino2020NewNorm} must contain metric dependence in curvilinear coordinates.
Then, we demonstrate that stability cannot be achieved when the ESFR correction functions are solely used to reconstruct the flux on the surface. This issue has been presented on linear grids for Burgers' equation by Ranocha~\etal~\cite{ranocha2016summation} and for Euler's equations by Abe~\etal~\cite{abe2018stable}, although neither have found a solution to satisfy stability for general ESFR in split form. In~\cite{ranocha2016summation}, the authors investigated the issue of the dense ESFR contribution to the split forms, where they proved stability only for the DG case. In~\cite{abe2018stable}, the authors' numerically demonstrated stability for the \enquote{$\text{g}_2$ lumped-Lobatto} ESFR scheme, which is equivalent to a collocated DG scheme on Gauss-Lobatto-Legendre nodes~\cite{de2014connections} and previously shown to be stable in split form by Gassner~\cite{gassner2016split}. Following the general nonlinearly stable FR framework developed in Cicchino \textit{et al.}~\cite{CicchinoNonlinearlyStableFluxReconstruction2021} {\color{black}for Burgers' equation}, the ESFR filter/divergence of the correction functions is incorporated on the nonlinear volume terms to ensure nonlinear stability within the broken Sobolev-norm. {\color{black} This differs from the literature where the ESFR correction functions were only used to reconstruct the flux on the surface~\cite{huynh_flux_2007,wang2009unifying,vincent_new_2011,jameson_non-linear_2012,castonguay_energy_2013,huynh2014high,williams_energy_2014,sheshadri2016stability,abe2018stable}. 
In addition, the proposed scheme is in contrast from schemes where the ESFR norm\footnote{By ESFR norm, we refer to the $(\bm{M}+\bm{K})$ modified Mass matrix form in Allaneau and Jameson~\cite[Eq.(13)]{allaneau_connections_2011}} was applied to the conservative discretization; either filtering the strong form surface integral~\cite{allaneau_connections_2011,zwanenburg_equivalence_2016,ranocha2016summation,Cicchino2020NewNorm}, or filtering the entire weak form~\cite{montoya2021unifying}; since such stated schemes are only linearly stable.
}

The remainder of this article is organized as follows: In Section 2, we introduce the mathematical notations, definitions of metrics, and establish the relationships between the physical and reference spaces. In Section 3, the DG scheme is derived in both conservative and non-conservative strong forms. We subsequently prove that the two forms are inherently different under exact integration and metric terms, and introduce the DG split form. In Section 4, the classical ESFR scheme is established, and the proposed novel nonlinearly stable FR scheme is derived. In subsequent Sections 5 and 6, we provide proofs of the free-stream preservation, local and global conservation, and stability of the proposed stable FR split form. The theoretical results are numerically verified in Sec.~\ref{sec:Results}, where the classical ESFR scheme in split form (ESFR filter only applied to the facet terms) diverges while, our proposed ESFR split form (ESFR filter applied to facet and volume terms) remains stable and maintains the correct orders of accuracy.

\section{Math Notation and Definitions}\label{sec: Math notation}

Consider the scalar 3D conservation law,
\begin{equation}
\begin{aligned}
    \frac{\partial}{\partial t}u(\bm{x}^c,t) +\nabla\cdot\bm{f}(u(\bm{x}^c,t) )=0,\: t\geq 0,\:\bm{x}^c\coloneqq[x \text{ }y\text{ }z]\in\Omega,\\
   u(\bm{x}^c,0)=u_0(\bm{x}^c),
\end{aligned}
    \nonumber
\end{equation}

\noindent where $\bm{f}(u(\bm{x}^c,t) )\in\mathbb{R}^{1\times d}$ stores the fluxes in each of the $d$ directions, and the superscript $c$ refers to Cartesian coordinates. In this paper row vector notation will be used. 
The computational domain $\bm{\Omega}^h$ is partitioned into $M$ non-overlapping elements, $\bm{\Omega}_m$, where the domain is represented by the union of the elements, \textit{i.e.}
\begin{equation}
    \bm{\Omega} \simeq \bm{\Omega}^h \coloneqq \bigcup_{m=1}^M \bm{\Omega}_m.
    \nonumber
\end{equation}
Each element $m$ has a surface denoted by $\bm{\Gamma}_m$. The global approximation, $u^h(\bm{x}^c,t)$, is constructed from the direct sum of each local approximation, $u_m^h(\bm{x}^c,t)$, \textit{i.e.} 
\begin{equation}
    u(\bm{x}^c,t) \simeq u^h(\bm{x}^c,t)=\bigoplus_{m=1}^M u_m^h(\bm{x}^c,t).
    \nonumber
\end{equation}
Throughout this paper, all quantities with a subscript $m$ are specifically unique to the element $m$.
{\color{black}On each element, we represent the solution with $N_p$ linearly independent modal or nodal basis of a maximum order of $p$; where, $N_p\coloneqq(p+1)^d$. 
The solution representation is},
$
     u_m^h(\bm{x}^c,t) \coloneqq  \sum_{i=1}^{N_p} {\chi}_{m,i}(\bm{x}^c)\hat{u}_{m,i}(t).
    \nonumber
$ 
The elementwise residual is,
\begin{equation}
     R_m^h(\bm{x}^c,t)=\frac{\partial}{\partial t}u_m^h(\bm{x}^c,t) + \nabla \cdot \bm{f}(u_m^h(\bm{x}^c,t)).\label{eq: residual}
\end{equation}

\noindent The basis functions in each element are defined as,
\begin{equation}
     \bm{\chi}(\bm{x}^c) \coloneqq [\chi_{1}(\bm{x}^c)\text{, }\chi_{2}(\bm{x}^c)\text{, } \dots\text{, } \chi_{N_p}(\bm{x}^c)] = \bm{\chi}({x})\otimes \bm{\chi}({y}) \otimes \bm{\chi}({z}) \in \mathbb{R}^{1\times N_p},
\end{equation}

\noindent where $\otimes$ is the tensor product.

The physical coordinates are mapped to the reference element $ \bm{\xi}^r \coloneqq \{ [\xi \text{, } \eta \text{, }\zeta]:-1\leq \xi,\eta,\zeta\leq1 \}$ by

\begin{equation}
      \bm{x}_m^c( \bm{\xi}^r)\coloneqq  \bm{\Theta}_m( \bm{\xi}^r)=
     \sum_{i=1}^{N_{t,m}} {\Theta}_{m,i}( \bm{\xi}^r)\hat{ \bm{x}}_{m,i}^c,
\end{equation}
where $ {\Theta}_{m,i}$ 
are the mapping shape functions of the $N_{t,m}$ physical mapping control points $\hat{\bm{x}}_{m ,i}^c $.

To transform Eq.~(\ref{eq: residual}) to the reference basis, {\color{black}as in refs}~\cite{gassner2018br1,manzanero2019entropy,kopriva2006metric,thomas1979geometric,vinokur2002extension}, we introduce the physical 

\begin{equation}
    \bm{a}_j \coloneqq \frac{\partial \bm{x}^c}{\partial \xi^j},\text{ } j=1,2,3
    \nonumber
\end{equation}

\noindent and reference 

\begin{equation}
    \bm{a}^j \coloneqq \nabla\xi^j,\text{ } j=1,2,3
    \nonumber
\end{equation}

\noindent vector bases. We then introduce the determinant of the metric Jacobian as

\begin{equation}
  J^\Omega \coloneqq  |\bm{J}^\Omega| =  \bm{a}_1\cdot (\bm{a}_2\times \bm{a}_3),
\end{equation}

\noindent and the metric Jacobian cofactor matrix as ~\cite{gassner2018br1,zwanenburg_equivalence_2016,manzanero2019entropy,kopriva2019free},

\begin{equation}
    \bm{C}^T \coloneqq J^\Omega(\bm{J}^\Omega)^{-1}=\begin{bmatrix}J^\Omega\bm{a}^1\\
    J^\Omega\bm{a}^2\\
    J^\Omega\bm{a}^3 \end{bmatrix}
    = \begin{bmatrix}
    J^\Omega\bm{a}^\xi\\
    J^\Omega\bm{a}^\eta\\
    J^\Omega\bm{a}^\zeta
    \end{bmatrix}.
\end{equation}

The metric cofactor matrix is formulated by the \enquote{conservative curl} form from~\cite[Eq. 36]{kopriva2006metric} {\color{black}so }as to discretely satisfy the Geometric Conservation Law (GCL) 

\begin{equation}
   \sum_{i=1}^{3}\frac{\partial (J^\Omega (\bm{a}^i)_n)}{\partial \xi^i}=0\text{, }n=1,2,3 \Leftrightarrow \sum_{i=1}^{3}\frac{\partial}{\partial \xi^i}
    (\bm{C})_{ni}=0\text{, }n=1,2,3
    \Leftrightarrow \nabla^r\cdot(\bm{C})=\bm{0},
    \label{eq: GCL}
\end{equation}

\noindent which is detailed in Sec.~\ref{sec: Discrete GCL} {\color{black}for a fixed mesh}, where $(\; )_{ni}$ represents the $n^{\text{th}}$ row, $i^{\text{th}}$ column component of a matrix.

{\color{black}Having established the transformations mapping the physical to the reference coordinates on each element, the differential volume and surface elements can be defined as,}

\begin{equation}
    d\bm{\Omega}_m = J_m^\Omega d\bm{\Omega}_r\text{, similarly }d\bm{\Gamma}_m = J_m^\Gamma d\bm{\Gamma}_r.\label{eq: diff element def}
\end{equation}

\noindent The reference flux for each element $m$ is defined as

\begin{equation}
    \bm{f}_m^r = \bm{C}_m^T\cdot \bm{f}_m 
   \Leftrightarrow \bm{f}_{m,j}^r =\sum_{i=1}^d (\bm{C}_{m}^T)_{ji}\bm{f}_{m,i} \Leftrightarrow \bm{f}_m^r
   =\bm{f}_{m}\bm{C}_m,
   \label{eq: contravariant flux def}
\end{equation}

\noindent where the dot product notation for tensor-vector operations is introduced. The relationship between the physical and reference unit normals is given as~\cite[Appendix B.2]{zwanenburg_equivalence_2016},

\begin{equation}
    \hat{\bm{n}}_m = \frac{1}{J_m^\Gamma}\bm{C}_m\cdot \hat{\bm{n}}^r 
    =\frac{1}{J_m^\Gamma} \hat{\bm{n}}^r \bm{C}_m^T,\label{eq: normals def}
\end{equation}

\noindent for a water-tight mesh. 
 Additionally, the definition of the divergence operator derived from divergence theorem in curvilinear coordinates can be expressed as~\cite[Eq. (2.22) and (2.26)]{gassner2018br1},

\begin{equation}
    \nabla\cdot \bm{f}_m 
    =\frac{1}{J_m^\Omega} \nabla^r \cdot\Big( \bm{f}_m\bm{C}_m \Big)= \frac{1}{J_m^\Omega} \nabla^r \cdot \bm{f}_m^r,\label{eq: divergence def}
\end{equation}

\noindent and the gradient of a scalar as~\cite[Eq. (2.21)]{gassner2018br1},

\begin{equation}
    \nabla {\chi} = \frac{1}{J_m^\Omega} \bm{C}_m \cdot \nabla^r \chi = \frac{1}{J_m^\Omega} \Big(\nabla^r \chi \Big) \bm{C}_m^T.\label{eq: gradient def}
\end{equation}


Thus, substituting Eq.~(\ref{eq: divergence def}) into Eq.~(\ref{eq: residual}), the reference elementwise residual is,

\begin{equation}
    R_m^{h,r}(\bm{\xi}^r,t)\coloneqq R_m^{h}(\bm{\Theta}_m(\bm{\xi}^r),t)= \frac{\partial}{\partial t}u_m^h(\bm{\Theta}_m(\bm{\xi}^r),t) + \frac{1}{J_m^\Omega}\nabla^r \cdot \bm{f}^r(u_m^h(\bm{\Theta}_m(\bm{\xi}^r),t)).\label{eq: residual reference}
\end{equation}
\section{Discontinuous Galerkin}

In this section we present a {\color{black}provably} stable DG discretization for curvilinear coordinates~\cite{Fernandez2019curvitensor,chan2019discretely,aalund2019encapsulated} to act as the cornerstone for our {\color{black}provably} stable FR schemes. We derive the DG strong form for both \enquote{conservative} and \enquote{non-conservative} formulations, and prove that they are inherently different for curvilinear coordinates; even with analytically exact metric terms and exact integration. This difference necessitates a split form to ensure nonlinear stability on curvilinear coordinates. Specifically we cover:

\begin{enumerate}
    \item Deriving the conservative DG strong form by transforming the physical DG weak form to reference space. Then, projecting the reference flux onto the reference polynomial basis, and finally, integrating the volume terms by parts in the reference space.
    \item Deriving the non-conservative DG strong form by projecting the physical flux onto a physical basis in the physical DG weak form. Then integrating the volume terms by parts in the physical space, and finally, transforming the physical DG non-conservative strong form to the reference space.
    \item Comparing the two forms, prove that they are inherently different, even under exact integration with analytically exact metric terms, and that discrete integration by parts in the physical space is not satisfied for either form. Then, combining the two forms to discretely \enquote{mimic} integration by parts in the physical space.
\end{enumerate}

\subsection{DG - Conservative Strong Form}\label{sec: DG Var}



In a Galerkin framework, we left multiply the physical residual Eq.~(\ref{eq: residual}) by an orthogonal test function. Choosing the test function to be the same as the basis function, integrating in physical space, and applying integration by parts in physical space, we arrive at the weak form, 

\begin{equation}
\begin{split}
    \int_{\bm{\Omega}_m} {\chi}_{m,i}(\bm{x}^c)\frac{\partial}{\partial t}u_m^h(\bm{x}^c,t) d \bm{\Omega}_m
    -\int_{\bm{\Omega}_m} \nabla{\chi}_{m,i}(\bm{x}^c) \cdot
    \bm{f}(u_m^h(\bm{x}^c,t)) d\bm{\Omega}_m 
    +\int_{\bm{\Gamma}_m}{\chi}_{m,i}(\bm{x}^c) \hat{\bm{n}}_m \cdot \bm{f}^*(u_m^h(\bm{x}^c,t))d \bm{\Gamma}_m = {0},\\ 
    \forall i=1,\dots,N_p
\end{split}\label{eq: Weak DG PHYS}
\end{equation}

\noindent where $\bm{f}^*(u_m^h(\bm{x}^c,t))$ represents the physical numerical flux.

Now, we transform the physical DG weak form, Eq.~(\ref{eq: Weak DG PHYS}), to the reference space, by using the definitions of the differential volume and surface elements, physical gradient operator and physical unit normals  (Equations~(\ref{eq: diff element def}),~(\ref{eq: normals def}),~(\ref{eq: gradient def})), 

\begin{equation}
    \begin{split}
    \int_{\bm{\Omega}_r} {\chi}_{i}(\bm{\xi}^r) J_m^\Omega \frac{\partial}{\partial t} u_m^h(\bm{\Theta}_m(\bm{\xi}^r),t) d \bm{\Omega}_r
    -\int_{\bm{\Omega}_r} \Big(\frac{1}{J_m^\Omega} \nabla^r{\chi}_{i}(\bm{\xi}^r)\bm{C}_m^T\Big) 
    J_m^\Omega \cdot \bm{f}(u_m^h(\bm{\Theta}_m(\bm{\xi}^r),t)) d\bm{\Omega}_r \\
    +\int_{\bm{\Gamma}_r}{\chi}_{i}(\bm{\xi}^r) J_m^\Gamma \frac{1}{J_m^\Gamma} 
    \hat{\bm{n}}^r \bm{C}_m^T \cdot \bm{f}^*(u_m^h(\bm{\Theta}_m(\bm{\xi}^r),t))d \bm{\Gamma}_r = {0},
    \:\forall i=1,\dots,N_p.
\end{split}\label{eq: Weak DG transform}
\end{equation}

\noindent 

\noindent Notice the change of variables since $\bm{\chi}_m(\bm{x}^c)\coloneqq \bm{\chi}(\bm{\Theta}_m^{-1}(\bm{x}^c))$ are implicitly defined through polynomial basis functions in the reference space. From the definition Eq.~(\ref{eq: contravariant flux def}), the reference flux is substituted for $\bm{C}_m^T\cdot \bm{f}(u_m^h(\bm{\Theta}_m(\bm{\xi}^r),t))$ 
in the volume integral.
We then project the reference flux in Eq.~(\ref{eq: Weak DG transform}) onto the reference polynomial basis functions, 
and substitute the basis expansion for the solution. The variational DG weak form in reference space is thus,

\begin{equation}
\begin{split}
    \int_{\bm{\Omega}_r} {\chi}_i(\bm{\xi}^r) J_m^\Omega\bm{\chi}(\bm{\xi}^r) \frac{d}{d t} \hat{\bm{u}}_m(t)^T d \bm{\Omega}_r
    -\int_{\bm{\Omega}_r} \nabla^r{\chi}_i(\bm{\xi}^r) 
    \cdot \bm{\chi}(\bm{\xi}^r)\hat{\bm{f}}_m^r(t)^T d\bm{\Omega}_r 
    +\int_{\bm{\Gamma}_r} {\chi}_i(\bm{\xi}^r) 
    \hat{\bm{n}}^r\bm{C}_m^T \cdot \bm{f}_m^{*}(u_m^h(\bm{\Theta}_m(\bm{\xi}^r),t))d \bm{\Gamma}_r = {0},\\ \:\forall i=1,\dots,N_p.
\end{split}\label{eq: variational Weak DG}
\end{equation}

\noindent Next Eq.~(\ref{eq: variational Weak DG}), the reference DG weak form, is integrated by parts in the reference space resulting in, 

\begin{equation}
\begin{split}
    \int_{\bm{\Omega}_r} {\chi}_i(\bm{\xi}^r) J_m^\Omega\bm{\chi}(\bm{\xi}^r) \frac{d}{d t} \hat{\bm{u}}_m(t)^T d \bm{\Omega}_r
    +\int_{\bm{\Omega}_r} {\chi}_i(\bm{\xi}^r) \Bigg(\sum_{j=1}^{N_p}\nabla^r{\chi}_j(\bm{\xi}^r)
    \cdot \hat{\bm{f}}_{m,j}^r(t)\Bigg) d\bm{\Omega}_r 
    +\int_{\bm{\Gamma}_r}{\chi}_i(\bm{\xi}^r) \Big[ 
    \hat{\bm{n}}^r\bm{C}_m^T \cdot \bm{f}^*_m - \hat{\bm{n}}^r \cdot \bm{\chi}(\bm{\xi}^r)\bm{\hat{f}}^r_m(t)^T
    \Big]d \bm{\Gamma}_r = {0},\\ \forall i=1,\dots,N_p.
\end{split}\label{eq: variational Strong DG}
\end{equation}

\noindent In the general case, the interpolation of the nonlinear reference flux to the face does not equal the metric terms evaluated at the face multiplied with the flux on the face. 

{\color{black}Next, we introduce $N_{vp}$ volume and $N_{fp}$ facet cubature nodes, $\bm{\xi}_v^r$ and $\bm{\xi}_{f,k}^r$ respectively. We also introduce $\bm{W}$ and $\bm{J}_m$ as diagonal operators storing the quadrature weights and the determinant of the metric Jacobian at the volume cubature nodes.} We present the discretization of Eq.~(\ref{eq: variational Strong DG}), the discrete conservative DG strong form, as

\begin{equation}
    \bm{M}_m\frac{d}{d t} \hat{\bm{u}}_m(t)^T
    + \bm{S}_\xi \hat{\bm{f}}_{1m}^r(t)^T
    + \bm{S}_\eta \hat{\bm{f}}_{2m}^r(t)^T
    + \bm{S}_\zeta \hat{\bm{f}}_{3m}^r(t)^T
    +\sum_{f=1}^{N_f}\sum_{k=1}^{N_{fp}} \bm{\chi}(\bm{\xi}_{f,k}^r)^T {W}_{f,k}
    [\hat{\bm{n}}^r\bm{C}_m(\bm{\xi}_{f,k}^r)^T\cdot \bm{f}^*_m - \hat{\bm{n}}^r\cdot \bm{\chi}(\bm{\xi}_{f,k}^r)\hat{\bm{f}}^r_m(t)^T]
    =\bm{0}^T,\label{eq: Strong DG Discrete}
\end{equation}

\noindent where $N_f$ represents the number of faces on the element. 
The discrete mass and stiffness matrices are defined as,

\begin{equation}
    \begin{aligned}
      (\bm{M}_m)_{ij}\approx\int_{\Omega_r}J_m^\Omega {\chi}_i(\bm{\xi}^r){\chi}_j(\bm{\xi}^r)d\Omega_r \to \bm{M}_m= \bm{\chi}(\bm{\xi}_v^r)^T\bm{W}\bm{J}_m\bm{\chi}(\bm{\xi}_v^r),\\
      (\bm{S}_\xi)_{ij}=\int_{\Omega_r} {\chi}_i(\bm{\xi}^r)\frac{\partial}{\partial\xi}{\chi}_j(\bm{\xi}^r)d\Omega_r \rightarrow \bm{S}_\xi= \bm{\chi}(\bm{\xi}_v^r)^T\bm{W}\bm{\chi}_\xi(\bm{\xi}_v^r),
    \end{aligned}
    \nonumber
\end{equation}

\noindent and similarly for the other reference directions. The equality for the stiffness matrices holds for quadrature rules of at least $2p-1$ in  strength. Furthermore, we introduce the $\text{L}_2$ projection operator as $\bm{\Pi}\coloneqq\bm{M}^{-1}\bm{\chi}(\bm{\xi}_v^r)^T\bm{W}$, where the metric independent mass matrix is $\bm{M}=\bm{\chi}(\bm{\xi}_v^r)^T\bm{W}\bm{\chi}(\bm{\xi}_v^r)$. Thus, the modal coefficients of the reference flux are the $\text{L}_2$ projection of the reference flux, $\hat{\bm{f}}_m^r(t)^T=\bm{\Pi}(\bm{f}_m^{r^{T}})$. 

\subsection{DG - Non-Conservative Strong Form}\label{sec:DG non cons}

Returning to the physical DG weak form, Eq.~(\ref{eq: Weak DG PHYS}), as discussed in~\cite{moxey2019interpolation,botti2012influence}, there is no claim that the physical flux has a polynomial basis function expansion for curvilinear elements. 
{\color{black} We term the scheme \enquote{non-conservative} because it does not recover the definition of the reference divergence operator in Eq.~(\ref{eq: divergence def})~\cite{Fernandez2019curvitensor,chan2019discretely}.}
Following the approach in~\cite{NDG,karniadakis2013spectral}, we substitute the solution expansion and project the physical flux onto the basis functions. 
Eq.~(\ref{eq: Weak DG PHYS}) is integrated by parts in physical space and yields the \enquote{non-conservative} DG strong form in physical space,

\begin{equation}
\begin{split}
    \int_{\bm{\Omega}_m} {\chi}_{m,i}(\bm{x}^c) \bm{\chi}_{m}(\bm{x}^c)\frac{d}{d t}\hat{\bm{u}}_{m}(t)^T d \bm{\Omega}_m
    +\int_{\bm{\Omega}_m} {\chi}_{m,i}(\bm{x}^c)\Big( \sum_{j=1}^{N_p} \nabla {\chi}_{m,j}(\bm{x}^c) \cdot
    \hat{\bm{f}}_{m,j}(t)
    \Big) d\bm{\Omega}_m 
    +\int_{\bm{\Gamma}_m}{\chi}_{m,i}(\bm{x}^c) \hat{\bm{n}}_m \cdot 
    (\bm{f}^*_m - \bm{f}_m)
    d \bm{\Gamma}_m = {0},\\ \:\forall i=1,\dots,N_p.
\end{split}\label{eq: strong DG PHYS}
\end{equation}

To discretely represent the derivative of the physical flux in the physical space, it must be represented by the derivative of a basis expansion in the physical space multiplied by its modal coefficients. {\color{black}Although in the continuous sense 
the physical divergence operator could be recovered by commuting the basis functions across the dot product in Eq.~(\ref{eq: strong DG PHYS}); doing so would remove the claim that the physical flux has a basis function expansion.
Only in the reference space can the basis function be brought across the dot product since the derivative of a polynomial basis function on the reference element exist.}
Thus, discretely applying integration by parts in the physical space to arrive at Eq.~(\ref{eq: strong DG PHYS}) would necessitate that  $\bm{\chi}_m(\bm{x}^c)=\bm{\chi}(\bm{\Theta}_m^{-1}(\bm{x}^c))$ is a polynomial basis.
It is clear when transforming Eq.~(\ref{eq: strong DG PHYS}) to the reference space, and substituting the definition of the gradient for curvilinear elements, Eq.~(\ref{eq: gradient def}), that it is inconsistent with the previous formulation in Eq.~(\ref{eq: variational Strong DG}),

\begin{equation}
\begin{split}
    \int_{\bm{\Omega}_r} {\chi}_i(\bm{\xi}^r) J_m^\Omega\bm{\chi}(\bm{\xi}^r) \frac{d}{d t} \hat{\bm{u}}_m(t)^T d \bm{\Omega}_r
    +\int_{\bm{\Omega}_r} {\chi}_i(\bm{\xi}^r)\Bigg(\sum_{j=1}^{N_p} \Big(  \nabla^r
     {\chi}_j(\bm{\xi}^r)\bm{C}_m^T \Big)\cdot  \hat{\bm{f}}_{m,j}(t)\Bigg) d\bm{\Omega}_r 
    +\int_{\bm{\Gamma}_r} {\chi}_i(\bm{\xi}^r)  
   \hat{\bm{n}}^r\bm{C}_m^T\cdot (\bm{f}^*_m-\bm{f}_m)
    d \bm{\Gamma}_r = {0},\\ \: \forall i=1,\dots,N_p.
\end{split}\label{eq: variational Strong DG inconsistent}
\end{equation}

\noindent 
Explicitly, the metric cofactor matrix appears on the outside of the reference divergence/gradient operator. Only if the mesh is linear, skew-symmetric, or symmetric with uniform constant wave speeds for linear advection, will the volume integrals in Equations~(\ref{eq: variational Strong DG}) and~(\ref{eq: variational Strong DG inconsistent}) be equivalent in discrete form.

\begin{thrm}
    The volume terms in Eq.~(\ref{eq: variational Strong DG}) and Eq.~(\ref{eq: variational Strong DG inconsistent}) are inherently different for a curvilinear mesh
    ; even with exact integration and exact metric terms. \label{thm: difference cons and non cons}
\end{thrm} 

\begin{proof}

Consider just one of the divergence terms in the volume integral, 

\begin{equation}
    \begin{split}
        \text{Conservative DG: } \sum_{k=1}^{d} \int_{\bm{\Omega}_r} \chi_{i}(\bm{\xi}^r)
      \sum_{j=1}^{N_p}  \frac{\partial \chi_j(\bm{\xi}^r)}{\partial \xi_k}\Bigg[ \bm{\Pi}
        \Big({J}_{m}^\Omega \frac{\partial \xi_k}{\partial x} \bm{f}_{x}(u_m^h)   \Big)\Bigg]_j d\bm{\Omega}_r,\\
        \text{Non-Conservative DG: } \sum_{k=1}^{d} \int_{\bm{\Omega}_r} \chi_{i}(\bm{\xi}^r)
        J_m^\Omega \frac{\partial \xi_k}{\partial x}\sum_{j=1}^{N_p}
        \frac{\partial \chi_j(\bm{\xi}^r) }{\partial \xi_k} \Bigg[ \bm{\Pi}\Big( \bm{f}_x(u_m^h)  \Big)\Bigg]_j d\bm{\Omega}_r.
    \end{split}
\end{equation}

\noindent If we are to consider both exact integration and exact metric terms ($J_m^\Omega \frac{\partial \xi_k}{\partial x}$), then the two forms cannot be equivalent for a general $f_x(u_m^h)$.

\end{proof}

\begin{remark}
Only for the specific case of linear advection with a polynomial representation of the mesh can the two forms be equivalent through polynomial exactness; provided they are both exactly integrated and the nonlinear term in the conservative form is projected onto a sufficiently high polynomial space.
\end{remark}


\subsection{DG - Split Form}\label{sec: DG Split Form}

{\color{black}For the objective of developing provably stable schemes}, alike ~\cite{aalund2019encapsulated,fernandez2019entropy,chan2019discretely}, we introduce the split form by adding a half of the conservative DG Strong form Eq.~(\ref{eq: variational Strong DG}) with the non-conservative DG Strong form Eq.~(\ref{eq: variational Strong DG inconsistent}) to {\color{black}discretely satisfy integration by parts},

\begin{equation}
    \begin{split}
    \int_{\bm{\Omega}_r} {\chi}_i(\bm{\xi}^r) J_m^\Omega\bm{\chi}(\bm{\xi}^r) \frac{d}{d t} \hat{\bm{u}}_m(t)^T d \bm{\Omega}_r
    +\frac{1}{2}\int_{\bm{\Omega}_r} {\chi}_i(\bm{\xi}^r) \Bigg(\sum_{j=1}^{N_p}\nabla^r{\chi}_j(\bm{\xi}^r) 
    \cdot \hat{\bm{f}}_{m,j}^r(t) \Bigg)d\bm{\Omega}_r 
    +\frac{1}{2}\int_{\bm{\Omega}_r} {\chi}_i(\bm{\xi}^r) \Bigg(\sum_{j=1}^{N_p} \Big( \nabla^r
    {\chi}_j(\bm{\xi}^r)\bm{C}_m^T\Big) \cdot  \hat{\bm{f}}_{m,j}(t)\Bigg) d\bm{\Omega}_r\\
    +\int_{\bm{\Gamma}_r} {\chi}_i(\bm{\xi}^r)^T  \Big[
    \hat{\bm{n}}^r\bm{C}_m^T\cdot(\bm{f}^*_m-\frac{1}{2}\bm{f}_m)-\hat{\bm{n}}^r\cdot \frac{1}{2}\bm{\chi}(\bm{\xi}^r)\hat{\bm{f}}^r_m(t)^T\Big]
    d \bm{\Gamma}_r = {0},\: \forall i=1,\dots,N_p.
\end{split}\label{eq: var DG split}
\end{equation}

Note that the surface splitting naturally accommodates arbitrary sets of volume and facet cubature nodes. 
Recasting Eq.~(\ref{eq: var DG split}) into discrete form by evaluating at volume and facet cubature nodes, we have the DG split form,

\begin{equation}
\begin{split}
   \bm{M}_m\frac{d}{d t} \hat{\bm{u}}_m(t)^T
   +\frac{1}{2}\bm{\chi}(\bm{\xi}_v^r)^T\bm{W}\nabla^r\bm{\chi}(\bm{\xi}_v^r)\cdot \hat{\bm{f}}_m^r(t)^T
    +\frac{1}{2}\bm{\chi}(\bm{\xi}_v^r)^T\bm{W}
   \Tilde{\nabla}^r\bm{\chi}(\bm{\xi}_v^r)\cdot   \hat{\bm{f}}_{m}(t)^T 
    +\sum_{f=1}^{N_f}\sum_{k=1}^{N_{fp}} \bm{\chi}(\bm{\xi}_{f,k}^r)^T W_{f,k}[\hat{\bm{n}}^r\cdot \bm{f}^{C,r}_m]=\bm{0}^T,
    \end{split}\label{eq: DG split}
\end{equation}

\noindent where we introduced $\Tilde{\nabla}^r\bm{\chi}(\bm{\xi}_v^r) = \begin{pmatrix}
  \nabla^r{\chi}_1(\bm{\xi}_{v,1}^r)\bm{C}_m(\bm{\xi}_{v,1}^r)^T &\dots&   \nabla^r{\chi}_{N_{p}}(\bm{\xi}_{v,1}^r)\bm{C}_m(\bm{\xi}_{v,1}^r)^T\\
\vdots & \ddots & \vdots\\
 \nabla^r{\chi}_1(\bm{\xi}_{v,N_{vp}}^r) \bm{C}_m(\bm{\xi}_{v,N_{vp}}^r) ^T&\dots&   \nabla^r{\chi}_{N_{p}}(\bm{\xi}_{v,N_{vp}}^r)\bm{C}_m(\bm{\xi}_{v,N_{vp}}^r)^T
\end{pmatrix}$ to store the transformed reference gradient of the basis functions evaluated at volume cubature nodes. Also, we introduced\\ $\bm{f}^{C,r}_m=\bm{f}^*_m\bm{C}_m(\bm{\xi}_{f,k}^r)-\frac{1}{2}\bm{f}_m\bm{C}_m(\bm{\xi}_{f,k}^r)-\frac{1}{2}\bm{\chi}(\bm{\xi}_{f,k}^r)\bm{\hat{f}}_m^r(t)^T$ as the difference between the reference transformation of the physical numerical flux, the physical flux and the interpolated reference flux on the face.

\section{Energy Stable Flux Reconstruction}\label{sec: ESFR}
\subsection{ESFR - Classical Formulation}\label{sec: ESFR classical}

Following an ESFR framework, the reference flux is composed of a discontinuous and a corrected component,
\begin{equation}
    \bm{f}^r(u_m^h(\bm{\Theta}_m(\bm{\xi}^r),t))\coloneqq 
    \bm{f}^{D,r}(u_m^h(\bm{\Theta}_m(\bm{\xi}^r),t)) + \sum_{f=1}^{N_f}\sum_{k=1}^{N_{f_{p}}} \bm{g}^{f,k}(\bm{\xi}^r)[\hat{\bm{n}}^r\cdot ({\bm{f}}_m^{*,r}-{\bm{f}}_m^r)].\label{eq: ESFR corr flux}
\end{equation}

\noindent 
The vector correction functions $\bm{g}^{f,k}(\bm{\xi}^r)\in\mathbb{R}^{1\times d}$ associated with face $f$, facet cubature node $k$ in the reference element, are defined as the tensor product of the $p+1$ order one-dimensional correction functions ($\bm{\phi}$ stores a basis of order $p+1$), with the corresponding $p$-th order basis functions in the other reference directions.
\begin{equation}
\begin{split}
    \bm{g}^{f,k}(\bm{\xi}^r) = [\Big(\bm{\phi}(\xi)\otimes \bm{\chi}({\eta}) \otimes \bm{\chi}({\zeta})\Big)\Big(\hat{\bm{g}}^{f,k}_1\Big)^T , \Big(\bm{\chi}({\xi})\otimes \bm{\phi}(\eta)\otimes \bm{\chi}({\zeta})\Big)\Big(\hat{\bm{g}}^{f,k}_2\Big)^T,\Big(\bm{\chi}({\xi}) \otimes \bm{\chi}({\eta}) \otimes \bm{\phi}(\zeta)\Big) \Big(\hat{\bm{g}}^{f,k}_3\Big)^T ] \\
    = [{g}^{f,k}_1(\bm{\xi}^r), {g}^{f,k}_2(\bm{\xi}^r), {g}^{f,k}_3(\bm{\xi}^r)],
    \end{split}
\end{equation}

\noindent such that 
\begin{equation}
    \bm{g}^{f,k}(\bm{\xi}_{f_i, k_j}^r) \cdot \hat{\bm{n}}^r_{f_i,k_j} = \begin{cases}
    1, \;\;\text{ if } f_i = f,\text{ and } k_j=k\\
    0, \;\;\text{ otherwise.}
    \end{cases}\label{eq: ESFR conditions}
\end{equation}

Coupled with the symmetry condition $g^L(\xi^r)=-g^R(-\xi^r)$ to satisfy Eq.~(\ref{eq: ESFR conditions}), the one-dimensional ESFR fundamental assumption from~\cite{vincent_new_2011} is, 

\begin{equation}
    \int_{-1}^{1} \nabla^r {\chi}_i({\xi}^r)
      {g}^{f,k}({\xi}^r) d\xi  
    - c\frac{\partial^{p} {\chi}_i({\xi}^r) ^T}{\partial \xi^p}
   \frac{\partial^{p+1} {g}^{f,k}({\xi}^r)}{\partial \xi^{p+1}} = {0},\: \forall i=1,\dots,N_p,\label{eq: 1D ESFR fund assumpt}
\end{equation}

\noindent and similarly for the other reference directions.

Akin to~\cite{castonguay2012newTRI,williams_energy_2014}, consider introducing the differential operator,

\begin{equation}
\begin{split}
&\text{2D:} \indent \partial^{(s,v)}=\frac{\partial^{s+v}}{\partial \xi^s\partial \eta^v},\: \text{such that } s=\{ 0,p\},\: v=\{ 0,p\},\: s+v\geq p ,\\
&\text{3D:}\indent     \partial^{(s,v,w)}=\frac{\partial^{s+v+w}}{\partial \xi^s\partial \eta^v\partial \zeta^w},\: \text{such that } s=\{ 0,p\},\: v=\{ 0,p\},\: w=\{ 0,p\},\: s+v+w\geq p,\label{eq: partial deriv ESFR def}
\end{split}
\end{equation}

\noindent with its corresponding correction parameter

\begin{equation}
\begin{split}
&\text{2D:} \indent c_{(s,v)}=c_{1D}^{(\frac{s}{p}+\frac{v}{p})},\\
&\text{3D:}\indent     c_{(s,v,w)}=c_{1D}^{(\frac{s}{p}+\frac{v}{p}+\frac{w}{p})}.
    \end{split}
\end{equation}

\noindent {\color{black} Note that the total degree is $dim\times p$ for a tensor-product basis that is of order $p$ in each direction.}

\noindent For example,

\begin{equation}
    \partial^{(0,p,0)} = \frac{\partial^p}{\partial \eta^p},\: 
     c_{(0,p,0)} = c_{1D},\:
    \partial^{(p,0,p)}= \frac{\partial^{2p}}{\partial \xi^p \partial \zeta^p},\: c_{(p,0,p)} = c_{1D}^2,\:
    \partial^{(p,p,p)}= \frac{\partial^{3p}}{\partial \xi^p \partial \eta^p \partial \zeta^p},\: c_{(p,p,p)} = c_{1D}^3.
    \nonumber
\end{equation}

\noindent Since $\int_{\bm{\Omega}_r}\partial^{(s,v,w)}\bm{\chi}(\bm{\xi}^r)^T\partial^{(s,v,w)}\Big(\nabla^r\bm{\chi}(\bm{\xi}^r)\Big) d\bm{\Omega}_r$ composes of the complete broken Sobolev-norm for each $s$, $v$, $w$~\cite{sheshadri2016stability,Cicchino2020NewNorm}, the tensor product ESFR fundamental assumption, that recovers the VCJH schemes exactly for linear elements is defined as,

\begin{equation}
        \int_{\bm{\Omega}_r} \nabla^r {\chi}_i(\bm{\xi}^r) \cdot
     \bm{g}^{f,k}(\bm{\xi}^r) d\bm{\Omega}_r  
    - \sum_{s,v,w}c_{(s,v,w)}\partial^{(s,v,w)} {\chi}_i(\bm{\xi}^r) 
   \partial^{(s,v,w)} \Big( \nabla^r \cdot \bm{g}^{f,k}(\bm{\xi}^r)\Big) = {0},\: \forall i=1,\dots, N_p,\label{eq: ESFR fund assumpt}
\end{equation}

\noindent where $\sum_{s,v,w}$ sums over all possible $s$, $v$, $w$ combinations in Eq.~(\ref{eq: partial deriv ESFR def}).

To discretely represent the divergence of the correction functions, we introduce the correction field \newline ${h}^{f,k}(\bm{\xi}^r)\in P_{3p}(\bm{\Omega}_r)$ associated with the face $f$ cubature node $k$ as,

\begin{equation}
    {h}^{f,k}(\bm{\xi}^r) = \bm{\chi}(\bm{\xi}^r)\Big(\hat{\bm{h}}^{f,k}\Big)^T
    =\nabla^r \cdot \bm{g}^{f,k}(\bm{\xi}^r).
\end{equation}
To arrive at the ESFR strong form, we substitute the ESFR reference flux, Eq.~(\ref{eq: ESFR corr flux}), into the elementwise reference residual, Eq.~(\ref{eq: residual reference}), project it onto the polynomial basis, and evaluate at cubature nodes,

\begin{equation}
    \bm{\chi}(\bm{\xi}_v^r)\frac{d}{dt}\hat{\bm{u}}_m(t)^T 
    +\bm{J}_m^{-1}\nabla^r\bm{\chi}(\bm{\xi}_v^r) \cdot \hat{\bm{f}}_{m}^{D,r}(t)^T
    +\bm{J}_m^{-1}\sum_{f=1}^{N_f}\sum_{k=1}^{N_{fp}} \bm{\chi}(\bm{\xi}_v^r)\Big(\hat{\bm{h}}^{f,k}\Big)^T [\hat{\bm{n}}^r\cdot ({\bm{f}}_m^{*,r}-{\bm{f}}_m^r)]   =\bm{0}^T.
    \label{eq:ESFR strong}
\end{equation}

Since Eq.~(\ref{eq:ESFR strong}) does not mimic integration by parts in the physical domain, as previously demonstrated in Sections ~\ref{sec:DG non cons} and~\ref{sec: DG Split Form}, we introduce the split form in compact form,

\begin{equation}
    \begin{split}
        \bm{\chi}(\bm{\xi}_v^r)\frac{d}{dt}\hat{\bm{u}}_m(t)^T 
    +\frac{1}{2} \bm{J}_m^{-1}\nabla^r\bm{\chi}(\bm{\xi}_v^r)\cdot \hat{\bm{f}}_{m}^{D,r}(t)^T
    +\frac{1}{2}\bm{J}_m^{-1}\Tilde{\nabla}^r\bm{\chi}(\bm{\xi}_v^r) \cdot \hat{\bm{f}}_{m}^{D}(t)^T
    +\bm{J}_m^{-1}\sum_{f=1}^{N_f}\sum_{k=1}^{N_{fp}} \bm{\chi}(\bm{\xi}_v^r)\Big(\hat{\bm{h}}^{f,k}\Big)^T [{\bm{n}}^r \cdot {\bm{f}}^{C,r}_m]   =\bm{0}^T.
    \end{split}\label{eq: ESFR split unstable}
\end{equation}

Unfortunately, Eq.~(\ref{eq: ESFR split unstable}), which we will coin as the \enquote{Classical ESFR split form} is not energy stable since the nonlinearity introduced by both the metric cofactor matrix and determinant of the Jacobian prevents the volume terms from vanishing within the broken Sobolev-norm introduced in~\cite{jameson_proof_2010}. Fortunately, there is a modified form of Eq.~(\ref{eq: ESFR split unstable}) which is provably stable and recovers the Classical ESFR scheme for linear problems. We will term the proposed split form, which we now derive, as the \enquote{ESFR split form}.

To derive the proposed ESFR split form, we recast ESFR as a filtered DG scheme. To do so, as shown in~\cite{zwanenburg_equivalence_2016, allaneau_connections_2011,Cicchino2020NewNorm}, we integrate Eq.~(\ref{eq: ESFR split unstable}) with respect to the basis function as the test function in the physical domain. Using the definitions of the differential volume and surface elements, Eq.~(\ref{eq: diff element def}), we integrate the divergence of the correction functions by parts,

\begin{equation}
\begin{split}
    \int_{\bm{\Omega}_r} {\chi}_i(\bm{\xi}^r)  J_m^\Omega \bm{\chi}(\bm{\xi}^r)\frac{d}{dt}\hat{\bm{u}}_m(t)^T d\bm{\Omega}_r
  +\frac{1}{2} \int_{\bm{\Omega}_r} {\chi}_i(\bm{\xi}^r)  \nabla^r\bm{\chi}(\bm{\xi}^r)  \cdot \hat{\bm{f}}_m^{D,r}(t)^T d\bm{\Omega}_r
    +\frac{1}{2} \int_{\bm{\Omega}_r} {\chi}_i(\bm{\xi}^r)  \Tilde{\nabla}^r \bm{\chi}(\bm{\xi}^r) \cdot \hat{\bm{f}}_{m}^{D}(t)^T d\bm{\Omega}_r\\
    + \int_{\bm{\Gamma}_r}  {\chi}_i(\bm{\xi}^r)  \Big(\hat{\bm{n}}^r\cdot\bm{g}^{f,k}(\bm{\xi}^r)\Big)[\hat{\bm{n}}^r\cdot \bm{f}^{C,r}_m ]d\bm{\Gamma}_r
    -\int_{\bm{\Omega}_r} \nabla^r {\chi}_i(\bm{\xi}^r)  \cdot \bm{g}^{f,k}(\bm{\xi}^r)[\hat{\bm{n}}^r\cdot \bm{f}^{C,r}_m] d\bm{\Omega}_r =  {0},\: \forall i=1,\dots , N_p.
\end{split}\label{eq:ESFR int parts}
\end{equation}

From the ESFR correction functions' surface condition, Eq.~(\ref{eq: ESFR conditions}), the facet integral in Eq.~(\ref{eq:ESFR int parts}) is the exact same as the facet integral in the DG strong split form Eq.~(\ref{eq: var DG split}). Also, the reference discontinuous flux for the ESFR scheme is the same as the reference flux for a DG scheme (from definition). Thus, we will drop the $D$ superscript for the flux.

Next, as in the ESFR literature, we apply the differential operator $\partial^{(s,v,w)}$ 
on Eq.~(\ref{eq: ESFR split unstable}), then left multiply and integrate with respect to the $\partial^{(s,v,w)}$ derivative of the basis function as the test function in the physical domain~\cite{Cicchino2020NewNorm,zwanenburg_equivalence_2016,huynh_flux_2007,vincent_insights_2011,castonguay_phd,jameson2010proof}. Then 
a scalar $c_{(s,v,w)}$ is incorporated and the expression is summed over 
all $(s,v,w)$ combinations. The order of those steps is extremely important as it ensures a {\color{black}positive-definite} broken Sobolev-norm, 
which solves the issue presented in~\cite{svard2004coordinate,ranocha2017extended}. 
This results in,

\begin{equation}
\begin{split}
    \sum_{s,v,w} c_{(s,v,w)}\int_{\bm{\Omega}_r} J_m^\Omega \partial^{(s,v,w)}{\chi}_i(\bm{\xi}^r)  \partial^{(s,v,w)}\bm{\chi}(\bm{\xi}^r)\frac{d}{dt}\hat{\bm{u}}_m(t)^T d\bm{\Omega}_r
     + \sum_{s,v,w}c_{(s,v,w)}\int_{\bm{\Omega}_r} J_m^\Omega \partial^{(s,v,w)}{\chi}_i(\bm{\xi}^r)  \partial^{(s,v,w)}\Big(\frac{1}{J_m^\Omega}\nabla^r\cdot \bm{g}^{f,k}(\bm{\xi}^r)\Big)[\hat{\bm{n}}^r\cdot \bm{f}_m^{C,r}] d\bm{\Omega}_r 
    \\
     +\sum_{s,v,w}\frac{c_{(s,v,w)}}{2}\int_{\bm{\Omega}_r} J_m^\Omega 
     \partial^{(s,v,w)}{\chi}_i(\bm{\xi}^r)
     \partial^{(s,v,w)}
     \Big[\frac{1}{J_m^\Omega}\nabla^r\bm{\chi}(\bm{\xi}^r)\cdot \hat{\bm{f}}_m^r(t)^T +
     \frac{1}{J_m^\Omega}\Tilde{\nabla}^r \bm{\chi}(\bm{\xi}^r) \cdot \hat{\bm{f}}_{m}(t)^T
     \Big] d\bm{\Omega}_r=  {0},\: \forall i=1,\dots , N_p.\label{eq: ESFR 3p deriv}
\end{split}
\end{equation}

\noindent Adding Eqs.~(\ref{eq:ESFR int parts}) and~(\ref{eq: ESFR 3p deriv}) together results in,

\begin{equation}
\begin{split}
 \int_{\bm{\Omega}_r} \Big( {\chi}_i(\bm{\xi}^r)  J_m^\Omega \bm{\chi}(\bm{\xi}^r)
 +
  \sum_{s,v,w} c_{(s,v,w)} J_m^\Omega \partial^{(s,v,w)}{\chi}_i(\bm{\xi}^r)  \partial^{(s,v,w)}\bm{\chi}(\bm{\xi}^r)
 \Big)\frac{d}{dt}\hat{\bm{u}}_m(t)^T d\bm{\Omega}_r
        \\
   +\frac{1}{2} \int_{\bm{\Omega}_r} {\chi}_i(\bm{\xi}^r)  \nabla^r\bm{\chi}(\bm{\xi}^r)  \cdot \hat{\bm{f}}_m^{r}(t)^T d\bm{\Omega}_r
    +\frac{1}{2} \int_{\bm{\Omega}_r} {\chi}_i(\bm{\xi}^r)  \Tilde{\nabla}^r \bm{\chi}(\bm{\xi}^r) \cdot \hat{\bm{f}}_{m}(t)^T d\bm{\Omega}_r
    + \int_{\bm{\Gamma}_r}  {\chi}_i(\bm{\xi}^r)  [\hat{\bm{n}}^r\cdot \bm{f}^{C,r}_m ]d\bm{\Gamma}_r\\
     +\sum_{s,v,w}\frac{c_{(s,v,w)}}{2}\int_{\bm{\Omega}_r} J_m^\Omega
     \partial^{(s,v,w)}{\chi}_i(\bm{\xi}^r)
     \partial^{(s,v,w)}
     \Big[\frac{1}{J_m^\Omega}\nabla^r\bm{\chi}(\bm{\xi}^r)\cdot \hat{\bm{f}}_m^r(t)^T +
     \frac{1}{J_m^\Omega}\Tilde{\nabla}^r \bm{\chi}(\bm{\xi}^r) \cdot \hat{\bm{f}}_{m}(t)^T
     \Big] d\bm{\Omega}_r
    \\
    -\Bigg(\int_{\bm{\Omega}_r} \nabla^r {\chi}_i(\bm{\xi}^r)  \cdot \bm{g}^{f,k}(\bm{\xi}^r) d\bm{\Omega}_r
   - \sum_{s,v,w}c_{(s,v,w)}\int_{\bm{\Omega}_r}  J_m^\Omega \partial^{(s,v,w)}{\chi}_i(\bm{\xi}^r)  \partial^{(s,v,w)}\Big(\frac{1}{J_m^\Omega}\nabla^r\cdot \bm{g}^{f,k}(\bm{\xi}^r)\Big) d\bm{\Omega}_r\Bigg)[\hat{\bm{n}}^r\cdot \bm{f}_m^{C,r}]
    =  {0},\: \forall i=1,\dots , N_p.\label{eq: ESFR add equations together}
\end{split}
\end{equation}

\noindent Note that $[\hat{\bm{n}}^r\cdot \bm{f}_m^{C,r}]$ is a constant evaluated on the surface, so it can be factored out of the last volume integrals~\cite{vincent_new_2011,castonguay2012newTRI}.

{ 
The root of the instability of the classical ESFR in split form is demonstrated in the third line of Eq.~(\ref{eq: ESFR add equations together}). On linear grids, the determinant of the Jacobian and the metric cofactor matrix are both constants, and render the $\partial^{(s,v,w)}$ derivative of the divergence of the discontinuous flux to be skew-symmetric~\cite{sheshadri2016stability}. However, for curvilinear elements, the determinant of the Jacobian and the metric cofactor matrix are both nonlinear polynomials. Thus, the $\partial^{(s,v,w)}$ derivative of the volume terms does not vanish in Eq.~(\ref{eq: ESFR add equations together}). Ranocha~\etal in~\cite{ranocha2016summation} circumvented the issue by setting the ESFR contribution to zero and solving for the DG case ($c_{(s,v,w)}=0$). In the case of Abe~\etal~\cite{abe2018stable}, the authors showed stability for Huynh's $\text{g}_2$ lumped-Lobatto scheme. This was expected since Huynh's $\text{g}_2$ lumped-Lobatto scheme is equivalent to a collocated DG scheme on GLL nodes~\cite{de2014connections}. 

}

An additional issue introduced by ESFR on curvilinear grids is that the aforementioned ESFR stability condition (fundamental assumption) in Eq.~(\ref{eq: ESFR fund assumpt}) (or the 1D analogous Eq.~(\ref{eq: 1D ESFR fund assumpt})) only holds true on linear grids. That is, because in Eq.~(\ref{eq: ESFR add equations together}), if the determinant of the Jacobian was constant, then it would be factored off in the last integral and the $\partial^{(s,v,w)}$ derivative of the corresponding mode of the correction functions would then be factored out of the integral~\cite{Cicchino2020NewNorm,castonguay_phd,castonguay_new_2012,vincent_extended_2015,vincent_insights_2011,huynh_reconstruction_2009,witherden2016high}. On general curvilinear coordinates, this is not true, even for analytically exact metric terms and exact integration as per Theorem~\ref{thm: difference cons and non cons}, and the complete ESFR fundamental assumption for three-dimensional tensor product curvilinear elements should be,

\begin{equation}
    \int_{\bm{\Omega}_r} \nabla^r {\chi}_i(\bm{\xi}^r) \cdot
     \bm{g}^{f,k}(\bm{\xi}^r) d\bm{\Omega}_r  
    - \sum_{s,v,w} c_{(s,v,w)}\int_{\bm{\Omega}_r}J_m^\Omega \partial^{(s,v,w)}{\chi}_i(\bm{\xi}^r)  
  \partial^{(s,v,w)}\Big( \frac{1}{J_m^\Omega}\nabla^r \cdot \bm{g}^{f,k}(\bm{\xi}^r)\Big)d\bm{\Omega}_r = {0},\: \forall i=1,\dots,N_p.\label{eq: CURV ESFR fund assumpt}
\end{equation}

\noindent If the grid is constant/linear then Eq.~(\ref{eq: CURV ESFR fund assumpt}) simplifies to Eq.~(\ref{eq: ESFR fund assumpt}) with a constant scaling of the volume of the reference element on $c_{(s,v,w)}$. To extend Eq.~(\ref{eq: CURV ESFR fund assumpt}) for triangular and prismatic curvilinear grids, one should change the $\partial^{(s,v,w)}$ derivative with the operator $\bm{D}^{p,v,w}$ presented in~\cite{castonguay2012newTRI,zwanenburg_equivalence_2016,castonguay_phd}, and the analysis/result is the same.

Therefore, using the metric dependent ESFR stability criteria, Eq.~(\ref{eq: CURV ESFR fund assumpt}) in Eq.~(\ref{eq: ESFR add equations together}), and evaluating bilinear forms at cubature nodes results in,

\begin{equation}
\begin{split}
     \Big(\bm{M}_m+\bm{K}_m\Big)\frac{d}{d t} \hat{\bm{u}}_m(t)^T
    + \frac{1}{2}\bm{\chi}(\bm{\xi}_v^r)^T\bm{W}\nabla^r\bm{\chi}(\bm{\xi}_v^r)\cdot \hat{\bm{f}}_m^r(t)^T 
     + \frac{1}{2}\bm{\chi}(\bm{\xi}_v^r)^T\bm{W}\Tilde{\nabla}^r\bm{\chi}(\bm{\xi}_v^r) \cdot \hat{\bm{f}}_m(t)^T
    +\sum_{f=1}^{N_f}\sum_{k=1}^{N_{fp}} \bm{\chi}(\bm{\xi}_{f,k}^r)^T W_{f,k}[\hat{\bm{n}}^r\cdot \bm{f}_m^{C,r}]\\
    +\sum_{s,v,w}\frac{c_{(s,v,w)}}{2} \partial^{(s,v,w)}\bm{\chi}(\bm{\xi}_v^r)^T
    \bm{J}_m \bm{W} \partial^{(s,v,w)}\bm{\chi}(\bm{\xi}_v^r)\bm{\Pi}\Big[
    \bm{J}_m^{-1}\nabla^r\cdot \bm{\chi}(\bm{\xi}_v^r)\hat{\bm{f}}_m^r(t)^T
    +
     \bm{J}_m^{-1}\Tilde{\nabla}^r\bm{\chi}(\bm{\xi}_v^r) \cdot \hat{\bm{f}}_m(t)^T
    \Big]
    =\bm{0}^T.\label{eq: ESFR Classical equiv}
    \end{split}
\end{equation}

\noindent Eq.~(\ref{eq: ESFR Classical equiv}) is the filtered DG equivalent of the Classical ESFR split form presented in Eq.~(\ref{eq: ESFR split unstable}), with

\begin{equation}
    \begin{split}
    (\bm{K}_m)_{ij} \approx \sum_{s,v,w } c_{(s,v,w)} \int_{ {\Omega}_r} J_m^\Omega \partial^{(s,v,w)} \chi_i(\bm{\xi}^r) 
    \partial^{(s,v,w)}\chi_j(\bm{\xi}^r) d {\Omega_r}\\
    \to 
    \bm{K}_m = 
     \sum_{s,v,w} c_{(s,v,w)}
    \partial^{(s,v,w)}\bm{\chi}(\bm{\xi}_v^r) ^T \bm{W}\bm{J}_m
   \partial^{(s,v,w)}\bm{\chi}(\bm{\xi}_v^r)
   =\sum_{s,v,w } c_{(s,v,w)}\Big(\bm{D}_\xi^s \bm{D}_\eta^v\bm{D}_\zeta^w \Big)^T\bm{M}_m\Big(\bm{D}_\xi^s \bm{D}_\eta^v\bm{D}_\zeta^w \Big),
    \end{split}\label{eq:Km}
\end{equation}

\noindent where $\bm{D}_\xi^s=\Big(\bm{M}^{-1}\bm{S}_\xi\Big)^s$ is the strong form differential operator raised to the power $s$, and similarly for the other reference directions.

\begin{remark}
{\color{black}Note the inclusion of $\bm{J}_m$ within $\bm{K}_m$ in Eq.~(\ref{eq:Km}). It allows the broken Sobolev-norm $\bm{M}_m+\bm{K}_m$ to be symmetric positive definite (for values of $c_{1D}>c_{-}$). This naturally arises from the order of applying the differential operator, then integrating in physical space in Eq.~(\ref{eq: ESFR 3p deriv}), and re-defining the resultant curvilinear ESFR fundamental assumption Eq.~(\ref{eq: CURV ESFR fund assumpt}). This varies from the literature where the Jacobian was either a constant~\cite{allaneau_connections_2011,castonguay_phd,vincent_new_2011,jameson_non-linear_2012} or for curvilinear ESFR~\cite{mengaldo2016connections,zwanenburg_equivalence_2016} where the determinant of the Jacobian was left multiplied to Eq.~(\ref{eq: ESFR split unstable}). The $\partial^{(s,v,w)}$ derivative was then applied to the entire discretization (to have the $\partial^{(s,v,w)}$ derivative applied directly on the reference divergence operator), which would arise in the $\partial^{(s,v,w)}$ derivative of the determinant of the metric Jacobian $\bm{J}_m$ in the norm. Explicitly, $\partial^{(s,v,w)}\Big(\bm{J}_m \frac{d}{dt}\bm{u}_m^T\Big) \neq \bm{J}_m\partial^{(s,v,w)}\Big(\frac{d}{dt}\bm{u}_m^T\Big)$, and hence $\partial^{(s,v,w)}\bm{\chi}(\bm{\xi}_v^r)^T\bm{W}\partial^{(s,v,w)}\Big(\bm{J}_m\Big)$ is not a norm.\label{remark: Jac dependent sob norm}}
\end{remark}

\begin{remark}
{\color{black}The stated approach is unlike what is adopted in~\cite{allaneau_connections_2011,zwanenburg_equivalence_2016,ranocha2016summation} where $\bm{K}_m$ was constructed using the Legendre differential operator then transformed to the basis of the scheme.} {\color{black}Here $c_{(s,v,w)}$ must take the value from a normalized Legendre reference basis.}
\end{remark}
However, Eq.~(\ref{eq: ESFR Classical equiv}) is not provably stable since the final term does not vanish in the broken Sobolev-norm.

\begin{lemma}
Eq.~(\ref{eq: ESFR Classical equiv}) is equivalent to a DG scheme with the ESFR filter applied solely to the facet integral.
\label{lem:ESFR facet filet}
\end{lemma}

\begin{proof}

Rearranging Eq.~(\ref{eq: ESFR Classical equiv}) by substituting $\bm{\Pi}_m=\bm{M}_m^{-1}\bm{\chi}(\bm{\xi}_v^r)^T\bm{W}\bm{J}_m$, thus $\bm{\chi}(\bm{\xi}_v^r)^T\bm{W}=\bm{\chi}(\bm{\xi}_v^r)^T\bm{W}\bm{J}_m\bm{J}_m^{-1} = \bm{M}_m\bm{\Pi}_m\bm{J}_m^{-1}$, and using Chan~\cite[Theorem 4]{chan2019discretely} results in,

\begin{equation}
    \begin{split}
     \Big(\bm{M}_m+\bm{K}_m\Big)\frac{d}{d t} \hat{\bm{u}}_m(t)^T
    + \frac{1}{2}\bm{M}_m\bm{\Pi}_m
    \Big[\bm{J}_m^{-1}\nabla^r\bm{\chi}(\bm{\xi}_v^r)\cdot \hat{\bm{f}}_m^r(t)^T 
    +\bm{J}_m^{-1}\Tilde{\nabla}^r\bm{\chi}(\bm{\xi}_v^r) \cdot \hat{\bm{f}}_m(t)^T\Big]
    +\sum_{f=1}^{N_f}\sum_{k=1}^{N_{fp}} \bm{\chi}(\bm{\xi}_{f,k}^r)^T W_{f,k}[\hat{\bm{n}}^r\cdot \bm{f}_m^{C,r}]\\
    +\frac{1}{2} \bm{K}_m\bm{\Pi}_m\Big[ 
    \bm{J}_m^{-1}\nabla^r\bm{\chi}(\bm{\xi}_v^r)\cdot \hat{\bm{f}}_m^r(t)^T
    +
     \bm{J}_m^{-1}\Tilde{\nabla}^r\bm{\chi}(\bm{\xi}_v^r) \cdot \hat{\bm{f}}_m(t)^T
    \Big]
    =\bm{0}^T,\label{eq: ESFR Classical equiv 2}
    \end{split}
\end{equation}

\noindent which simplifies to

\begin{equation}
    \begin{split}
  \Big(\bm{M}_m+\bm{K}_m\Big)\frac{d}{d t} \hat{\bm{u}}_m(t)^T
    + \frac{1}{2}\Big(\bm{M}_m +\bm{K}_m \Big)\bm{\Pi}_m
    \Big[\bm{J}_m^{-1}\nabla^r\bm{\chi}(\bm{\xi}_v^r)\cdot \hat{\bm{f}}_m^r(t)^T 
    +\bm{J}_m^{-1}\Tilde{\nabla}^r\bm{\chi}(\bm{\xi}_v^r) \cdot \hat{\bm{f}}_m(t)^T\Big]\\
    +\sum_{f=1}^{N_f}\sum_{k=1}^{N_{fp}} \bm{\chi}(\bm{\xi}_{f,k}^r)^T W_{f,k}[\hat{\bm{n}}^r\cdot \bm{f}_m^{C,r}]
    =\bm{0}^T.\label{eq: ESFR Classical equiv 3}
    \end{split}
\end{equation}

\noindent Recalling the definition of $\bm{\Pi}_m=\bm{M}_m^{-1}\bm{\chi}(\bm{\xi}_v^r)^T\bm{W}\bm{J}_m$ and solving for $\frac{d}{dt}\hat{\bm{u}}_m(t)^T$ in Eq.~(\ref{eq: ESFR Classical equiv 3}) results in,

\begin{equation}
    \begin{split}
  \frac{d}{d t} \hat{\bm{u}}_m(t)^T
    + \frac{1}{2}\bm{M}_m^{-1}\bm{\chi}(\bm{\xi}_v^r)^T\bm{W}
    \Big[\nabla^r\bm{\chi}(\bm{\xi}_v^r)\cdot \hat{\bm{f}}_m^r(t)^T 
    +\Tilde{\nabla}^r\bm{\chi}(\bm{\xi}_v^r) \cdot \hat{\bm{f}}_m(t)^T\Big]
    +\Big(\bm{M}_m+\bm{K}_m\Big)^{-1}\sum_{f=1}^{N_f}\sum_{k=1}^{N_{fp}} \bm{\chi}(\bm{\xi}_{f,k}^r)^T W_{f,k}[\hat{\bm{n}}^r\cdot \bm{f}_m^{C,r}]
    =\bm{0}^T,\label{eq: ESFR Classical equiv 4}
    \end{split}
\end{equation}

\noindent which concludes the proof since the ESFR filter is only applied to the facet integral in Eq.~(\ref{eq: ESFR Classical equiv 4}).

\end{proof}

The proof in Lemma~\ref{lem:ESFR facet filet} shows that Eq.~(\ref{eq: ESFR Classical equiv}) recovers the divergence of the correction functions applied solely to the face in Eq.~(\ref{eq: ESFR Classical equiv 4}), as seen in the literature~\cite{allaneau_connections_2011,zwanenburg_equivalence_2016}. That is, from Allaneau and Jameson~\cite{allaneau_connections_2011}, and Zwanenburg and Nadarajah~\cite[Eq. 2.19]{zwanenburg_equivalence_2016},

\begin{equation}
    \bm{\Pi}\bm{J}_m^{-1}\bm{\chi}(\bm{\xi}^r_v)\Big(\hat{\bm{h}}^{f,k} \Big)^T = \Big(\bm{M}_m+\bm{K}_m\Big)^{-1} \bm{\chi}(\bm{\xi}_{f,k}^r)^T W_{f,k}.
\end{equation}

\noindent 
{\color{black}Explicitly, Eqs.~(\ref{eq: ESFR split unstable}),~(\ref{eq: ESFR Classical equiv}), and~(\ref{eq: ESFR Classical equiv 4}) are all equivalent expressions of ESFR.}

\subsection{ESFR - Proposed Nonlinearly Stable Flux Reconstruction}\label{sec: ESFR Prov Stable}

As shown by Cicchino \textit{et al.}~\cite{CicchinoNonlinearlyStableFluxReconstruction2021}, provable nonlinear stability can be established for FR schemes by incorporating the ESFR filter/divergence of the correction functions on the volume integrals. This results in our proposed ESFR split form,

\begin{equation}
    \begin{split}
     \frac{d}{d t} \hat{\bm{u}}_m(t)^T
    + \frac{1}{2}\Big(\bm{M}_m+\bm{K}_m\Big)^{-1}\bm{\chi}(\bm{\xi}_v^r)^T\bm{W}\Big[\nabla^r\bm{\chi}(\bm{\xi}_v^r)\cdot \hat{\bm{f}}_m^r(t)^T 
   +\Tilde{\nabla}^r\bm{\chi}(\bm{\xi}_v^r) \cdot  \hat{\bm{f}}_m(t)^T \Big] 
   +\Big(\bm{M}_m+\bm{K}_m\Big)^{-1}\sum_{f=1}^{N_f}\sum_{k=1}^{N_{fp}} 
   \bm{\chi}(\bm{\xi}_{f,k}^r)^T W_{f,k}[\hat{\bm{n}}^r\cdot \bm{f}_m^{C,r}]=\bm{0}^T.\label{eq: ESFR equiv dudt}
    \end{split}
\end{equation}

\noindent Or in equivalent form which simplifies the stability and conservation analysis,

\begin{equation}
\begin{split}
     \Big(\bm{M}_m+\bm{K}_m\Big)\frac{d}{d t} \hat{\bm{u}}_m(t)^T
    + \frac{1}{2}\bm{\chi}(\bm{\xi}_v^r)^T\bm{W}\Big[\nabla^r\bm{\chi}(\bm{\xi}_v^r)\cdot \hat{\bm{f}}_m^r(t)^T 
   +\Tilde{\nabla}^r\bm{\chi}(\bm{\xi}_v^r) \cdot  \hat{\bm{f}}_m(t)^T \Big]
    +\sum_{f=1}^{N_f}\sum_{k=1}^{N_{fp}} \bm{\chi}(\bm{\xi}_{f,k}^r)^T W_{f,k}[\hat{\bm{n}}^r\cdot \bm{f}_m^{C,r}]=\bm{0}^T.\label{eq: ESFR equiv}
    \end{split}
\end{equation}

\noindent 
Eq.~(\ref{eq: ESFR equiv dudt}) is on design order as proved in Sec.~\ref{sec:ESFR Recovering Existing Schemes}.

\begin{remark}
We present the equivalent form of Eq.~(\ref{eq: ESFR equiv dudt}) in SBP notation in Sec.~\ref{sec:SBP deriv} based on~\cite{chan2019skew}.
\end{remark}




We note that the computational implementation of the proposed ESFR schemes differ significantly from existing FR implementations in the literature~\cite{witherden2014pyfr}. On general curved meshes and for general quadrature rules, the FR norm matrix over each element is dense. 
Thus, when implementing the proposed energy stable FR scheme, the FR norm matrix must be constructed and inverted over each individual element. On affine elements, the inverse of each elemental norm matrix can be computed through a constant scaling of a single reference norm matrix, \textcolor{black}{as is typically done in ESFR by solving for the correction functions}~\cite{castonguay2012newTRI,williams_energy_2014}. 
However, for a static curved mesh, these matrix inverses can be precomputed and stored. This increases both storage costs and the number of memory transfers necessary for the proposed schemes.

In comparison, the most common FR schemes~\cite{castonguay2012newTRI,vincent_new_2011,huynh_flux_2007,wang2009unifying} avoid introducing a mass/norm matrix altogether by formulating the main computational steps of the scheme as operations on the reference element. Collocated DG schemes (and the equivalent FR schemes) on curved elements yield a trivially invertible diagonal mass (norm) matrix, with values of the determinant of the metric Jacobian at collocation points appearing as weights for each diagonal entry. For dense mass (norm) matrices appearing in high-order DG on curved meshes, it is possible to approximate the inverse in an efficient, energy stable, and high order accurate fashion using a weight-adjusted approximation to the mass matrix~\cite{chan2019discretely}. However, because the norm matrices constructed in this work are constructed as the sum of two matrices, 
it is not currently possible to directly apply such an approach.

\subsection{ESFR - Accuracy of Metric Dependent ESFR Schemes}
\label{sec:ESFR Recovering Existing Schemes}

Following the work of~\cite{allaneau_connections_2011,zwanenburg_equivalence_2016,Cicchino2020NewNorm}, we consider a normalized, $p$-th order Legendre reference basis \newline $\bm{\chi}_{ref}(\bm{\xi}^r)=\bm{\chi}_{ref}(\xi)\otimes\bm{\chi}_{ref}(\eta)\otimes\bm{\chi}_{ref}(\zeta)$ on $\bm{\xi}^r\in[-1,1]^3$. The motivation behind using an orthonormal reference basis rather than an orthogonal reference basis is that it allows $\bm{K}_m$ to be constructed directly with the differential operator and mass matrix of the scheme~\cite[Sec. 3.1]{Cicchino2020NewNorm}. Thus, we introduce the transformation operator $\bm{T}=\bm{\Pi}_{ref}\bm{\chi}(\bm{\xi}^r_v)$, where $\bm{\Pi}_{ref}=\bm{M}_{ref}^{-1}\bm{\chi}_{ref}(\bm{\xi}_v^r)^T\bm{W}$, such that $\bm{K}_m=\bm{T}^T\bm{K}_{m,ref}\bm{T}$.

Next, we explicitly formulate $\bm{K}_{m,ref}$ to derive the metric Jacobian dependent ESFR filter. To express $\bm{K}_{m,ref}$ we introduce the modal differential operators for a normalized Legendre reference basis $\hat{\bm{D}}_\xi^s=(\bm{M}_{ref}^{-1}\bm{S}_{\xi,ref})^s$, similarly for $\hat{\bm{D}}_\eta^v$ and $\hat{\bm{D}}_\zeta^w$, 
to result in,

\begin{equation}
\begin{split}
\bm{K}_{m,ref} = \sum_{s,v,w} c_{(s,v,w)} 
    \partial^{(s,v,w)}\bm{\chi}_{ref}(\bm{\xi}_v^r)^T \bm{W}\bm{J}_m
    \partial^{(s,v,w)}\bm{\chi}_{ref}(\bm{\xi}_v^r)
    =\sum_{s,v,w} c_{(s,v,w)} \Big(\hat{\bm{D}}_{\xi}^s \hat{\bm{D}}_{\eta}^v \hat{\bm{D}}_{\zeta}^w \Big)^T \bm{M}_{m,ref} \Big(\hat{\bm{D}}_{\xi}^s \hat{\bm{D}}_{\eta}^v \hat{\bm{D}}_{\zeta}^w \Big)\\
  \implies   (\bm{K}_{m,ref})_{ij} \approx \sum_{s,v,w} c_{(s,v,w)} \int_{\bm{\Omega}_r} J_m^\Omega
    \partial^{(s,v,w)}\chi_{ref,i}(\bm{\xi}^r)\partial^{(s,v,w)}\chi_{ref,j}(\bm{\xi}^r) d\bm{\Omega}_r
\end{split}\label{eq: Km ref}
\end{equation}

Typically when deriving the correction functions~\cite{vincent_new_2011,castonguay2012newTRI,witherden2016high} or ESFR filter~\cite{allaneau_connections_2011,zwanenburg_equivalence_2016,Cicchino2020NewNorm}, we would utilize the orthogonality of the reference basis functions. However, for curvilinear coordinates, the reference basis functions are not orthogonal on $d \bm{\Omega}_m={J}_m^\Omega d\bm{\Omega}_r$. That is,

\begin{equation}
\begin{split}
    \int_{\bm{\Omega}_r} \chi_{ref,i}(\bm{\xi}^r)\chi_{ref,j}(\bm{\xi}^r)d\bm{\Omega}_r =\delta_{ij},
    \end{split}
\end{equation}
    \noindent but, 
\begin{equation}
\begin{split}
    \int_{\bm{\Omega}_r} J_m^\Omega \chi_{ref,i}(\bm{\xi}^r)\chi_{ref,j}(\bm{\xi}^r) d\bm{\Omega}_r\neq \alpha \delta_{ij},\: \text{where } \alpha=\text{const},\:\text{unless } J_m^\Omega=\text{const},
\end{split}\label{eq: orthog basis def}
\end{equation}

\noindent where the last inequality holds even under exact integration and the analytically exact $J_m^\Omega$. An equality would be present in Eq.~(\ref{eq: orthog basis def}) if and only if $\bm{\chi}_{ref}(\bm{\Theta}_m^{-1}(\bm{x}^c))$ is also an orthogonal polynomial basis; but thus far there is no claim that $\bm{\chi}_{ref}(\bm{\Theta}_m^{-1}(\bm{x}^c))$ or $\bm{\chi}(\bm{\Theta}_m^{-1}(\bm{x}^c))$ are polynomial in the analysis. Using Eq.~(\ref{eq: orthog basis def}) in Eq.~(\ref{eq: Km ref}) directly shows that for a tensor-product basis, $\bm{K}_{m,ref}$ is not diagonal for curvilinear coordinates, and is diagonal only under the constant metric Jacobian case. 

To prove the order of accuracy for curvilinear ESFR schemes, we demonstrate which modes the ESFR filter, $\bm{F}_{m,ref}$, operates on; such that $\frac{d }{dt}\hat{\bm{u}}_{ref}(t)^T\bigg|_{\text{ESFR}} = \bm{F}_{m,ref}\frac{d }{dt}\hat{\bm{u}}_{ref}(t)^T\bigg|_{\text{DG}}$.

\begin{thrm}
For general curvilinear coordinates, the ESFR filter operator is applied to all modes of the discretization, not just the highest order mode; even for triangular/prismatic elements~\cite{castonguay2012newTRI,witherden2014analysis}, the $3p$-th broken Sobolev-norm considered in~\cite{Cicchino2020NewNorm}, and all other cases where the corresponding $\bm{K}_{m,ref}$ would be diagonal with a single entry on the highest mode.\label{thm: filter all modes}
\end{thrm}

\begin{proof}
We substitute Eq.~(\ref{eq: orthog basis def}) when constructing the metric Jacobian dependent mass matrix for a normalized Legendre reference basis,

\begin{equation}
    (\bm{M}_{m,ref})_{ij} \approx \int_{\bm{\Omega}_r} J_m^\Omega \chi_{ref,i}(\bm{\xi}^r)\chi_{ref,j}(\bm{\xi}^r) d\bm{\Omega}_r\neq \alpha \delta_{ij},\: \text{where } \alpha=\text{const},
\end{equation}

\noindent that shows the reference mass matrix is dense, even with exact integration and analytically exact metric terms. Thus, we let $\bm{M}_{m,ref}=\begin{bmatrix}
    a&b\\b&c
    \end{bmatrix}$ and $\bm{K}_{m,ref}=\begin{bmatrix}
0&0\\0&d
\end{bmatrix}$ to consider the special case where the correction functions are only applied on the highest mode, \textit{i.e.} prismatic/triangular curvilinear elements and tensor-product curvilinear elements using the $3p$-th broken Sobolev-norm~\cite{Cicchino2020NewNorm}. 
This implies $ \Big( \bm{M}_{m,ref}+\bm{K}_{m,ref}\Big)^{-1}=\frac{1}{a(c+d)-b^2}\begin{bmatrix}
c+d&-b\\-b&a
\end{bmatrix}$. Thus \newline $\bm{F}_{m,ref}=\Big( \bm{M}_{m,ref}+\bm{K}_{m,ref}\Big)^{-1}\bm{M}_{m,ref} = \begin{bmatrix}
1&\frac{bd}{a(c+d)-b^2}\\0&\frac{ac-b^2}{a(c+d)-b^2}
\end{bmatrix}$. Therefore, considering the complete case for \newline $\bm{F}_{m,ref}=\Big( \bm{M}_{m,ref}+\bm{K}_{m,ref}\Big)^{-1}\bm{M}_{m,ref}$
implies the filter has influence on all modes, rather than just the highest mode, which varies from the literature for linear grids~\cite{allaneau_connections_2011,zwanenburg_equivalence_2016,Cicchino2020NewNorm}.
\end{proof}

Typically in the ESFR literature~\cite{huynh_flux_2007,jameson_non-linear_2012,wang2009unifying,vincent_new_2011,castonguay2012newTRI}, the scheme is shown to lose at most one order of accuracy because the divergence of the correction functions correspond to the highest mode of the scheme. Unfortunately, this is only true for constant metric Jacobians, or {\color{black}non-positive-definite} 
norms as discussed in Remark~\ref{remark: Jac dependent sob norm}. Theorem~\ref{thm: filter all modes} directly proves that ESFR schemes can lose all orders for general curvilinear coordinates; even without considering our proposed ESFR split form and instead considering the classical VCJH schemes with or without exact integration, and/or with or without analytically exact or discrete metric terms. This result is dependent on the metric dependence within the ESFR fundamental assumption, Eq.~(\ref{eq: CURV ESFR fund assumpt}), in curvilinear coordinates.
In Section~\ref{sec:ESFR deriv test}, we numerically show for three-dimensions that the scheme loses all orders at the approximate location~\cite[Figure 3.6]{castonguay_phd} that the one-dimensional ESFR/VCJH schemes lose one order of accuracy; at a value of $c\gg c_+$.

\section{Discrete GCL}\label{sec: Discrete GCL}

In this section we briefly review how to compute $\bm{C}_m$ to ensure both the correct orders of accuracy, free-stream preservation and surface metric terms being consistent between interior and exterior cells. The main idea from Kopriva~\cite{kopriva2006metric} was to satisfy the GCL (Eq.~(\ref{eq: GCL})) \textit{a priori} by equivalently expressing the reference vector basis multiplied by the determinant of the Jacobian in curl form. With the interpolation being within the curl, discrete GCL is satisfied since it is the divergence of the curl. That is, by expressing the reference transformation (metric cofactor matrix) as,

\begin{equation}
\begin{split}
  \text{(Cross Product Form) } J^\Omega \bm{a}^i = \bm{a}_j \times \bm{a}_k\text{, }i=1,2,3\text{ }(i,j,k)\text{ cyclic,}\\
  \Leftrightarrow  \text{(Conservative Curl Form) }  J^\Omega (\bm{a}^i)_n = -\hat{\bm{e}}_i \cdot \nabla^r\times\Big(\bm{x}_l^c\nabla^r \bm{x}_m^c \Big)\text{, }i=1,2,3\text{, }n=1,2,3\text{ }(n,m,l)\text{ cyclic,}\\
  \Leftrightarrow   \text{(Invariant Curl Form) }   J^\Omega (\bm{a}^i)_n = -\frac{1}{2}\hat{\bm{e}}_i \cdot \nabla^r\times\Big(\bm{x}_l^c\nabla^r \bm{x}_m^c  - \bm{x}_m^c\nabla^r \bm{x}_l^c \Big)\text{, }i=1,2,3\text{, }n=1,2,3\text{ }(n,m,l)\text{ cyclic,}
    \end{split}\nonumber
\end{equation}

\noindent where $\hat{\bm{e}} = [\hat{x}, \hat{y}, \hat{z}]$ is the physical unit vector~\footnote{this is not to be confused with the previous definition of $\hat{\bm{x}}^c$ which are the mapping support points (grid points)}.
Then for the conservative curl form, the GCL can be written as,

\begin{equation}
    \sum_{i=1}^3 \frac{\partial (J^\Omega (\bm{a}^i)_n)}{\partial \xi^i} = -\nabla^r\cdot
    \Big( \nabla^r\times\Big(\bm{x}_l^c\nabla^r \bm{x}_m^c \Big) \Big)=0\text{, }n=1,2,3\text{ }(n,m,l)\text{ cyclic,}\nonumber
\end{equation}

\noindent and similarly for the invariant curl form. Thus Kopriva~\cite{kopriva2006metric} proved that to discretely satisfy GCL \textit{a priori}, one must interpolate to the flux nodes (volume or facet cubature nodes) before applying the curl. That is, the discrete conservative curl form reads as,

\begin{equation}
    J^\Omega (\bm{a}^i)_n = (\bm{C})_{ni} = -\hat{\bm{e}}_i \cdot \nabla^r\times\bm{\Theta}(\bm{\xi}^r)\Big(\bm{x}_l^c\nabla^r \bm{x}_m^c \Big)\text{, }i=1,2,3\text{, }n=1,2,3\text{ }(n,m,l)\text{ cyclic,}\label{eq: conservative curl Kopriva}
\end{equation}

\noindent and similarly for the invariant curl form. For general three-dimensional curvilinear elements, Kopriva~\cite{kopriva2006metric} also proved that the cross product form does not discretely satisfy GCL, thus the conservative or invariant curl forms should always be used.

One of the primary issues raised by Abe \etal~\cite{abe2015freestream} was that Eq.~(\ref{eq: conservative curl Kopriva}) does not ensure that the normals match at each facet cubature node. It is to be noted that Abe \etal~\cite{abe2015freestream} considered only the invariant curl form, but the methodology is also consistent for the conservative curl form. To circumvent the issue, one main result from~\cite{abe2015freestream} was to have two separate interpolation operators, one for the \enquote{grid points} (mapping support points) and another for the cubature (flux) nodes. This distinction was made because in high-order grid generation, it is typical to have the exact corners of the elements, making them continuous finite elements at the grid points~\cite{johnen2013geometrical,turner2018curvilinear}. Thus, Abe \etal~\cite[Eqs. (31)-(34), (41) and (42)]{abe2015freestream} evaluated $\Big(\bm{x}_l^c\nabla^r \bm{x}_m^c \Big)$ in Eq.~(\ref{eq: conservative curl Kopriva}) 
at the grid nodes, and computed the mapping shape functions 
at the flux nodes prior to the application of the curl operator~\cite[Eq. (43)]{abe2015freestream}. By doing so, consistency is ensured at each face since the grid nodes are continuous; and GCL is satisfied at each quadrature point because the final interpolation is performed within the curl operator. Therefore, we have the discrete conservative curl form,

\begin{equation}
    \begin{split}
    (\bm{C})_{ni}= -\hat{\bm{e}}_i \cdot \nabla^r\times\bm{\Theta}(\bm{\xi}_{\text{flux nodes}}^r)
    \Big[
    \bm{\Theta}(\bm{\xi}_{\text{grid nodes}}^r)\hat{\bm{x}}_l^{c^T}
    \nabla^r \bm{\Theta}(\bm{\xi}_{\text{grid nodes}}^r)\hat{\bm{x}}_m^{c^T} \Big]
    \text{, }
    i=1,2,3\text{, }n=1,2,3\text{ }(n,m,l)\text{ cyclic,}
    \end{split}\label{eq: discrete cons curl abe}
\end{equation}

\noindent and similarly for the discrete invariant curl form,

\begin{equation}
    \begin{split}
          (\bm{C})_{ni}= -\frac{1}{2}\hat{\bm{e}}_i \cdot \nabla^r\times\bm{\Theta}(\bm{\xi}_{\text{flux nodes}}^r)\Big[
          \bm{\Theta}(\bm{\xi}_{\text{grid nodes}}^r)\hat{\bm{x}}_l^{c^T}
    \nabla^r \bm{\Theta}(\bm{\xi}_{\text{grid nodes}}^r)\hat{\bm{x}}_m^{c^T}
     -
     \bm{\Theta}(\bm{\xi}_{\text{grid nodes}}^r)\hat{\bm{x}}_m^{c^T}
    \nabla^r \bm{\Theta}(\bm{\xi}_{\text{grid nodes}}^r)\hat{\bm{x}}_l^{c^T}\Big]
          \text{, }\\i=1,2,3\text{, }n=1,2,3\text{ }(n,m,l)\text{ cyclic,}
    \end{split}\label{eq: discrete inv curl abe}
\end{equation}

\noindent where we assumed the mapping shape functions are collocated on the mapping support points $\hat{\bm{x}}^c$. In all numerical results we used Eq.~(\ref{eq: discrete cons curl abe}) with GLL as the grid nodes.

\section{Free-Stream Preservation, Conservation and Stability}

{\color{black}{In this section, we prove free-stream preservation, conservation, and stability for our proposed provably stable FR split form, Eq.~(\ref{eq: ESFR equiv dudt}). For free-stream preservation, we prove that it is essential to distinguish between grid nodes and flux nodes in Eqs.~(\ref{eq: discrete cons curl abe}) and~(\ref{eq: discrete inv curl abe}). Then, by satisfying GCL, we demonstrate conservation. Lastly, to illustrate the stability of the scheme, we show that it is essential to incorporate the divergence of the correction functions on the volume terms.}}

\subsection{Free-Stream Preservation}\label{sec:Free-stream}
We first demonstrate that the surface splitting from Eq.~(\ref{eq: ESFR equiv dudt}) satisfies free-stream preservation if the metric terms are computed via Eq.~(\ref{eq: discrete cons curl abe}) or~(\ref{eq: discrete inv curl abe}). We start by substituting $\bm{f}_m=\bm{\alpha}=$ constant and $\frac{d}{d t} \hat{\bm{u}}_m(t)^T=\bm{0}^T$ into Eq.~(\ref{eq: ESFR equiv dudt}),

\begin{equation}
    \begin{split}
    \frac{1}{2}\Big(\bm{M}_m+\bm{K}_m\Big)^{-1}\bm{\chi}(\bm{\xi}_v^r)^T\bm{W}\Big[\nabla^r\cdot \bm{\alpha}\bm{C}_m(\bm{\xi}_v^r) 
   +\Tilde{\nabla}^r \cdot \bm{\alpha} \Big] \\
   +\Big(\bm{M}_m+\bm{K}_m\Big)^{-1}\sum_{f=1}^{N_f}\sum_{k=1}^{N_{fp}} 
   \bm{\chi}(\bm{\xi}_{f,k}^r)^T W_{f,k}[\hat{\bm{n}}^r\bm{C}_m(\bm{\xi}_{f,k}^r)^T\cdot (\bm{\alpha}-\frac{1}{2}\bm{\alpha}) - \frac{1}{2}\hat{\bm{n}}^r\cdot \bm{\chi}(\bm{\xi}_{f,k}^r)\bm{\Pi}(\bm{\alpha}\bm{C}_m(\bm{\xi}_{v}^r))].
    \end{split}
\end{equation}

\noindent Factoring out the constant, utilizing GCL Eq.~(\ref{eq: GCL}), and the divergence of a constant is zero we are left with,

\begin{equation}
    \begin{split}
  \implies \Big(\bm{M}_m+\bm{K}_m\Big)^{-1}\sum_{f=1}^{N_f}\sum_{k=1}^{N_{fp}} 
   \bm{\chi}(\bm{\xi}_{f,k}^r)^T W_{f,k}[\frac{1}{2}\hat{\bm{n}}^r\bm{C}_m(\bm{\xi}_{f,k}^r)^T\cdot \bm{1} - \frac{1}{2}\hat{\bm{n}}^r\cdot \bm{\chi}(\bm{\xi}_{f,k}^r)\bm{\Pi}(\bm{1}\bm{C}_m(\bm{\xi}_{v}^r))].
    \end{split}\label{eq: free-stream1}
\end{equation}

\noindent Since the metrics computed in Eq.~(\ref{eq: discrete cons curl abe}) or~(\ref{eq: discrete inv curl abe}) are computed at a continuous set of grid nodes (included on the boundary), with only the last interpolation performed at the flux nodes, 

\begin{equation}
\begin{split}
   \Big( \bm{\chi}(\bm{\xi}_{f,k}^r)\bm{\Pi}(\bm{1}\bm{C}_m(\bm{\xi}_{v}^r))\Big)_{ni}
    =-\hat{\bm{e}}_i \cdot \bm{\Theta}(\bm{\xi}_{f,k}^r) \bm{\Theta}(\bm{\xi}_{v}^r)^{-1} \bm{\Theta}(\bm{\xi}_{v}^r)\nabla^r\times\bm{\Theta}(\bm{\xi}_{\text{grid nodes}}^r)
    \Big[
    \bm{\Theta}(\bm{\xi}_{\text{grid nodes}}^r)\hat{\bm{x}}_l^{c^T}
    \nabla^r \bm{\Theta}(\bm{\xi}_{\text{grid nodes}}^r)\hat{\bm{x}}_m^{c^T} \Big]
    \text{, }\\
    i=1,2,3\text{, }n=1,2,3\text{ }(n,m,l)\text{ cyclic,}\\
    =\Big(\bm{1}\bm{C}_m(\bm{\xi}_{f,k}^r)\Big)_{ni},
    \end{split}\label{eq: proj metrics vol to face}
\end{equation}

\noindent and similarly for the invariant curl formulation. 
Note that $\bm{\Theta}(\bm{\xi}_{v}^r)^{-1}$ is always true and appears from a change of basis in Eq.~(\ref{eq: proj metrics vol to face}). Thus Eq.~(\ref{eq: free-stream1}) becomes,

\begin{equation}
    \begin{split}
   \implies\Big(\bm{M}_m+\bm{K}_m\Big)^{-1}\sum_{f=1}^{N_f}\sum_{k=1}^{N_{fp}} 
   \bm{\chi}(\bm{\xi}_{f,k}^r)^T W_{f,k}[\frac{1}{2}\hat{\bm{n}}^r\bm{C}_m(\bm{\xi}_{f,k}^r)^T\cdot \bm{1} - \frac{1}{2}\hat{\bm{n}}^r\cdot (\bm{1}\bm{C}_m(\bm{\xi}_{f,k}^r))]=\bm{0}^T,
    \end{split}
\end{equation}

\noindent which concludes the proof since free-stream is preserved.\qed

\subsection{Conservation}\label{sec:Conservation}

To prove global and local conservation, we use quadrature rules exact of at least $2p-1$, and consider the $(\bm{M}_m+\bm{K}_m)$-norm,

\begin{equation}
\begin{split}
    \hat{\bm{1}}\Big(\bm{M}_m+\bm{K}_m\Big)\frac{d}{dt}\hat{\bm{u}}_m(t)^T 
  = -\frac{1}{2}\bm{1}\bm{W}\nabla^r\bm{\chi}(\bm{\xi}_v^r)\cdot \hat{\bm{f}}_m^r(t)^T
  -\frac{1}{2}\bm{1}\bm{W}\Tilde{\nabla}^r\bm{\chi}(\bm{\xi}_v^r)\cdot \hat{\bm{f}}_m(t)^T
  - \sum_{f=1}^{N_f}\sum_{k=1}^{N_{fp}} 1 W_{f,k}[\hat{\bm{n}}^r\cdot \bm{f}_m^{C,r}],
\label{splitformtoshowconservation}
\end{split}
\end{equation}

\noindent where $\hat{\bm{1}}$ is implicitly defined by $\bm{1}=[1,\dots,1]=\Big(\bm{\chi}(\bm{\xi}_v^r)\hat{\bm{1}}^T\Big)^T$. Discretely integrating both volume terms by parts yields the following expression for the right-hand-side of Eq.~(\ref{splitformtoshowconservation}), 
\begin{equation}
\begin{split}
=\frac{1}{2}\nabla^r(\bm{1})\bm{W}\cdot \bm{\chi}(\bm{\xi}_v^r)\hat{\bm{f}}_m^r(t)^T
+\frac{1}{2}(\nabla^r\cdot\bm{1}\bm{C}_m(\bm{\xi}_v^r))\bm{W}\cdot \bm{\chi}(\bm{\xi}_v^r)\hat{\bm{f}}_m(t)^T 
-\sum_{f=1}^{N_f}\sum_{k=1}^{N_{fp}} 1 W_{f,k}[\hat{\bm{n}}^r\bm{C}_m(\bm{\xi}_{f,k}^r)^T\cdot \bm{f}_m^{*}].
    \end{split}
\end{equation}

\noindent {\color{black}Utilizing the property of GCL from Eq.~(\ref{eq: GCL}) and that the gradient of a scalar is zero, allows the volume terms to vanish and local conservation is established,}

\begin{equation}
    \begin{split}
  \therefore     \hat{\bm{1}}\Big(\bm{M}_m+\bm{K}_m\Big)\frac{d}{dt}\hat{\bm{u}}(t)^T 
     =-\sum_{f=1}^{N_f}\sum_{k=1}^{N_{fp}} 1 W_{f,k}[\hat{\bm{n}}^r\bm{C}_m(\bm{\xi}_{f,k}^r)^T\cdot \bm{f}_m^{*}].
    \end{split}
\end{equation}

From the assumption of a conforming, water-tight mesh, then the interior normal equals the negative of the exterior normal, provided the surface metrics are computed by Eqs.~(\ref{eq: discrete cons curl abe}) or~(\ref{eq: discrete inv curl abe}), which concludes the proof for global conservation with periodic boundary conditions. \qed

\subsection{Stability}\label{sec:Stability}

We consider the broken Sobolev-norm ${W^{3p,2}_{\delta}}=\bm{M}_m+\bm{K}_m$ to demonstrate stability,

\begin{equation}
    \begin{split}
        \hat{\bm{u}}_m(t)\Big(\bm{M}_m+\bm{K}_m\Big)\frac{d}{d t}\hat{\bm{u}}_m(t)^T
        =\frac{1}{2}\frac{d}{d t}\|\bm{u}\|_{W^{3p,2}_{\delta}}^2=\frac{1}{2}\frac{d}{d t}\|\bm{u}\|_{\bm{M}_m+\bm{K}_m}^2\\
     =   -\frac{1}{2}\bm{u}_m\bm{W}\nabla^r\bm{\chi}(\bm{\xi}_v^r)\cdot \hat{\bm{f}}_m^r(t)^T
        -\frac{1}{2}\bm{u}_m\bm{W}\Tilde{\nabla}^r\bm{\chi}(\bm{\xi}_v^r)\cdot \hat{\bm{f}}_m(t)^T
        -\sum_{f=1}^{N_f}\sum_{k=1}^{N_{fp}} u_m W_{f,k}[\hat{\bm{n}}^r\cdot \bm{f}_m^{C,r}]
    \end{split}
\end{equation}

\noindent {\color{black}Next, we discretely integrate the first volume term by parts in the reference space},

\begin{equation}
    \begin{split}
   = \frac{1}{2}\nabla^r(\bm{u}_m)\bm{W}\cdot\bm{\chi}(\bm{\xi}_v^r) \hat{\bm{f}}_m^r(t)^T
        -\frac{1}{2}\bm{u}_m\bm{W}\Tilde{\nabla}^r\bm{\chi}(\bm{\xi}_v^r)\cdot \hat{\bm{f}}_m(t)^T
        -\sum_{f=1}^{N_f}\sum_{k=1}^{N_{fp}}u_m W_{f,k}[\hat{\bm{n}}^r\bm{C}_m(\bm{\xi}_{f,k}^r)^T\cdot (\bm{f}_m^{*}-\frac{1}{2}\bm{f}_m)].
    \end{split}\label{eq: stab step}
\end{equation}

\noindent Since the two volume terms in Eq.~(\ref{eq: stab step}) are equivalent,
\begin{equation}
    \begin{split}
        \nabla^r(\bm{u}_m)\bm{W}\cdot\bm{\chi}(\bm{\xi}_v^r) \hat{\bm{f}}_m^r(t)^T
        = \nabla^r(\bm{u}_m)\bm{W}\cdot{\bm{f}}_m\bm{C}_m(\bm{\xi}_v^r) 
        =\nabla^r(\bm{u}_m)\bm{W}\bm{C}_m(\bm{\xi}_v^r)^T\cdot{\bm{f}}_m
        =\bm{f}_m\bm{W}\cdot\nabla^r(\bm{u}_m) \bm{C}_m(\bm{\xi}_v^r)^T,\\
       \because \bm{f}_m = \bm{a}\bm{u}_m\implies 
        \nabla^r(\bm{u}_m)\bm{W}\cdot\bm{\chi}(\bm{\xi}_v^r) \hat{\bm{f}}_m^r(t)^T
        =\bm{u}_m\bm{W}\Tilde{\nabla}^r\bm{\chi}(\bm{\xi}_v^r)\cdot \hat{\bm{f}}_m(t)^T,
    \end{split}
\end{equation}

\noindent they discretely cancel for linear advection, and we are left with

\begin{equation}
    \frac{1}{2}\frac{d}{d t}\|\bm{u}\|_{W^{3p,2}_{\delta}}^2
 = -\sum_{f=1}^{N_f}\sum_{k=1}^{N_{fp}}u_m W_{f,k}[\hat{\bm{n}}^r\bm{C}_m(\bm{\xi}_{f,k}^r)^T\cdot (\bm{f}_m^{*}-\frac{1}{2}\bm{f}_m)],\label{eq: stab final cond}
\end{equation}

\noindent which concludes the proof since it is the same stability claim as that for a linear grid. Thus energy is conserved for a central numerical flux, and energy monotonically decreases for an upwind numerical flux with periodic boundary conditions~\cite{Cicchino2020NewNorm,Fernandez2019curvitensor}.\qed

For completeness, we present the operator form of the above stability proof in Appendix~\ref{sec: oper stab proof}.

\section{Results}\label{sec:Results}

In this section, we use the open-source Parallel High-order Library for PDEs (\texttt{PHiLiP}, \newline \url{https://github.com/dougshidong/PHiLiP.git})~\cite{shi2021full}, developed at the Computational Aerodynamics Group at McGill University, to numerically verify all proofs. Three tests are used: the first verifies Thm.~\ref{thm: filter all modes} for three-dimensions, the second verifies that the ESFR filter operator (divergence of the correction functions) must be applied to the volume for stability, and the third verifies Thm.~\ref{thm: difference cons and non cons} and Remark~\ref{remark: Jac dependent sob norm}. When we refer to \enquote{ESFR Classical Split}, we are using the split form with the ESFR correction functions only applied to the surface, whereas \enquote{ESFR Split} refers to our proposed provably stable ESFR split form with the correction functions applied on both the volume and surface terms.

{\color{black}For the order of accuracy (OOA) tests,} to compute the $L_2$ error, an overintegration of $p+10$ was used to provide sufficient strength,
\begin{equation}
    L_2 -error = \sqrt{\sum_{m=1}^{M}{\int_\Omega{(u_m-u)^2d\Omega}}}
    =\sqrt{\sum_{m=1}^{M}(\bm{u}_m^T-\bm{u}_{exact}^T)\bm{W}_m\bm{J}_m(\bm{u}_m-\bm{u}_{exact})}.
\end{equation}

\noindent We additionally compute the $L_\infty$ error as $sup(\bm{u}_m^T - \bm{u}_{exact}^T)$. In all experiments, our basis functions $\bm{\chi}(\bm{\xi}^r)$ are Lagrange polynomials constructed on GLL quadrature nodes. Our \enquote{grid nodes}, or mapping-support-points, are GLL quadrature nodes. Our \enquote{flux nodes} for integration are GL quadrature nodes. Lastly, all schemes were conservative on the order of $1e-15$.

\subsection{ESFR Derivative Test}\label{sec:ESFR deriv test}

The first numerical test addresses Thm.~\ref{thm: filter all modes}, where we solve for the divergence of the flux, ${\nabla}\cdot \bm{f}$. In this test, we only solve for the volume terms. We take the heavily warped grid in Fig.~\ref{fig: heavy warped}, defined by Eq.~(\ref{eq: heavily warped grid fn}), and distribute the flux from Eq.~(\ref{eq: nonlin flux for test}). Then we solve for the volume terms in Eq.~(\ref{eq: ESFR equiv dudt}), that approximate $\nabla\cdot\bm{f}$, for varying values of $c$.

\begin{figure}[H]
    \centering
    \includegraphics[scale=0.5]{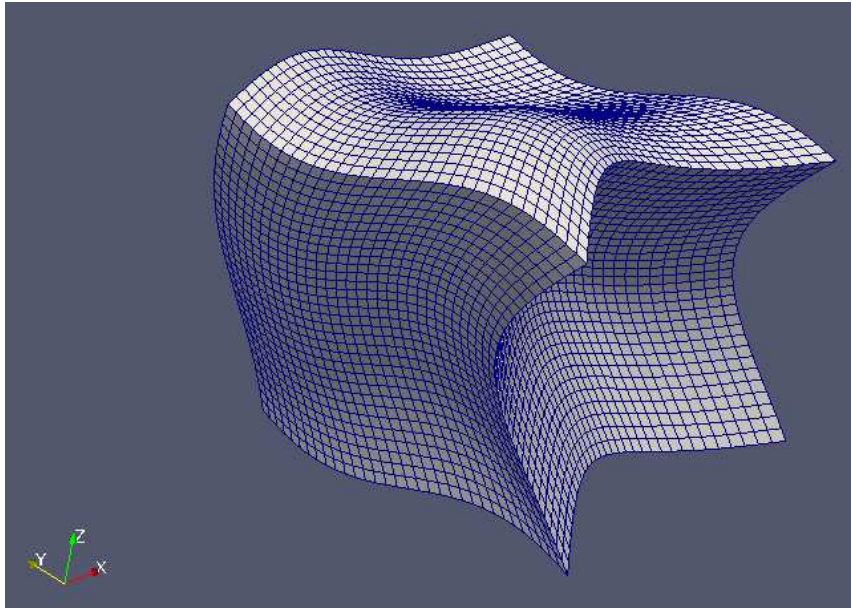}
    \caption{$3D$ Warped Grid}\label{fig: heavy warped}
\end{figure}

\begin{equation}
\begin{split}
 &   x = \xi +\frac{1}{10}\Big(\cos{\pi\eta}+\cos{\pi\zeta}\Big),\\
 &   y = \eta + \frac{1}{10}\exp{(1-\eta)}\Big(\sin{\pi\xi}+\sin{\pi\zeta}\Big),\\
 &   z = \zeta + \frac{1}{20}(\sin{2\pi \xi}+\sin{2\pi \eta}),\\
  &  [\xi,\eta,\zeta]\in[0,1]^3.
\end{split}\label{eq: heavily warped grid fn}
\end{equation}

\begin{equation}
    \begin{split}
  &  \bm{f} = [\exp{(-10x^2)}, \exp{(-10\pi y^3)}, \exp{(-10\sin{z})}],\\
  &     \nabla\cdot \bm{f}_{\text{exact}} = -10\Big(2x\exp{(-10x^2)} +3\pi y^2\exp{(-10\pi y^3)} + \cos{(z)} \exp{(-10\sin{z})}
       \Big).
    \end{split}\label{eq: nonlin flux for test}
\end{equation}

The maximum GCL computed for the grid was {\cal O}(1e-15). First, we demonstrate in Tables~\ref{tab: ooa1} through~\ref{tab: ooa4} that the error levels change as we increase $c$, but the orders remain unchanged until $c\gg c_+$. 
Next, for the polynomial order range, $p=2$ through $p=5$, we verify that applying the ESFR filter operator does not affect the order of accuracy up to a certain value, and by Thm.~\ref{thm: filter all modes}, the scheme loses all orders of accuracy at this value. The black star is the location of $c_+$ in Figures~\ref{fig: drop1} through~\ref{fig: drop2}. The drop off value of $c$ closely resembles the values obtained by Castonguay~\cite[Figure 3.6]{castonguay_phd}.

  \begin{table}[]
\parbox{.45\linewidth}{
\centering
  \begin{tabular}{c c c c c c  c}
dx &  $c_{DG}$ & OOA & $c_{+}$ & OOA\\ \hline 
3.125e-02 &  1.949e-02 &  - &  1.860e-02 & - \\
1.563e-02 &  2.587e-03  & 2.91   &   2.467e-03  &     2.91\\
7.813e-03  & 3.285e-04  &  2.98  &  3.133e-04   &   2.98\\
3.906e-03  &  4.122e-05  &  2.99  &   3.931e-05   &  2.99
\end{tabular} 
 \caption{$L_2$ Convergence Table $p=3$} \label{tab: ooa1}
 }
\hfill
\parbox{.45\linewidth}{
\centering
\centering
  \begin{tabular}{c c c c c c  c}
dx &  $c_{DG}$ & OOA & $c_{+}$ & OOA\\ \hline 
3.125e-02 &  4.436e-01 &  - &  3.713e-01 & - \\
1.563e-02 &   7.033e-02 &   2.66 &   5.746e-02  &   2.69  \\
7.813e-03  & 9.885e-03  &  2.83  &   8.056e-03  &  2.83 \\
3.906e-03  & 1.309e-03   & 2.92   &   1.072e-03   &  2.91
\end{tabular} 
 \caption{$L_\infty$ Convergence Table $p=3$} 
 }
 \end{table}


   \begin{table}[]
\parbox{.45\linewidth}{
\centering
  \begin{tabular}{c c c c c c  c}
dx &  $c_{DG}$ & OOA & $c_{+}$ & OOA\\ \hline 
2.5000e-02 & 1.618e-03 &-   & 1.559e-03  &  - \\
1.2500e-02 &  1.070e-04 & 3.92  &  1.032e-04   &  3.92 \\
  6.2500e-03  &  6.784e-06 &  3.98 &   6.542e-06   &  3.98 \\
 3.1250e-03  &  4.256e-07   & 3.99 &   4.103e-07  &  3.99
\end{tabular} \label{tab: ooa3}
 \caption{$L_2$ Convergence Table $p=4$} 
 }
\hfill
\parbox{.45\linewidth}{
\centering
\centering
  \begin{tabular}{c c c c c c  c}
dx &  $c_{DG}$ & OOA & $c_{+}$ & OOA\\ \hline 
2.5000e-02 & 4.648e-02 & -  & 4.085e-02   & -  \\
1.2500e-02 & 3.388e-03  & 3.78  &   3.033e-03 &  3.75 \\
  6.2500e-03  &  2.296e-04 & 3.88  &   2.056e-04   & 3.88  \\
 3.1250e-03  &  1.496e-05   & 3.94 &  1.341e-05   & 3.94
\end{tabular} 
 \caption{$L_\infty$ Convergence Table $p=4$} \label{tab: ooa4}
 }

 \end{table}

\begin{figure}
\begin{minipage}{.5\textwidth}
    \centering
    \includegraphics[scale=0.5]{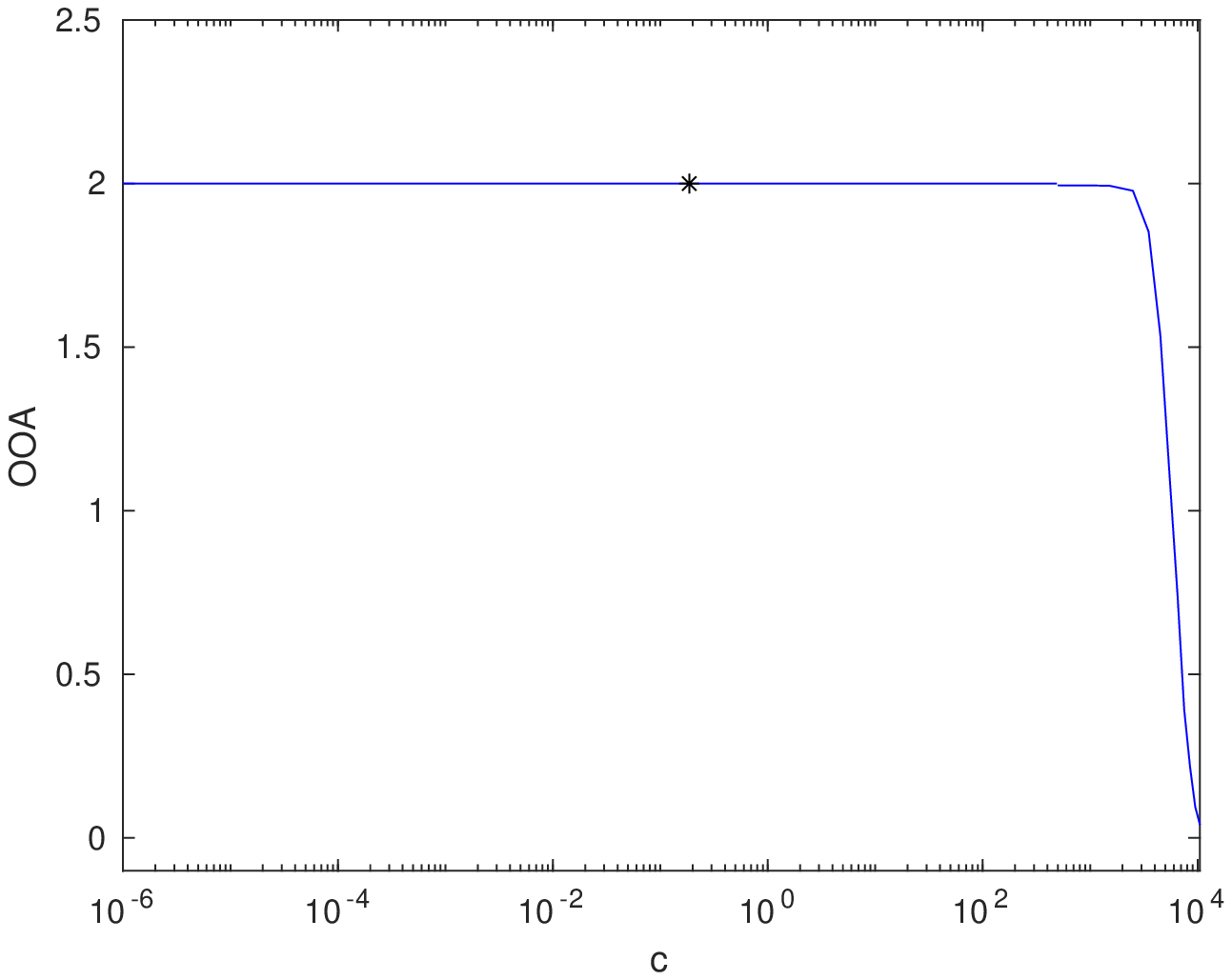}
    \caption{3D $c$ vs OOA for $p=2$}\label{fig: drop1}
\end{minipage}
\begin{minipage}{.5\textwidth}
    \centering
    \includegraphics[scale=0.5]{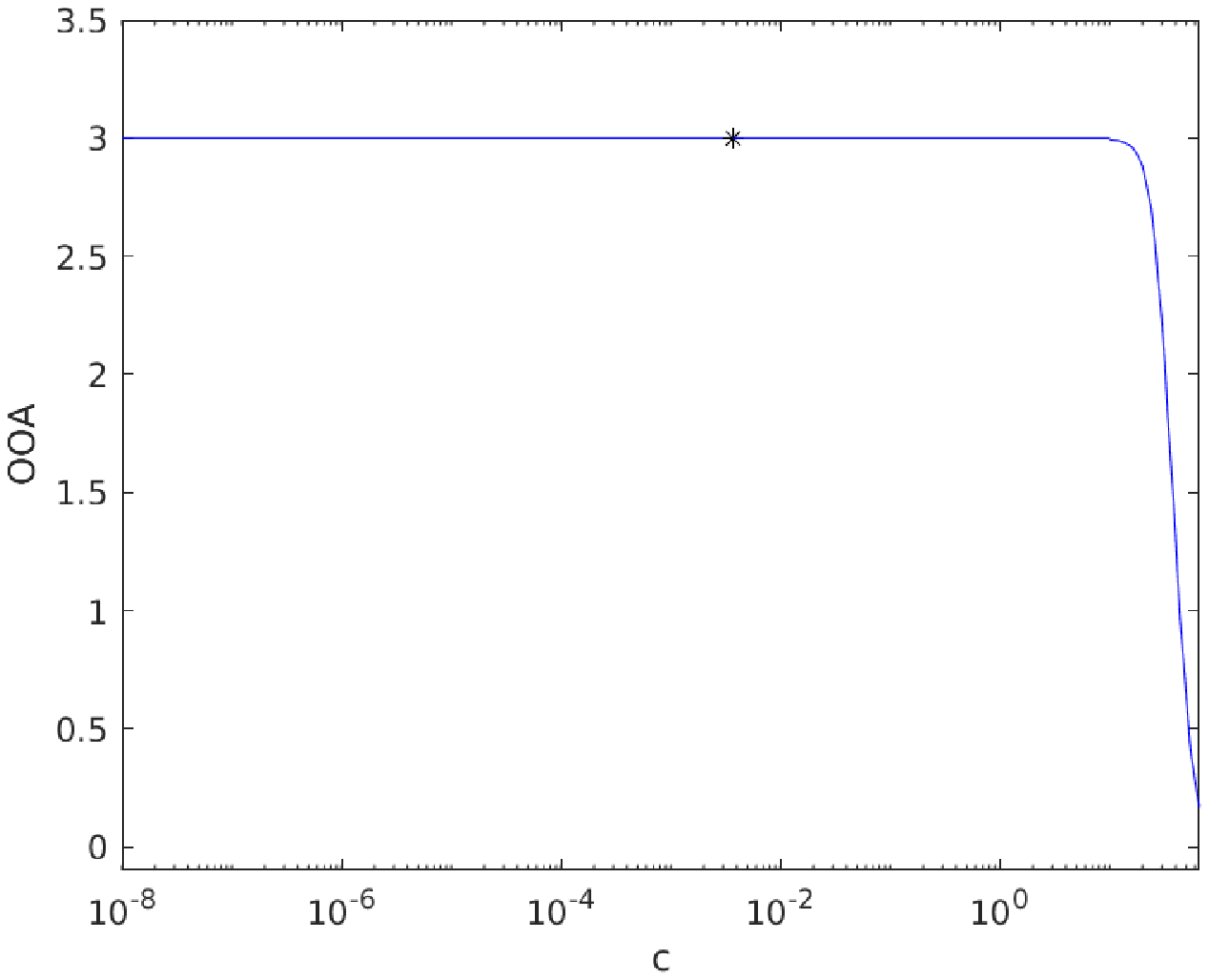}
    \caption{3D $c$ vs OOA for $p=3$}
\end{minipage}
\end{figure}
\begin{figure}
\begin{minipage}{.5\textwidth}
    \centering
    \includegraphics[scale=0.5]{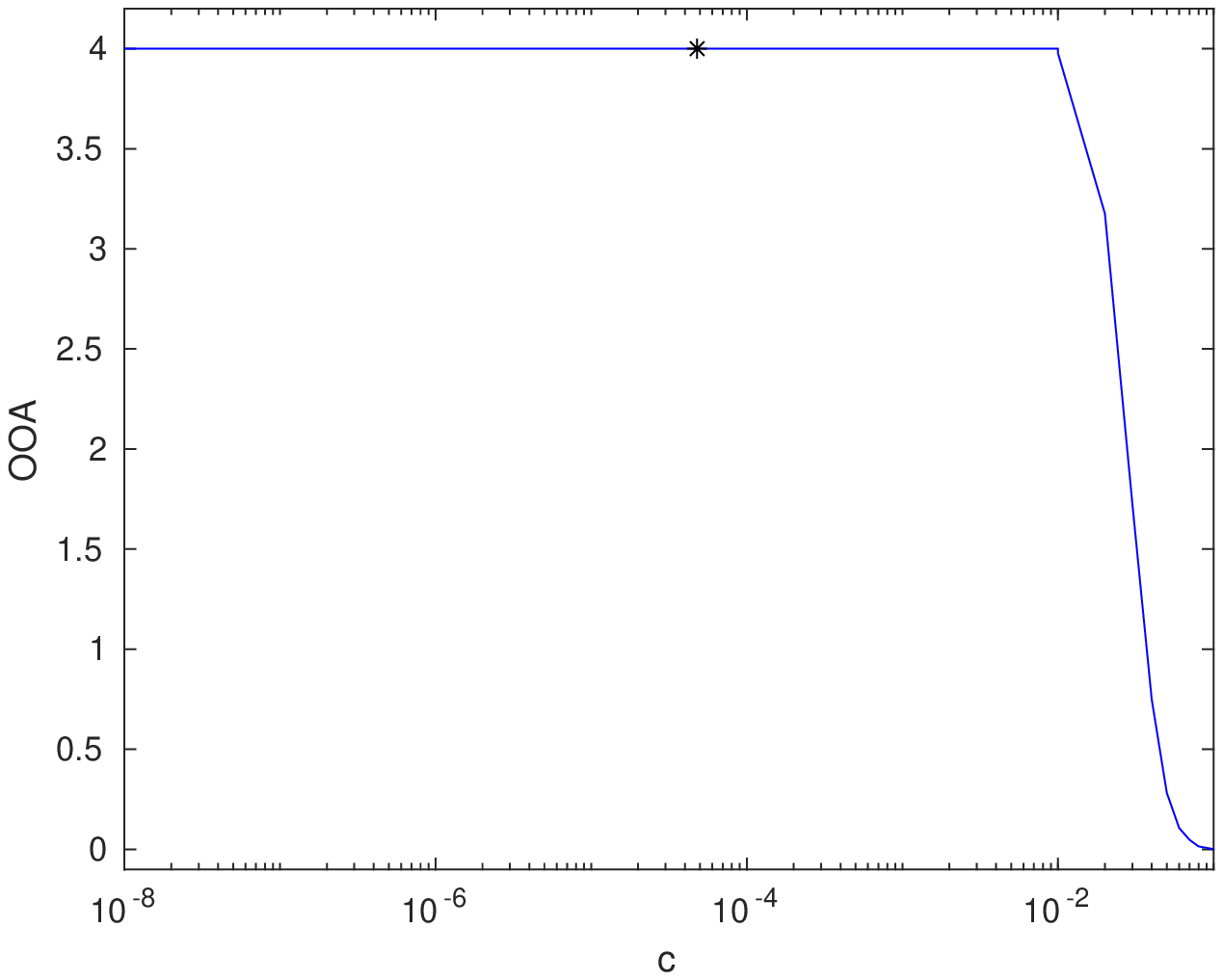}
    \caption{3D $c$ vs OOA for $p=4$}
\end{minipage}
\begin{minipage}{.5\textwidth}
    \centering
    \includegraphics[scale=0.5]{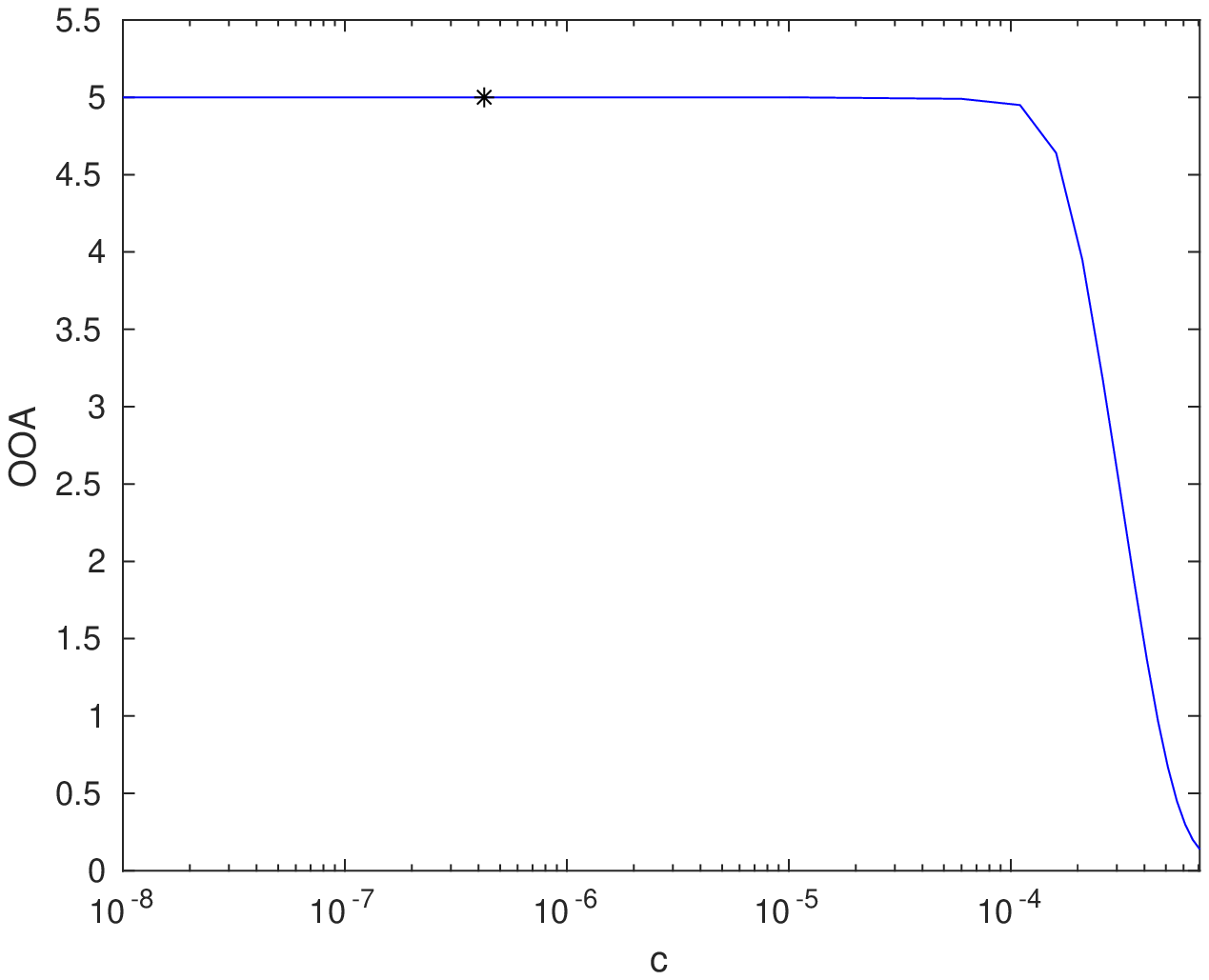}
    \caption{3D $c$ vs OOA for $p=5$}\label{fig: drop2}
\end{minipage}
\end{figure}

\subsection{Nonsymmetric Grid}

As illustrated by Thm~\ref{thm: difference cons and non cons}, the nonlinear metric terms vanish for both symmetric and skew-symmetric grids; resulting in a false-positive stable solution. Thus, a nonsymmetric grid was chosen to ensure that nonlinear metric terms are present. The warping for the nonsymmetric grid is {\color{black}similar to that used by Wu \textit{et al}.~\cite{wu2021high}}, defined by Eq.~(\ref{eq: nonsym grid}), and the grid is illustrated in Fig.~\ref{fig: nonsym grid},


\begin{figure}[H]
    \centering
    \includegraphics[scale=0.5]{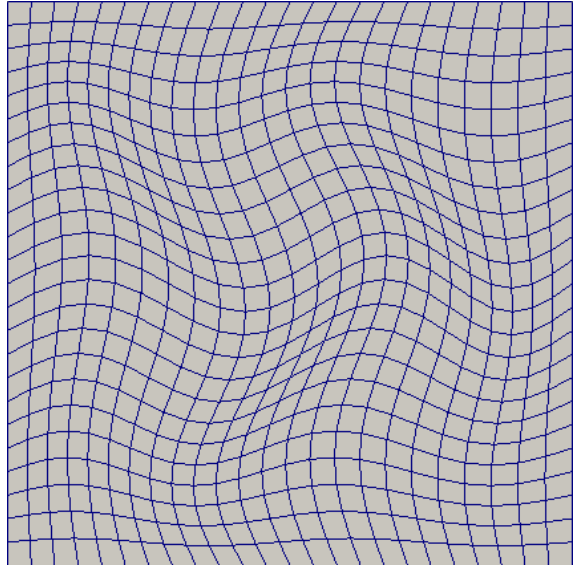}
    \caption{Warped Grid $p=3$}\label{fig: nonsym grid}
\end{figure}

\begin{equation}
\begin{split}
  &  x = \xi +\frac{1}{10}\cos{\frac{\pi}{2}\xi}\cos{\frac{3\pi}{2}\eta}\\
  &  y = \eta + \frac{1}{10}\sin{2\pi\xi}\cos{\frac{\pi}{2}\eta},\\
  &[\xi,\eta]\in[-1,1]^2 .
\end{split}\label{eq: nonsym grid}
\end{equation}

\noindent We apply the following linear advection problem in Eq.~(\ref{eq: lin adv grid 1}),

\begin{equation}
\begin{split}
    \frac{\partial u}{\partial t} + 1.1\frac{\partial u}{\partial x} -\frac{\pi}{e}\frac{\partial u}{\partial y} = 0,\\
    u(x,y,0) = e^{-20(x^2+y^2)},
    \end{split}\label{eq: lin adv grid 1}
\end{equation}

\noindent with periodic boundary conditions. We integrate with a Runge-Kutta-4 integrator, using a timestep $dt=0.05 dx$, with $dx$ being the average distance between two quadrature nodes. The grid is partitioned into $8^2$ elements. 
All results are uncollocated, with the solution built on the GLL nodes and integrated on the GL nodes. Since our metrics were computed via Eq.~(\ref{eq: discrete cons curl abe}) and surface splitting was used, we were able to use GL nodes for both volume and surface integration without any form of optimization seen in the literature~\cite{Fernandez2019curvitensor}. We first show energy results for uncollocated integration in Table~\ref{tab:grid 1 stability}, then uncollocated overintegration in Table~\ref{tab:grid 1 stability overint}.

 \begin{table}[t]
\resizebox{\textwidth}{!}{
 \begin{tabular}{c c c c }
Scheme &  Flux & Energy Conserved $\mathcal{O}$(1e-12) & Energy Monotonically Decrease \\ \hline 
Cons. DG & Central & No & No \\ \hline 
Cons. DG & Upwind & No & No \\ \hline 
EFSR Split $c_{DG}$ & Central & \textbf{Yes} & \textbf{Yes} \\ \hline 
EFSR Split $c_{DG}$ & Upwind & No & \textbf{Yes} \\ \hline 
EFSR Split $c_{+}$ & Central  & \textbf{Yes} &\textbf{Yes} \\ \hline 
EFSR Split $c_{+}$ & Upwind &No & \textbf{Yes} \\ \hline 
EFSR Classical Split $c_{+}$ & Central  & No & No \\ \hline 
EFSR Classical Split $c_{+}$ & Upwind & No & No \footnotemark[1]
\end{tabular} 
 }
 \caption{Energy Results $p=3,4$ Uncollocated $N_{vp}=(p+1)^2$ Grid 1} 
 \label{tab:grid 1 stability}
 \end{table} 
 
 \footnotetext[1]{Although the energy did not monotonically decrease for this case, it did not diverge either. Instead it gradually decreased over time giving a false positive.}
 
  \begin{table}[t]
\resizebox{\textwidth}{!}{
 \begin{tabular}{c c c c }
Scheme &  Flux & Energy Conserved $\mathcal{O}$(1e-12) & Energy Monotonically Decrease \\ \hline 
Cons. DG & Central & No & No \\ \hline 
Cons. DG & Upwind & No & No \\ \hline 
EFSR Split $c_{DG}$ & Central & \textbf{Yes} & \textbf{Yes} \\ \hline 
EFSR Split $c_{DG}$ & Upwind & No & \textbf{Yes} \\ \hline 
EFSR Split $c_{+}$ & Central  & \textbf{Yes} & \textbf{Yes} \\ \hline 
EFSR Split $c_{+}$ & Upwind &No & \textbf{Yes} \\ \hline 
EFSR Classical Split $c_{+}$ & Central  & No & No \\ \hline 
EFSR Classical Split $c_{+}$ & Upwind & No & No \footnotemark[1]
\end{tabular} 
 }
 \caption{Energy Results $p=3,4$ Uncollocated $N_{vp}=(p+3)^2$ Grid 1} 
 \label{tab:grid 1 stability overint}
 \end{table} 

An interesting result in Tables~\ref{tab:grid 1 stability} and~\ref{tab:grid 1 stability overint} is the false positive for the ESFR Classical split with an upwind numerical flux. {\color{black} From the derivation of our proposed curvilinear FR schemes in Sec.~\ref{sec: ESFR classical}, specifically Eq.~(\ref{eq: ESFR 3p deriv}), there is no stability claim for the two terms $\bm{K}_m\Big( \bm{J}_m^{-1} \nabla^r \bm{\chi}(\bm{\xi}_v^r)\Big)$ and $\bm{K}_m\Big( \bm{J}_m^{-1}\Tilde{\nabla}^r\bm{\chi}(\bm{\xi}_v^r)\Big)$ as they can result in either a convergent or divergent scheme. The advantage of our proposed FR schemes is that they are provably stable. In the next test case, it will be shown that the ESFR classical split is divergent for a skew-symmetric grid.}


To demonstrate the orders of accuracy, we consider the linear advection problem in Eq.~(\ref{eq: IC grid 1}),

\begin{equation}
\begin{split}
    \frac{\partial u}{\partial t} + \frac{\partial u}{\partial x} +\frac{\partial u}{\partial y} = 0,\\
    [x,y]\in[-1,1]^2,\text{ } t\in[0,2]\\
        u(x,y,0) = \sin{\pi x}\sin{\pi y},\\
        u_{exact}(x,y,t) = \sin{\pi (x-t)}\sin{\pi (y-t)}.
    \end{split}\label{eq: IC grid 1}
\end{equation}

We used a timestep of $dt=0.5 dx$, where again $dx$ is the average distance between two quadrature nodes. The convergence rates are shown in Tables~\ref{tab: ooa1 grid 1} through~\ref{tab: ooa4 grid 1} for uncollocated integration, and Tables~\ref{tab: ooa5 grid 1} through~\ref{tab: ooa8 grid 1} for uncollocated overintegration.


 \begin{table}[]
\parbox{.45\linewidth}{
\centering
 \begin{tabular}{c c c c c c  c}
dx &  $c_{DG}$ & OOA & $c_{+}$ & OOA\\ \hline 
6.2500e-02 & 1.4592e-02 &  - & 3.9628e-02& - \\
3.1250e-02 & 1.1632e-03 &  3.65 &3.0945e-03& 3.68\\
1.5625e-02 &   7.4833e-05& 3.96&   1.7779e-04    & 4.12\\
7.8125e-03  &  4.7374e-06   & 3.98 &  1.0851e-05  & 4.03 \\
3.9062e-03   &   3.0227e-07  & 3.97   & 6.8311e-07 & 3.99 
\end{tabular} 
 \caption{$L_2$ Convergence Table $p=3$ $N_{vp}=(p+1)^2$ Upwind Numerical Flux Grid 1} \label{tab: ooa1 grid 1}
 
 }
\hfill
\parbox{.45\linewidth}{
\centering
  \begin{tabular}{c c c c c c  c}
dx &  $c_{DG}$ & OOA & $c_{+}$ & OOA\\ \hline 
6.2500e-02 & 4.7490e-02 &  - & 1.1192e-01& - \\
3.1250e-02 & 5.2854e-03 &3.17 &1.4510e-02 & 2.95\\
1.5625e-02 & 3.6961e-04  & 3.84 &     1.5445e-03  & 3.23\\
7.8125e-03  &  2.5181e-05   & 3.88 &   8.0837e-05 & 4.26 \\
3.9062e-03   &  1.6401e-06  &  3.94    &5.2672e-06  & 3.94
\end{tabular} 
 \caption{$L_\infty$ Convergence Table $p=3$ $N_{vp}=(p+1)^2$ Upwind Numerical Flux Grid 1} 
}
 
 \end{table}

 
  \begin{table}[]
\parbox{.45\linewidth}{
\centering
 \begin{tabular}{c c c c c c  c}
dx &  $c_{DG}$ & OOA & $c_{+}$ & OOA\\ \hline 
5.0000e-02&3.7766e-03 &- &  8.1107e-03  & - \\
 2.5000e-02  & 1.4876e-04 & 4.67 & 2.5675e-04   & 4.98\\
 1.2500e-02 & 5.1042e-06 & 4.87   &  9.2869e-06 &   4.79      \\
 6.2500e-03 & 1.6763e-07  & 4.93  &    3.1350e-07& 4.89    \\
 3.1250e-03&  5.4776e-09 & 4.94 &   1.0345e-08 & 4.92
\end{tabular} 
 \caption{$L_2$ Convergence Table $p=4$ $N_{vp}=(p+1)^2$ Upwind Numerical Flux Grid 1} 
 }
\hfill
\parbox{.45\linewidth}{
\centering
\centering
  \begin{tabular}{c c c c c c  c}
dx &  $c_{DG}$ & OOA & $c_{+}$ & OOA\\ \hline 
5.0000e-02&1.7943e-02 &- &  3.4934e-02  & - \\
 2.5000e-02  &  7.1420e-04&4.65   &  1.8378e-03  & 4.25\\
 1.2500e-02 & 2.7883e-05  & 4.67   &   6.7414e-05 & 4.77    \\
 6.2500e-03 & 1.0551e-06   &  4.72  &    2.9290e-06& 4.52    \\
 3.1250e-03&  3.4989e-08 & 4.91 &  1.0435e-07 &   4.81  
\end{tabular} 
 \caption{$L_\infty$ Convergence Table $p=4$ $N_{vp}=(p+1)^2$ Upwind Numerical Flux Grid 1} \label{tab: ooa4 grid 1}
 }

 \end{table}





 \begin{table}[]
\parbox{.45\linewidth}{
\centering
 \begin{tabular}{c c c c c c  c}
dx &  $c_{DG}$ & OOA & $c_{+}$ & OOA\\ \hline 
6.2500e-02 & 1.4539e-02 &  - &3.9565e-02 & - \\
3.1250e-02 &1.1594e-03  & 3.65& 3.0883e-03&3.68 \\
1.5625e-02 & 7.4762e-05  & 3.95&   1.7771e-04     & 4.12\\
7.8125e-03  &  4.7363e-06 &3.98     & 1.0849e-05   & 4.03 \\
3.9062e-03   &  3.0074e-07 & 3.98     & 6.8243e-07 & 3.99  
\end{tabular} 
 \caption{$L_2$ Convergence Table $p=3$ $N_{vp}=(p+3)^2$ Upwind Numerical Flux Grid 1} \label{tab: ooa5 grid 1}
 
 }
\hfill
\parbox{.45\linewidth}{
\centering
  \begin{tabular}{c c c c c c  c}
dx &  $c_{DG}$ & OOA & $c_{+}$ & OOA\\ \hline 
6.2500e-02 & 4.7467e-02 &  - &1.1172e-01 & - \\
3.1250e-02 &5.2245e-03 & 3.18&1.4400e-02 & 2.96\\
1.5625e-02 & 3.6902e-04  & 3.82 &    1.5435e-03  & 3.22\\
7.8125e-03  &  2.5167e-05  &3.87    &  8.0826e-05 & 4.26 \\
3.9062e-03   &1.6398e-06 &  3.94    &  5.2670e-06 & 3.94 
\end{tabular} 
 \caption{$L_\infty$ Convergence Table $p=3$ $N_{vp}=(p+3)^2$ Upwind Numerical Flux Grid 1} 
}
 
 \end{table}

 
  \begin{table}[]
\parbox{.45\linewidth}{
\centering
 \begin{tabular}{c c c c c c  c}
dx &  $c_{DG}$ & OOA & $c_{+}$ & OOA\\ \hline 
5.0000e-02&3.7361e-03 &- &   8.0479e-03 & - \\
 2.5000e-02  & 1.4812e-04 & 4.66 &    2.5660e-04& 4.97\\
 1.2500e-02 &5.0980e-06 &  4.86  &   9.2793e-06 & 4.79     \\
 6.2500e-03 &  1.6758e-07 &  4.93  & 3.1344e-07 &4.89       \\
 3.1250e-03& 5.4642e-09  & 4.94   & 1.0338e-08  &  4.92  
\end{tabular} 
 \caption{$L_2$ Convergence Table $p=4$ $N_{vp}=(p+3)^2$ Upwind Numerical Flux Grid 1} 
 }
\hfill
\parbox{.45\linewidth}{
\centering
\centering
  \begin{tabular}{c c c c c c  c}
dx &  $c_{DG}$ & OOA & $c_{+}$ & OOA\\ \hline 
5.0000e-02& 1.7408e-02&- & 3.4462e-02   & - \\
 2.5000e-02  & 7.0797e-04 & 4.62  & 1.8335e-03   & 4.23\\
 1.2500e-02 & 2.7798e-05 &4.67   &  6.7391e-05   & 4.77    \\
 6.2500e-03 &  1.0549e-06 & 4.72   &   2.9288e-06 &  4.52     \\
 3.1250e-03& 3.4983e-08 &   4.91   &  1.0434e-07  & 4.81  
\end{tabular} 
 \caption{$L_\infty$ Convergence Table $p=4$ $N_{vp}=(p+3)^2$ Upwind Numerical Flux Grid 1} \label{tab: ooa8 grid 1}
 }

 \end{table}

\subsection{Skew-Symmetric Grid}
For further verification, we conduct the same experiments on the skew-symmetric grid from Hennemann \textit{et al}.~\cite{hennemann2021provably}, shown in Fig.\ref{fig: skewsym grid}, with warping defined in Eq.~(\ref{eq: skewsym grid warp}),

\begin{figure}[H]
    \centering
    \includegraphics[scale=0.5]{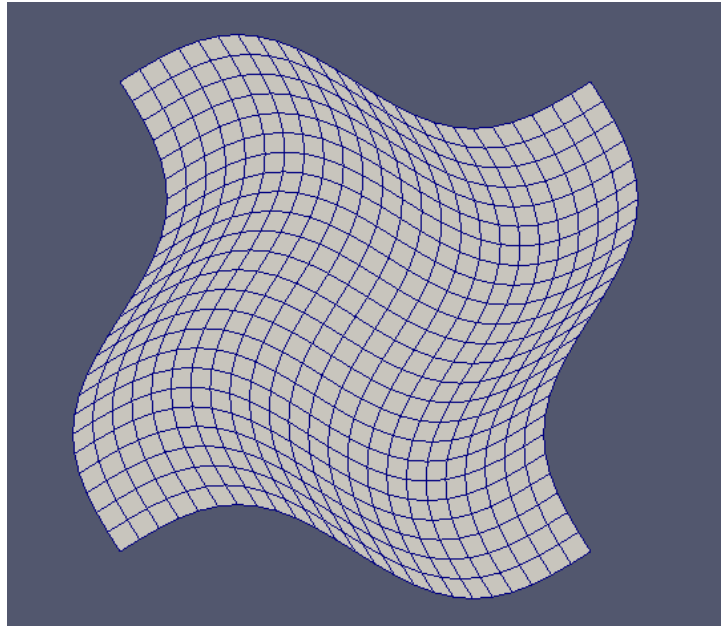}
    \caption{Second Warped Grid $p=3$}\label{fig: skewsym grid}
\end{figure}

\begin{equation}
\begin{split}
 &   x = \xi -0.1\sin{2\pi\eta},\\
 &   y = \eta +0.1\sin{2\pi\xi},\\
 &   [\xi,\eta]\in[0,1]^2.
\end{split}\label{eq: skewsym grid warp}
\end{equation}

\noindent The sole purpose of using a skew-symmetric grid, is to show that even in the case when the grid has all the nonlinear metric terms cancel out, as per Thm.~\ref{thm: difference cons and non cons}, Remark~\ref{remark: Jac dependent sob norm} holds because the determinant of the metric Jacobian cannot be factored out of the $\partial^{(s,v,w)}$ derivative. Thus, the ESFR correction functions must satisfy the metric dependent stability criteria in Eq.~(\ref{eq: CURV ESFR fund assumpt}). We use the same initial condition described in Eq.~(\ref{eq: lin adv grid 1}) for the energy results, 
presented in Tables~\ref{tab: energy grid2} and~\ref{tab: energy2 grid2}.

 \begin{table}[t]
\resizebox{\textwidth}{!}{
 \begin{tabular}{c c c c }
Scheme &  Flux & Energy Conserved $\mathcal{O}$(1e-12) & Energy Monotonically Decrease \\ \hline 
Cons. DG & Central & \textbf{Yes} & \textbf{Yes} \\ \hline 
Cons. DG & Upwind & No & \textbf{Yes}\\ \hline 
EFSR Split $c_{DG}$ & Central & \textbf{Yes} & \textbf{Yes} \\ \hline 
EFSR Split $c_{DG}$ & Upwind & No & \textbf{Yes} \\ \hline 
EFSR Split $c_{+}$ & Central  & \textbf{Yes} & \textbf{Yes} \\ \hline 
EFSR Split $c_{+}$ & Upwind &No & \textbf{Yes} \\ \hline 
EFSR Classical Split $c_{+}$ & Central  & No & No \\ \hline 
EFSR Classical Split $c_{+}$ & Upwind & No & No
\end{tabular} 
 }\label{tab: grid 1 stability}
 \caption{Energy Results $p=3,4$ Uncollocated $N_{vp}=(p+1)^2$ Grid 2} \label{tab: energy grid2}
 \end{table} 
  \begin{table}[t]
\resizebox{\textwidth}{!}{
 \begin{tabular}{c c c c }
Scheme &  Flux & Energy Conserved $\mathcal{O}$(1e-12) & Energy Monotonically Decrease \\ \hline 
Cons. DG & Central &\textbf{Yes} & \textbf{Yes}\\ \hline 
Cons. DG & Upwind & No & \textbf{Yes} \\ \hline 
EFSR Split $c_{DG}$ & Central & \textbf{Yes} & \textbf{Yes} \\ \hline 
EFSR Split $c_{DG}$ & Upwind & No &\textbf{Yes} \\ \hline 
EFSR Split $c_{+}$ & Central  &\textbf{Yes} & \textbf{Yes} \\ \hline 
EFSR Split $c_{+}$ & Upwind &No & \textbf{Yes} \\ \hline 
EFSR Classical Split $c_{+}$ & Central  & No & No \\ \hline 
EFSR Classical Split $c_{+}$ & Upwind & No & No
\end{tabular} 
 }\label{tab: grid 1 stability}
 \caption{Energy Results $p=3,4$ Uncollocated $N_{vp}=(p+3)^2$ Grid 2} \label{tab: energy2 grid2}
 \end{table}

 An interesting result from this grid is that conservative DG without the split form was stable, due to the skew-symmetry of the grid; while, ESFR classical split form was unstable. This highlights the importance of false positives while testing curvilinear grids. 
To demonstrate the orders of accuracy, we consider the linear advection problem from Eq.~(\ref{eq: IC grid 1}), and present the results in Tables~\ref{tab: ooa1 grid2} to~\ref{tab: ooa8 grid2}.

 \begin{table}[]
\parbox{.45\linewidth}{
\centering
 \begin{tabular}{c c c c c c  c}
dx &  $c_{DG}$ & OOA & $c_{+}$ & OOA\\ \hline 
6.2500e-02 & 6.9001e-03 &  - &2.2205e-02 & - \\
3.1250e-02 & 4.9929e-04 &3.79  & 2.1666e-03& 3.36\\
1.5625e-02 &  3.0374e-05 &4.04  &   9.5947e-05& 4.50     \\
7.8125e-03  &  1.9340e-06  &3.97    &  5.6723e-06 & 4.08   \\
3.9062e-03   &  1.2339e-07 &3.97   &  3.5426e-07& 4.00 
\end{tabular} 
 \caption{$L_2$ Convergence Table $p=3$ $N_{vp}=(p+1)^2$ Upwind Numerical Flux Grid 2} \label{tab: ooa1 grid2}
 
 }
\hfill
\parbox{.45\linewidth}{
\centering
  \begin{tabular}{c c c c c c  c}
dx &  $c_{DG}$ & OOA & $c_{+}$ & OOA\\ \hline 
6.2500e-02 & 4.0068e-02 &  - &1.6836e-01 & - \\
3.1250e-02 &3.7121e-03  &3.43  & 2.7181e-02& 2.63  \\
1.5625e-02 &  3.0497e-04 &3.61  &    1.2679e-03 & 4.42    \\
7.8125e-03  &   2.1527e-05   & 3.82  &  6.7106e-05 &  4.24   \\
3.9062e-03   &  1.4001e-06 &   3.94  &  4.0344e-06  &   4.06  
\end{tabular} 
 \caption{$L_\infty$ Convergence Table $p=3$ $N_{vp}=(p+1)^2$ Upwind Numerical Flux Grid 2} 
}
 
 \end{table}

 
  \begin{table}[]
\parbox{.45\linewidth}{
\centering
 \begin{tabular}{c c c c c c  c}
dx &  $c_{DG}$ & OOA & $c_{+}$ & OOA\\ \hline 
5.0000e-02&1.5174e-03  &- &   3.7476e-03 & - \\
 2.5000e-02  & 4.8840e-05 & 4.96&  1.0266e-04 &5.19   \\
 1.2500e-02 &  1.6575e-06& 4.88   &   3.7381e-06  &4.78    \\
 6.2500e-03 &  5.9007e-08& 4.81   &  1.4512e-07 &4.69      \\
 3.1250e-03&   2.1770e-09  &  4.76    & 5.2669e-09&   4.78 
\end{tabular} 
 \caption{$L_2$ Convergence Table $p=4$ $N_{vp}=(p+1)^2$ Upwind Numerical Flux Grid 2} 
 }
\hfill
\parbox{.45\linewidth}{
\centering
\centering
  \begin{tabular}{c c c c c c  c}
dx &  $c_{DG}$ & OOA & $c_{+}$ & OOA\\ \hline 
5.0000e-02& 9.6178e-03&- &  2.4719e-02  & - \\
 2.5000e-02  & 3.9077e-04 &4.62    &  1.2826e-03 & 4.27   \\
 1.2500e-02 & 1.5061e-05 &4.70  &   5.5561e-05  &  4.53   \\
 6.2500e-03 & 6.9115e-07 &  4.45    &  2.6736e-06  &    4.38   \\
 3.1250e-03& 2.4981e-08 &  4.79    & 1.0379e-07  & 4.69 
\end{tabular} 
 \caption{$L_\infty$ Convergence Table $p=4$ $N_{vp}=(p+1)^2$ Upwind Numerical Flux Grid 2} 
 }

 \end{table}





 \begin{table}[]
\parbox{.45\linewidth}{
\centering
 \begin{tabular}{c c c c c c  c}
dx &  $c_{DG}$ & OOA & $c_{+}$ & OOA\\ \hline 
6.2500e-02 & 6.8280e-03 &  - &2.2131e-02 & - \\
3.1250e-02 & 4.9794e-04 &3.78  & 2.1647e-03& 3.35 \\
1.5625e-02 & 3.0357e-05  &4.04  &    9.5922e-05& 4.50    \\
7.8125e-03  &  1.9337e-06  & 3.97   &  5.6719e-06  &  4.08  \\
3.9062e-03   &  1.2338e-07  &3.97    & 3.5425e-07 &  4.00 
\end{tabular} 
 \caption{$L_2$ Convergence Table $p=3$ $N_{vp}=(p+3)^2$ Upwind Numerical Flux Grid 2} 
 
 }
\hfill
\parbox{.45\linewidth}{
\centering
  \begin{tabular}{c c c c c c  c}
dx &  $c_{DG}$ & OOA & $c_{+}$ & OOA\\ \hline 
6.2500e-02 & 3.9689e-02 &  - &1.6899e-01 & - \\
3.1250e-02 & 3.6977e-03 &  3.42&2.7110e-02  &2.64  \\
1.5625e-02 &  3.0464e-04 &3.60  &   1.2676e-03  &4.42    \\
7.8125e-03  &  2.1521e-05  & 3.82    &  6.7102e-05 & 4.24    \\
3.9062e-03   &  1.4000e-06  &  3.94    & 4.0343e-06  &  4.06  
\end{tabular} 
 \caption{$L_\infty$ Convergence Table $p=3$ $N_{vp}=(p+3)^2$ Upwind Numerical Flux Grid 2} 
}
 
 \end{table}

 
  \begin{table}[]
\parbox{.45\linewidth}{
\centering
 \begin{tabular}{c c c c c c  c}
dx &  $c_{DG}$ & OOA & $c_{+}$ & OOA\\ \hline 
5.0000e-02& 1.5058e-03&- &  3.7268e-03  & - \\
 2.5000e-02  & 4.8725e-05 &4.95   &  1.0251e-04 & 5.18  \\
 1.2500e-02 &1.6566e-06  & 4.88  &    3.7369e-06 &4.78    \\
 6.2500e-03 & 5.8997e-08  & 4.81   &  1.4511e-07  &  4.69    \\
 3.1250e-03& 2.1769e-09 & 4.76    &   5.2668e-09 &  4.78
\end{tabular} 
 \caption{$L_2$ Convergence Table $p=4$ $N_{vp}=(p+3)^2$ Upwind Numerical Flux Grid 2} 
 }
\hfill
\parbox{.45\linewidth}{
\centering
\centering
  \begin{tabular}{c c c c c c  c}
dx &  $c_{DG}$ & OOA & $c_{+}$ & OOA\\ \hline 
5.0000e-02& 9.5289e-03&- &  2.4636e-02  & - \\
 2.5000e-02  & 3.8874e-04 & 4.62  &  1.2817e-03 & 4.26  \\
 1.2500e-02 & 1.5046e-05&  4.69  &   5.5551e-05  & 4.53   \\
 6.2500e-03 & 6.9098e-07 &   4.44   &   2.6735e-06   &   4.38 \\
 3.1250e-03&  2.4979e-08 &   4.79  &   1.0379e-07  &   4.69  
\end{tabular} 
 \caption{$L_\infty$ Convergence Table $p=4$ $N_{vp}=(p+3)^2$ Upwind Numerical Flux Grid 2}  \label{tab: ooa8 grid2}
 }

 \end{table}

\section{Conclusion}\label{sec:Conclusion}

This article proved that discrete integration by parts is not satisfied in the physical space for DG conservative and non-conservative forms, as well as standard FR forms, even with analytically exact metric terms and exact integration
---provided that the basis functions are polynomial in the reference space. This lead to the formulation of metric dependent FR correction functions. Through the construction of metric dependent FR correction functions, the inclusion of metric Jacobian dependence within arbitrarily dense norms was derived and manifested through the FR broken Sobolev-norm. The resultant curvilinear expression had the correction functions filtering all modes of the discretization. 
The theoretical findings were numerically verified with a three-dimensional, heavily warped, non-symmetric grid, where the orders of convergence were lost at the equivalent correction parameter value $c$ as that of the one-dimensional ESFR scheme.

We derived dense, modal or nodal, FR schemes in curvilinear coordinates that ensured provable stability and conservation. This was achieved by incorporating the FR correction functions (FR filter operator) on both the volume and surface terms. Through a suite of curvilinear test-cases, one being non-symmetric and the other being skew-symmetric, the provable stability claim was numerically verified for our proposed FR schemes. The choice of grids highlighted the importance of assessing false-positives, especially in curvilinear coordinates where metric skew-symmetry has the metric cross-terms cancel out, as well as when metric symmetry combined with equivalent advection speeds in every physical direction results in an equivalence between the conservative and non-conservative forms. 
It was also numerically verified that FR schemes that solely use the correction functions to reconstruct the surface {\color{black}are} divergent in general curvilinear coordinates---in both conservative and in split form. Lastly, we demonstrate that the proposed FR scheme retains optimal orders of accuracy in the appropriate range of $c$ values.

\section{Acknowledgements}

We would like to gratefully acknowledge the financial support of the Natural Sciences and Engineering Research Council of Canada Discovery Grant Program and McGill University. Jesse Chan acknowledges support from the US National Science Foundation under awards DMS-1719818 and DMS-1943186.

\appendix
\section{Summation-by-Parts}

The proposed algorithm in this paper is inspired by developments in the SBP literature, but derived using standard techniques and arguments from both the DG and FR communities. In this section, we make the link to the SBP formalism direct by assembling the relevant SBP operators. 

The stiffness operators satisfy discrete integration by parts for quadrature rules of at least $2p-1$ strength,

\begin{equation}
\begin{split}
    \int_{\bm{\Omega}_r} \chi_{i}(\bm{\xi}^r) \nabla^r\chi_{j}(\bm{\xi}^r) d\bm{\Omega}_r 
    +\int_{\bm{\Omega}_r} \nabla^r\chi_{i}(\bm{\xi}^r) \chi_{j}(\bm{\xi}^r) d\bm{\Omega}_r 
    =\int_{\bm{\Gamma}_r} \chi_{i}(\bm{\xi}^r) \chi_{j}(\bm{\xi}^r)\hat{\bm{n}}^r d\bm{\Gamma}_r \\
    \Leftrightarrow
    \bm{\chi}(\bm{\xi}_v^r)^T\bm{W}\nabla^r\bm{\chi}(\bm{\xi}_v^r)
    +\nabla^r\bm{\chi}(\bm{\xi}_v^r)^T\bm{W}\bm{\chi}(\bm{\xi}_v^r)
    =\sum_{f=1}^{N_f}\sum_{k=1}^{N_{fp}} \bm{\chi}(\bm{\xi}_{f,k}^r)^T W_{f,k}\hat{\bm{n}}^r\bm{\chi}(\bm{\xi}_{f,k}^r).
\end{split}
\end{equation}

\subsection{SBP - Strong Form FR Split}\label{sec:SBP deriv}

We introduce the lifting operator,

\begin{equation}
    \bm{L}_q= \bm{M}^{-1}\sum_{f=1}^{N_f}\bm{\chi}(\bm{\xi}_{f}^r)^T \bm{W}_f,
\end{equation}

\noindent where $\bm{\chi}(\bm{\xi}_{f}^r)$ stores the basis functions evaluated at all facet cubature nodes on the face $f$, and $\bm{W}_f$ is a diagonal matrix storing the quadrature weights on the face $f$.

We now introduce the SBP operator~\cite{chan2019skew},

\begin{equation}
    \bm{Q}^i = \bm{W}(\bm{M}^{-1}\bm{S}_\xi)\bm{\Pi},\nonumber
\end{equation}

\noindent to formulate the \textit{skew-hybridized} SBP operator from Chan~\cite[Eq. (10)]{chan2019skew},

\begin{equation}
    \Tilde{\bm{Q}}_p^i = \frac{1}{2}\begin{bmatrix}
    \bm{Q}^i-(\bm{Q}^i)^T & \bm{W}\bm{\chi}(\bm{\xi}_v^r)\bm{L}_q\diag(\hat{\bm{n}}_f^\xi)\\
    -\sum_{f=1}^{N_f}\bm{W}_{f}\diag(\hat{\bm{n}}_f^\xi)\bm{\chi}(\bm{\xi}_{f}^r)\bm{\Pi}
    & \sum_{f=1}^{N_f}\bm{W}_{f}\diag(\hat{\bm{n}}_f^\xi)
    \end{bmatrix}.
\end{equation}

Next, similar to Chan~\cite[Eq. (27)]{chan2019skew}, we introduce the metric dependent hybridized SBP operator as,

\begin{equation}
    \bm{Q}_m^i = \frac{1}{2}\sum_{j=1}^{d}\Big(
    \diag{\begin{bmatrix}\bm{C}_m(\bm{\xi}_v^r)_{ji}\\\bm{C}_m(\bm{\xi}_{f}^r)_{ji}
    \end{bmatrix}} \Tilde{\bm{Q}}_p^j 
    +\Tilde{\bm{Q}}_p^j\diag{\begin{bmatrix}\bm{C}_m(\bm{\xi}_v^r)_{ji}\\\bm{C}_m(\bm{\xi}_{f}^r)_{ji}
    \end{bmatrix}}
    \Big).
\end{equation}

In equivalent form, we express Eq.~(\ref{eq: ESFR equiv dudt}) as,

\begin{equation}
\begin{split}
    \frac{d}{dt}\hat{\bm{u}}_m(t)^T + \Big[(\bm{M}_m+\bm{K}_m)^{-1}\bm{\chi}(\bm{\xi}_v^r)^T\:\: (\bm{M}_m+\bm{K}_m)^{-1}\sum_{f=1}^{N_f} \bm{\chi}(\bm{\xi}_{f}^r)^T\Big] \sum_{j=1}^{d} \Big( 2\bm{Q}_m^i \circ \bm{F}^j_{S} \Big)\bm{1}^T\\ +\sum_{j=1}^{d}(\bm{M}_m+\bm{K}_m)^{-1}\sum_{f=1}^{N_f} \diag{(\bm{n}_{m,j})}  (\bm{f}_j^* - \bm{f}_j(\tilde{\bm{u}_m})) = \bm{0}^T\\
    (\bm{F}^i_{S})_{jk} = \bm{f}_S^i(\Tilde{\bm{u}}_j,\Tilde{\bm{u}}_k),\: \forall 1\leq j+k\leq N_{vp}+ N_{fp},
 \end{split}
\end{equation}

\noindent where $\bm{n}_m = \hat{\bm{n}}^r\bm{C}_m^T$.

We would like to emphasize that incorporating the ESFR filter on the volume terms does not create a new ESFR differential operator, but instead is a modification on the norm that the DG volume is projected on. That is, we project both the volume and the surface terms to the $p$-th order broken Sobolev-space in the nonlinearly stable FR scheme, whereas in DG, the volume and surfaces are projected onto the {\color{black}$L_2$}-space.

\section{Stability Proof - Operator Form}\label{sec: oper stab proof}

Here we present the stability proof from Sec.~\ref{sec:Stability} in operator form. We start by applying the $(\bm{M}_m+\bm{K}_m)$-norm, and we quickly see that it cancels off with its respective inverse,

\begin{equation}
    \begin{split}
      &  \hat{\bm{u}}_m(t)\Big(\bm{M}_m+\bm{K}_m\Big)\frac{d}{d t}\hat{\bm{u}}_m(t)^T
        \\ 
        &=-\hat{\bm{u}}_m(t)\Big( \bm{M}_m+\bm{K}_m\Big)\Big( \bm{M}_m+\bm{K}_m\Big)^{-1}
        \sum_{i=1}^{d}\sum_{j=1}^{d}   \bm{\chi}(\bm{\xi}_v^r)^T\bm{W} 
        \Big( a_i \frac{ \partial \bm{\chi}(\bm{\xi}_v^r)}{\partial \xi_j} \bm{\Pi}  (J_m^\Omega\frac{\partial \xi_j}{\partial x_i})\bm{\chi}(\bm{\xi}_v^r)\hat{\bm{u}}_m(t)^T 
        + a_i (J_m^\Omega\frac{\partial \xi_j}{\partial x_i}) \frac{ \partial \bm{\chi}(\bm{\xi}_v^r)}{\partial \xi_j} \bm{\Pi}\bm{\chi}(\bm{\xi}_v^r)\hat{\bm{u}}_m(t)^T
        \Big)\\
    &  -\hat{\bm{u}}_m(t)\Big( \bm{M}_m+\bm{K}_m\Big)\Big( \bm{M}_m+\bm{K}_m\Big)^{-1}
        \sum_{f=1}^{N_f} \bm{\chi}(\bm{\xi}_{f}^r)^T \bm{W}_f \diag(\hat{\bm{n}}_f^r) \bm{f}_m^{{C,r}^T}.
    \end{split}
\end{equation}

\noindent Next, consider the volume terms with respect to a single $(i,j)$-pairing, substitute $\frac{\partial \bm{\chi}(\bm{\xi}_v^r)}{\partial \xi_j} = \bm{\chi}(\bm{\xi}_v^r)\bm{M}^{-1}\bm{S}_{\xi,j}$, and swap the metric terms with the quadrature weights in the second volume term,

\begin{equation}
\begin{split}
    \hat{\bm{u}}_m(t) \bm{\chi}(\bm{\xi}_v^r)^T\bm{W} 
        \Big( a_i \frac{ \partial \bm{\chi}(\bm{\xi}_v^r)}{\partial \xi_j} \bm{\Pi}  (J_m^\Omega\frac{\partial \xi_j}{\partial x_i})\bm{\chi}(\bm{\xi}_v^r)\hat{\bm{u}}_m(t)^T 
        + a_i (J_m^\Omega\frac{\partial \xi_j}{\partial x_i}) \frac{ \partial \bm{\chi}(\bm{\xi}_v^r)}{\partial \xi_j} \bm{\Pi}\bm{\chi}(\bm{\xi}_v^r)\hat{\bm{u}}_m(t)^T
        \Big)\\
        =a_i\hat{\bm{u}}_m(t) \bm{\chi}(\bm{\xi}_v^r)^T\bm{W} 
          \bm{\chi}(\bm{\xi}_v^r)\bm{M}^{-1}\bm{S}_{\xi,j} \bm{\Pi}  (J_m^\Omega\frac{\partial \xi_j}{\partial x_i})\bm{\chi}(\bm{\xi}_v^r)\hat{\bm{u}}_m(t)^T 
        + a_i\hat{\bm{u}}_m(t) \bm{\chi}(\bm{\xi}_v^r)^T  (J_m^\Omega\frac{\partial \xi_j}{\partial x_i}) \bm{W}\bm{\chi}(\bm{\xi}_v^r)\bm{M}^{-1}\bm{S}_{\xi,j} \bm{\Pi}\bm{\chi}(\bm{\xi}_v^r)\hat{\bm{u}}_m(t)^T
        .
        \end{split}
\end{equation}

\noindent We continue by substituting $\bm{\Pi}^T = \bm{W}\bm{\chi}(\bm{\xi}_v^r)\bm{M}^{-1}, \text{ and } \bm{\Pi}\bm{\chi}(\bm{\xi}_v^r) = \bm{M}^{-1}\bm{M}=\bm{I}$,

\begin{equation}
    \begin{split}
      &  =a_i\hat{\bm{u}}_m(t)\bm{\chi}(\bm{\xi}_v^r)^T\bm{\Pi}^T\bm{S}_{\xi,j} \bm{\Pi} (J_m^\Omega\frac{\partial \xi_j}{\partial x_i}) \bm{\chi}(\bm{\xi}_v^r)\hat{\bm{u}}_m(t)^T
        +a_i\hat{\bm{u}}_m(t)\bm{\chi}(\bm{\xi}_v^r)^T(J_m^\Omega\frac{\partial \xi_j}{\partial x_i}) \bm{\Pi}^T\bm{S}_{\xi,j} \hat{\bm{u}}_m(t)^T\\
      &  =a_i\Big(\bm{\Pi}\bm{\chi}(\bm{\xi}_v^r)\hat{\bm{u}}_m(t)^T\Big)^T
      \bm{S}_{\xi,j} \bm{\Pi} (J_m^\Omega\frac{\partial \xi_j}{\partial x_i}) \bm{\chi}(\bm{\xi}_v^r)\hat{\bm{u}}_m(t)^T
        +a_i\hat{\bm{u}}_m(t)\bm{\chi}(\bm{\xi}_v^r)^T(J_m^\Omega\frac{\partial \xi_j}{\partial x_i}) \bm{\Pi}^T\bm{S}_{\xi,j} \hat{\bm{u}}_m(t)^T.
    \end{split}
\end{equation}

\noindent Lastly, we substitute $\bm{\Pi}\bm{\chi}(\bm{\xi}_v^r) = \bm{M}^{-1}\bm{M}=\bm{I}$ once more and then perform integration-by-parts on the first stiffness matrix to arrive at,

\begin{equation}
\begin{split}
   &= -a_i\hat{\bm{u}}_m(t)
      \bm{S}_{\xi,j}^T \bm{\Pi} (J_m^\Omega\frac{\partial \xi_j}{\partial x_i}) \bm{\chi}(\bm{\xi}_v^r)\hat{\bm{u}}_m(t)^T
        +a_i\hat{\bm{u}}_m(t)\bm{\chi}(\bm{\xi}_v^r)^T(J_m^\Omega\frac{\partial \xi_j}{\partial x_i}) \bm{\Pi}^T\bm{S}_{\xi,j} \hat{\bm{u}}_m(t)^T
        +\sum_{f=1}^{N_f} a_i\hat{\bm{u}}_m(t)\bm{\chi}(\bm{\xi}_{f}^r)^T \diag(\hat{\bm{n}}^{\xi_j})\bm{\chi}(\bm{\xi}_{f})\hat{\bm{u}}_m(t)^T.
\end{split}
\end{equation}

\noindent The two volume terms are the transpose of each other, thus they cancel out and the resultant stability claim is the same as Eq.~(\ref{eq: stab final cond}) in Sec.~\ref{sec:Stability}. $\qed$

\bibliographystyle{model1-num-names}
\bibliography{bibliographie2}

\begin{thebibliography}{81}
\expandafter\ifx\csname natexlab\endcsname\relax\def\natexlab#1{#1}\fi
\providecommand{\url}[1]{\texttt{#1}}
\providecommand{\href}[2]{#2}
\providecommand{\path}[1]{#1}
\providecommand{\DOIprefix}{doi:}
\providecommand{\ArXivprefix}{arXiv:}
\providecommand{\URLprefix}{URL: }
\providecommand{\Pubmedprefix}{pmid:}
\providecommand{\doi}[1]{\href{http://dx.doi.org/#1}{\path{#1}}}
\providecommand{\Pubmed}[1]{\href{pmid:#1}{\path{#1}}}
\providecommand{\bibinfo}[2]{#2}
\ifx\xfnm\relax \def\xfnm[#1]{\unskip,\space#1}\fi
\bibitem[{Huynh(2007)}]{huynh_flux_2007}
\bibinfo{author}{H.~T. Huynh},
\newblock \bibinfo{title}{A {Flux} {Reconstruction} {Approach} to
  {High}-{Order} {Schemes} {Including} {Discontinuous} {Galerkin} {Methods}},
\newblock \bibinfo{publisher}{American Institute of Aeronautics and
  Astronautics}, \bibinfo{year}{2007}. \DOIprefix\doi{10.2514/6.2007-4079}.
\bibitem[{Wang and Gao(2009)}]{wang2009unifying}
\bibinfo{author}{Z.~J. Wang}, \bibinfo{author}{H.~Gao},
\newblock \bibinfo{title}{A unifying lifting collocation penalty formulation
  including the discontinuous {Galerkin}, spectral volume/difference methods
  for conservation laws on mixed grids},
\newblock \bibinfo{journal}{Journal of Computational Physics}
  \bibinfo{volume}{228} (\bibinfo{year}{2009}) \bibinfo{pages}{8161--8186}.
\bibitem[{Huynh et~al.(2014)Huynh, Wang, and Vincent}]{huynh2014high}
\bibinfo{author}{H.~Huynh}, \bibinfo{author}{Z.~J. Wang},
  \bibinfo{author}{P.~E. Vincent},
\newblock \bibinfo{title}{High-order methods for computational fluid dynamics:
  A brief review of compact differential formulations on unstructured grids},
\newblock \bibinfo{journal}{Computers \& fluids} \bibinfo{volume}{98}
  (\bibinfo{year}{2014}) \bibinfo{pages}{209--220}.
\bibitem[{Vincent et~al.(2011)Vincent, Castonguay, and
  Jameson}]{vincent_insights_2011}
\bibinfo{author}{P.~Vincent}, \bibinfo{author}{P.~Castonguay},
  \bibinfo{author}{A.~Jameson},
\newblock \bibinfo{title}{Insights from von {Neumann} analysis of
  {high}-{order} flux reconstruction schemes},
\newblock \bibinfo{journal}{Journal of Computational Physics}
  \bibinfo{volume}{230} (\bibinfo{year}{2011}) \bibinfo{pages}{8134--8154}.
\bibitem[{Williams and Jameson(2014)}]{williams_energy_2014}
\bibinfo{author}{D.~M. Williams}, \bibinfo{author}{A.~Jameson},
\newblock \bibinfo{title}{Energy {Stable} {Flux} {Reconstruction} {Schemes} for
  {Advection }–{Diffusion} {Problems} on {Tetrahedra}},
\newblock \bibinfo{journal}{Journal of Scientific Computing}
  \bibinfo{volume}{59} (\bibinfo{year}{2014}) \bibinfo{pages}{721--759}.
\bibitem[{Vincent et~al.(2015)Vincent, Farrington, Witherden, and
  Jameson}]{vincent_extended_2015}
\bibinfo{author}{P.~Vincent}, \bibinfo{author}{A.~Farrington},
  \bibinfo{author}{F.~Witherden}, \bibinfo{author}{A.~Jameson},
\newblock \bibinfo{title}{An extended range of stable-symmetric-conservative
  {Flux} {Reconstruction} correction functions},
\newblock \bibinfo{journal}{Computer Methods in Applied Mechanics and
  Engineering} \bibinfo{volume}{296} (\bibinfo{year}{2015})
  \bibinfo{pages}{248--272}.
\bibitem[{Castonguay(2012)}]{castonguay_phd}
\bibinfo{author}{P.~Castonguay},
\newblock \bibinfo{title}{High-order energy stable flux reconstruction schemes
  for fluid flow simulations on unstructured grids},
\newblock \bibinfo{journal}{Diss. Stanford University}  (\bibinfo{year}{2012}).
\bibitem[{Castonguay et~al.(2012)Castonguay, Vincent, and
  Jameson}]{castonguay_new_2012}
\bibinfo{author}{P.~Castonguay}, \bibinfo{author}{P.~E. Vincent},
  \bibinfo{author}{A.~Jameson},
\newblock \bibinfo{title}{A {New} {Class} of {High}-{Order} {Energy} {Stable}
  {Flux} {Reconstruction} {Schemes} for {Triangular} {Elements}},
\newblock \bibinfo{journal}{Journal of Scientific Computing}
  \bibinfo{volume}{51} (\bibinfo{year}{2012}) \bibinfo{pages}{224--256}.
\bibitem[{{Del Rey Fern\'andez} et~al.(2014){Del Rey Fern\'andez}, Hicken, and
  Zingg}]{fernandez2014review}
\bibinfo{author}{D.~C. {Del Rey Fern\'andez}}, \bibinfo{author}{J.~E. Hicken},
  \bibinfo{author}{D.~W. Zingg},
\newblock \bibinfo{title}{Review of summation-by-parts operators with
  simultaneous approximation terms for the numerical solution of partial
  differential equations},
\newblock \bibinfo{journal}{Computers \& Fluids} \bibinfo{volume}{95}
  (\bibinfo{year}{2014}) \bibinfo{pages}{171--196}.
\bibitem[{Sv{\"a}rd and Nordstr{\"o}m(2014)}]{svard2014review}
\bibinfo{author}{M.~Sv{\"a}rd}, \bibinfo{author}{J.~Nordstr{\"o}m},
\newblock \bibinfo{title}{Review of summation-by-parts schemes for
  initial--boundary-value problems},
\newblock \bibinfo{journal}{Journal of Computational Physics}
  \bibinfo{volume}{268} (\bibinfo{year}{2014}) \bibinfo{pages}{17--38}.
\bibitem[{{Del Rey Fern\'andez} et~al.(2014){Del Rey Fern\'andez}, Boom, and
  Zingg}]{fernandez2014generalized}
\bibinfo{author}{D.~C. {Del Rey Fern\'andez}}, \bibinfo{author}{P.~D. Boom},
  \bibinfo{author}{D.~W. Zingg},
\newblock \bibinfo{title}{A generalized framework for nodal first derivative
  summation-by-parts operators},
\newblock \bibinfo{journal}{Journal of Computational Physics}
  \bibinfo{volume}{266} (\bibinfo{year}{2014}) \bibinfo{pages}{214--239}.
\bibitem[{Carpenter et~al.(1994)Carpenter, Gottlieb, and
  Abarbanel}]{carpenter1994time}
\bibinfo{author}{M.~H. Carpenter}, \bibinfo{author}{D.~Gottlieb},
  \bibinfo{author}{S.~Abarbanel},
\newblock \bibinfo{title}{Time-stable boundary conditions for finite-difference
  schemes solving hyperbolic systems: methodology and application to high-order
  compact schemes},
\newblock \bibinfo{journal}{Journal of Computational Physics}
  \bibinfo{volume}{111} (\bibinfo{year}{1994}) \bibinfo{pages}{220--236}.
\bibitem[{Carpenter et~al.(1999)Carpenter, Nordstr{\"o}m, and
  Gottlieb}]{carpenter1999stable}
\bibinfo{author}{M.~H. Carpenter}, \bibinfo{author}{J.~Nordstr{\"o}m},
  \bibinfo{author}{D.~Gottlieb},
\newblock \bibinfo{title}{A stable and conservative interface treatment of
  arbitrary spatial accuracy},
\newblock \bibinfo{journal}{Journal of Computational Physics}
  \bibinfo{volume}{148} (\bibinfo{year}{1999}) \bibinfo{pages}{341--365}.
\bibitem[{Nordstr{\"o}m and Carpenter(1999)}]{nordstrom1999boundary}
\bibinfo{author}{J.~Nordstr{\"o}m}, \bibinfo{author}{M.~H. Carpenter},
\newblock \bibinfo{title}{Boundary and interface conditions for high-order
  finite-difference methods applied to the {Euler} and {Navier}--{Stokes}
  equations},
\newblock \bibinfo{journal}{Journal of Computational Physics}
  \bibinfo{volume}{148} (\bibinfo{year}{1999}) \bibinfo{pages}{621--645}.
\bibitem[{Nordstr{\"o}m and Carpenter(2001)}]{nordstrom2001high}
\bibinfo{author}{J.~Nordstr{\"o}m}, \bibinfo{author}{M.~H. Carpenter},
\newblock \bibinfo{title}{High-order finite difference methods,
  multidimensional linear problems, and curvilinear coordinates},
\newblock \bibinfo{journal}{Journal of Computational Physics}
  \bibinfo{volume}{173} (\bibinfo{year}{2001}) \bibinfo{pages}{149--174}.
\bibitem[{Carpenter et~al.(2010)Carpenter, Nordstr{\"o}m, and
  Gottlieb}]{carpenter2010revisiting}
\bibinfo{author}{M.~H. Carpenter}, \bibinfo{author}{J.~Nordstr{\"o}m},
  \bibinfo{author}{D.~Gottlieb},
\newblock \bibinfo{title}{Revisiting and extending interface penalties for
  multi-domain summation-by-parts operators},
\newblock \bibinfo{journal}{Journal of Scientific Computing}
  \bibinfo{volume}{45} (\bibinfo{year}{2010}) \bibinfo{pages}{118--150}.
\bibitem[{Sv{\"a}rd and {\"O}zcan(2014)}]{svard2014entropy}
\bibinfo{author}{M.~Sv{\"a}rd}, \bibinfo{author}{H.~{\"O}zcan},
\newblock \bibinfo{title}{Entropy-stable schemes for the euler equations with
  far-field and wall boundary conditions},
\newblock \bibinfo{journal}{Journal of Scientific Computing}
  \bibinfo{volume}{58} (\bibinfo{year}{2014}) \bibinfo{pages}{61--89}.
\bibitem[{Parsani et~al.(2015{\natexlab{a}})Parsani, Carpenter, and
  Nielsen}]{parsani2015entropy}
\bibinfo{author}{M.~Parsani}, \bibinfo{author}{M.~H. Carpenter},
  \bibinfo{author}{E.~J. Nielsen},
\newblock \bibinfo{title}{Entropy stable discontinuous interfaces coupling for
  the three-dimensional compressible {Navier}-{Stokes} equations.},
\newblock \bibinfo{journal}{J. Comput. Phys.} \bibinfo{volume}{290}
  (\bibinfo{year}{2015}{\natexlab{a}}) \bibinfo{pages}{132--138}.
\bibitem[{Parsani et~al.(2015{\natexlab{b}})Parsani, Carpenter, and
  Nielsen}]{parsani2015entropyWall}
\bibinfo{author}{M.~Parsani}, \bibinfo{author}{M.~H. Carpenter},
  \bibinfo{author}{E.~J. Nielsen},
\newblock \bibinfo{title}{Entropy stable wall boundary conditions for the
  three-dimensional compressible {N}avier--{S}tokes equations},
\newblock \bibinfo{journal}{Journal of Computational Physics}
  \bibinfo{volume}{292} (\bibinfo{year}{2015}{\natexlab{b}})
  \bibinfo{pages}{88--113}.
\bibitem[{Ranocha et~al.(2016)Ranocha, {\"O}ffner, and
  Sonar}]{ranocha2016summation}
\bibinfo{author}{H.~Ranocha}, \bibinfo{author}{P.~{\"O}ffner},
  \bibinfo{author}{T.~Sonar},
\newblock \bibinfo{title}{Summation-by-parts operators for correction procedure
  via reconstruction},
\newblock \bibinfo{journal}{Journal of Computational Physics}
  \bibinfo{volume}{311} (\bibinfo{year}{2016}) \bibinfo{pages}{299--328}.
\bibitem[{Ranocha et~al.(2017)Ranocha, {\"O}ffner, and
  Sonar}]{ranocha2017extended}
\bibinfo{author}{H.~Ranocha}, \bibinfo{author}{P.~{\"O}ffner},
  \bibinfo{author}{T.~Sonar},
\newblock \bibinfo{title}{Extended skew-symmetric form for summation-by-parts
  operators and varying {Jacobians}},
\newblock \bibinfo{journal}{Journal of Computational Physics}
  \bibinfo{volume}{342} (\bibinfo{year}{2017}) \bibinfo{pages}{13--28}.
\bibitem[{Montoya and Zingg(2021)}]{montoya2021unifying}
\bibinfo{author}{T.~Montoya}, \bibinfo{author}{D.~W. Zingg},
\newblock \bibinfo{title}{A unifying algebraic framework for discontinuous
  galerkin and flux reconstruction methods based on the summation-by-parts
  property},
\newblock \bibinfo{journal}{arXiv preprint arXiv:2101.10478v1}
  (\bibinfo{year}{2021}).
\bibitem[{Abgrall et~al.(2019)Abgrall, {\"O}ffner, and
  Ranocha}]{abgrall2019reinterpretation}
\bibinfo{author}{R.~Abgrall}, \bibinfo{author}{P.~{\"O}ffner},
  \bibinfo{author}{H.~Ranocha},
\newblock \bibinfo{title}{Reinterpretation and extension of entropy correction
  terms for residual distribution and discontinuous galerkin schemes},
\newblock \bibinfo{journal}{arXiv preprint arXiv:1908.04556}
  (\bibinfo{year}{2019}).
\bibitem[{Abgrall(2018)}]{abgrall2018general}
\bibinfo{author}{R.~Abgrall},
\newblock \bibinfo{title}{A general framework to construct schemes satisfying
  additional conservation relations. application to entropy conservative and
  entropy dissipative schemes},
\newblock \bibinfo{journal}{Journal of Computational Physics}
  \bibinfo{volume}{372} (\bibinfo{year}{2018}) \bibinfo{pages}{640--666}.
\bibitem[{Abgrall et~al.(2018)Abgrall, Meledo, and
  Oeffner}]{abgrall2018connection}
\bibinfo{author}{R.~Abgrall}, \bibinfo{author}{E.~l. Meledo},
  \bibinfo{author}{P.~Oeffner},
\newblock \bibinfo{title}{On the connection between residual distribution
  schemes and flux reconstruction},
\newblock \bibinfo{journal}{arXiv preprint arXiv:1807.01261}
  (\bibinfo{year}{2018}).
\bibitem[{Abgrall et~al.(2021)Abgrall, Nordstr{\"o}m, {\"O}ffner, and
  Tokareva}]{abgrall2021analysis}
\bibinfo{author}{R.~Abgrall}, \bibinfo{author}{J.~Nordstr{\"o}m},
  \bibinfo{author}{P.~{\"O}ffner}, \bibinfo{author}{S.~Tokareva},
\newblock \bibinfo{title}{Analysis of the sbp-sat stabilization for finite
  element methods part ii: entropy stability},
\newblock \bibinfo{journal}{Communications on Applied Mathematics and
  Computation}  (\bibinfo{year}{2021}) \bibinfo{pages}{1--23}.
\bibitem[{Fisher(2012)}]{fisher2012high}
\bibinfo{author}{T.~C. Fisher}, \bibinfo{title}{High-order L2 stable
  multi-domain finite difference method for compressible flows}, Ph.D. thesis,
  Purdue University, \bibinfo{year}{2012}.
\bibitem[{Fisher et~al.(2013)Fisher, Carpenter, Nordstr{\"o}m, Yamaleev, and
  Swanson}]{fisher2013discretely}
\bibinfo{author}{T.~C. Fisher}, \bibinfo{author}{M.~H. Carpenter},
  \bibinfo{author}{J.~Nordstr{\"o}m}, \bibinfo{author}{N.~K. Yamaleev},
  \bibinfo{author}{C.~Swanson},
\newblock \bibinfo{title}{Discretely conservative finite-difference
  formulations for nonlinear conservation laws in split form: Theory and
  boundary conditions},
\newblock \bibinfo{journal}{Journal of Computational Physics}
  \bibinfo{volume}{234} (\bibinfo{year}{2013}) \bibinfo{pages}{353--375}.
\bibitem[{Fisher and Carpenter(2013)}]{fisher2013high}
\bibinfo{author}{T.~C. Fisher}, \bibinfo{author}{M.~H. Carpenter},
\newblock \bibinfo{title}{High-order entropy stable finite difference schemes
  for nonlinear conservation laws: Finite domains},
\newblock \bibinfo{journal}{Journal of Computational Physics}
  \bibinfo{volume}{252} (\bibinfo{year}{2013}) \bibinfo{pages}{518--557}.
\bibitem[{Carpenter et~al.(2014)Carpenter, Fisher, Nielsen, and
  Frankel}]{carpenter2014entropy}
\bibinfo{author}{M.~H. Carpenter}, \bibinfo{author}{T.~C. Fisher},
  \bibinfo{author}{E.~J. Nielsen}, \bibinfo{author}{S.~H. Frankel},
\newblock \bibinfo{title}{Entropy stable spectral collocation schemes for the
  {Navier}--{Stokes} equations: Discontinuous interfaces},
\newblock \bibinfo{journal}{SIAM Journal on Scientific Computing}
  \bibinfo{volume}{36} (\bibinfo{year}{2014}) \bibinfo{pages}{B835--B867}.
\bibitem[{Parsani et~al.(2016)Parsani, Carpenter, Fisher, and
  Nielsen}]{parsani2016entropy}
\bibinfo{author}{M.~Parsani}, \bibinfo{author}{M.~H. Carpenter},
  \bibinfo{author}{T.~C. Fisher}, \bibinfo{author}{E.~J. Nielsen},
\newblock \bibinfo{title}{Entropy stable staggered grid discontinuous spectral
  collocation methods of any order for the compressible {N}avier--{S}tokes
  equations},
\newblock \bibinfo{journal}{SIAM Journal on Scientific Computing}
  \bibinfo{volume}{38} (\bibinfo{year}{2016}) \bibinfo{pages}{A3129--A3162}.
\bibitem[{Carpenter et~al.(2016)Carpenter, Parsani, Nielsen, and
  Fisher}]{carpenter2016towards}
\bibinfo{author}{M.~H. Carpenter}, \bibinfo{author}{M.~Parsani},
  \bibinfo{author}{E.~J. Nielsen}, \bibinfo{author}{T.~C. Fisher},
\newblock \bibinfo{title}{Towards an entropy stable spectral element framework
  for computational fluid dynamics},
\newblock in: \bibinfo{booktitle}{54th AIAA Aerospace Sciences Meeting},
  \bibinfo{year}{2016}, p. \bibinfo{pages}{1058}.
\bibitem[{Yamaleev and Carpenter(2017)}]{yamaleev2017family}
\bibinfo{author}{N.~K. Yamaleev}, \bibinfo{author}{M.~H. Carpenter},
\newblock \bibinfo{title}{A family of fourth-order entropy stable
  nonoscillatory spectral collocation schemes for the 1-d {Navier}--{S}tokes
  equations},
\newblock \bibinfo{journal}{Journal of Computational Physics}
  \bibinfo{volume}{331} (\bibinfo{year}{2017}) \bibinfo{pages}{90--107}.
\bibitem[{Crean et~al.(2018)Crean, Hicken, {Del Rey Fern\'andez}, Zingg, and
  Carpenter}]{crean2018entropy}
\bibinfo{author}{J.~Crean}, \bibinfo{author}{J.~E. Hicken},
  \bibinfo{author}{D.~C. {Del Rey Fern\'andez}}, \bibinfo{author}{D.~W. Zingg},
  \bibinfo{author}{M.~H. Carpenter},
\newblock \bibinfo{title}{Entropy-stable summation-by-parts discretization of
  the {Euler} equations on general curved elements},
\newblock \bibinfo{journal}{Journal of Computational Physics}
  \bibinfo{volume}{356} (\bibinfo{year}{2018}) \bibinfo{pages}{410--438}.
\bibitem[{Chen and Shu(2017)}]{chen2017entropy}
\bibinfo{author}{T.~Chen}, \bibinfo{author}{C.-W. Shu},
\newblock \bibinfo{title}{Entropy stable high order discontinuous {Galerkin}
  methods with suitable quadrature rules for hyperbolic conservation laws},
\newblock \bibinfo{journal}{Journal of Computational Physics}
  \bibinfo{volume}{345} (\bibinfo{year}{2017}) \bibinfo{pages}{427--461}.
\bibitem[{Crean et~al.(2020)Crean, {Del Rey Fern\'andez}, Carpenter, and
  Hicken}]{Crean2019Staggered}
\bibinfo{author}{J.~Crean}, \bibinfo{author}{D.~C. {Del Rey Fern\'andez}},
  \bibinfo{author}{M.~H. Carpenter}, \bibinfo{author}{J.~E. Hicken},
\newblock \bibinfo{title}{Staggered entropy-stable summation-by- parts
  discretization of the {Euler} equations on general curved elements},
\newblock \bibinfo{journal}{accepted in Journal of Computational Physics}
  (\bibinfo{year}{2020}).
\bibitem[{Chan(2018)}]{chan2018discretely}
\bibinfo{author}{J.~Chan},
\newblock \bibinfo{title}{On discretely entropy conservative and entropy stable
  discontinuous {Galerkin} methods},
\newblock \bibinfo{journal}{Journal of Computational Physics}
  \bibinfo{volume}{362} (\bibinfo{year}{2018}) \bibinfo{pages}{346--374}.
\bibitem[{Friedrich et~al.(2020)Friedrich, Shnücke, Winters, {Del Rey
  Fern\'andez}, Gassner, and Carpenter}]{FriedrichEntropy2020}
\bibinfo{author}{L.~Friedrich}, \bibinfo{author}{G.~Shnücke},
  \bibinfo{author}{A.~R. Winters}, \bibinfo{author}{D.~C. {Del Rey
  Fern\'andez}}, \bibinfo{author}{G.~J. Gassner}, \bibinfo{author}{M.~H.
  Carpenter},
\newblock \bibinfo{title}{Entropy stable space-time discontinuous {Galerkin}
  schemes with summation-by-parts property for hyperbolic conservation},
\newblock \bibinfo{journal}{(Submitted to the Journal of Scientific Computing)}
   (\bibinfo{year}{2020}).
\bibitem[{Sv{\"a}rd(2004)}]{svard2004coordinate}
\bibinfo{author}{M.~Sv{\"a}rd},
\newblock \bibinfo{title}{On coordinate transformations for summation-by-parts
  operators},
\newblock \bibinfo{journal}{Journal of Scientific Computing}
  \bibinfo{volume}{20} (\bibinfo{year}{2004}) \bibinfo{pages}{29--42}.
\bibitem[{Fern{\'a}ndez et~al.(2019)Fern{\'a}ndez, Boom, Carpenter, and
  Zingg}]{Fernandez2019curvitensor}
\bibinfo{author}{D.~C. D.~R. Fern{\'a}ndez}, \bibinfo{author}{P.~D. Boom},
  \bibinfo{author}{M.~H. Carpenter}, \bibinfo{author}{D.~W. Zingg},
\newblock \bibinfo{title}{Extension of tensor-product general- ized and
  dense-norm summation-by-parts operators to curvilinear coordinates},
\newblock \bibinfo{journal}{Journal of Scientific Computing}
  \bibinfo{volume}{80} (\bibinfo{year}{2019}) \bibinfo{pages}{1957--1996}.
\bibitem[{{\AA}lund and Nordstr{\"o}m(2019)}]{aalund2019encapsulated}
\bibinfo{author}{O.~{\AA}lund}, \bibinfo{author}{J.~Nordstr{\"o}m},
\newblock \bibinfo{title}{Encapsulated high order difference operators on
  curvilinear non-conforming grids},
\newblock \bibinfo{journal}{Journal of Computational Physics}
  \bibinfo{volume}{385} (\bibinfo{year}{2019}) \bibinfo{pages}{209--224}.
\bibitem[{Wintermeyer et~al.(2017)Wintermeyer, Winters, Gassner, and
  Kopriva}]{wintermeyer2017entropy}
\bibinfo{author}{N.~Wintermeyer}, \bibinfo{author}{A.~R. Winters},
  \bibinfo{author}{G.~J. Gassner}, \bibinfo{author}{D.~A. Kopriva},
\newblock \bibinfo{title}{An entropy stable nodal discontinuous {Galerkin}
  method for the two dimensional shallow water equations on unstructured
  curvilinear meshes with discontinuous bathymetry},
\newblock \bibinfo{journal}{Journal of Computational Physics}
  \bibinfo{volume}{340} (\bibinfo{year}{2017}) \bibinfo{pages}{200--242}.
\bibitem[{Moxey et~al.(2019)Moxey, Sastry, and Kirby}]{moxey2019interpolation}
\bibinfo{author}{D.~Moxey}, \bibinfo{author}{S.~P. Sastry},
  \bibinfo{author}{R.~M. Kirby},
\newblock \bibinfo{title}{Interpolation error bounds for curvilinear finite
  elements and their implications on adaptive mesh refinement},
\newblock \bibinfo{journal}{Journal of Scientific Computing}
  \bibinfo{volume}{78} (\bibinfo{year}{2019}) \bibinfo{pages}{1045--1062}.
\bibitem[{Mengaldo et~al.(2016)Mengaldo, De~Grazia, Vincent, and
  Sherwin}]{mengaldo2016connections}
\bibinfo{author}{G.~Mengaldo}, \bibinfo{author}{D.~De~Grazia},
  \bibinfo{author}{P.~E. Vincent}, \bibinfo{author}{S.~J. Sherwin},
\newblock \bibinfo{title}{On the connections between discontinuous {Galerkin}
  and flux reconstruction schemes: extension to curvilinear meshes},
\newblock \bibinfo{journal}{Journal of Scientific Computing}
  \bibinfo{volume}{67} (\bibinfo{year}{2016}) \bibinfo{pages}{1272--1292}.
\bibitem[{Jan S.~Hesthaven(2008)}]{NDG}
\bibinfo{author}{T.~W. Jan S.~Hesthaven}, \bibinfo{title}{Nodal Discontinuous
  Galerkin methods Algorithms, Analysis and Applications},
  \bibinfo{publisher}{Springer}, \bibinfo{year}{2008}.
  \DOIprefix\doi{10.1007/978-0-387-72067-8}.
\bibitem[{Karniadakis and Sherwin(2013)}]{karniadakis2013spectral}
\bibinfo{author}{G.~Karniadakis}, \bibinfo{author}{S.~Sherwin},
  \bibinfo{title}{Spectral/hp element methods for computational fluid
  dynamics}, \bibinfo{publisher}{Oxford University Press},
  \bibinfo{year}{2013}.
\bibitem[{Zwanenburg and Nadarajah(2016)}]{zwanenburg_equivalence_2016}
\bibinfo{author}{P.~Zwanenburg}, \bibinfo{author}{S.~Nadarajah},
\newblock \bibinfo{title}{Equivalence between the {Energy} {Stable} {Flux}
  {Reconstruction} and {Filtered} {Discontinuous} {Galerkin} {Schemes}},
\newblock \bibinfo{journal}{Journal of Computational Physics}
  \bibinfo{volume}{306} (\bibinfo{year}{2016}) \bibinfo{pages}{343--369}.
\bibitem[{Teukolsky(2016)}]{teukolsky2016formulation}
\bibinfo{author}{S.~A. Teukolsky},
\newblock \bibinfo{title}{Formulation of discontinuous {Galerkin} methods for
  relativistic astrophysics},
\newblock \bibinfo{journal}{Journal of Computational Physics}
  \bibinfo{volume}{312} (\bibinfo{year}{2016}) \bibinfo{pages}{333--356}.
\bibitem[{Chan and Wilcox(2019)}]{chan2019discretely}
\bibinfo{author}{J.~Chan}, \bibinfo{author}{L.~C. Wilcox},
\newblock \bibinfo{title}{On discretely entropy stable weight-adjusted
  discontinuous {Galerkin} methods: curvilinear meshes},
\newblock \bibinfo{journal}{Journal of Computational Physics}
  \bibinfo{volume}{378} (\bibinfo{year}{2019}) \bibinfo{pages}{366--393}.
\bibitem[{Vincent et~al.(2011)Vincent, Castonguay, and
  Jameson}]{vincent_new_2011}
\bibinfo{author}{P.~E. Vincent}, \bibinfo{author}{P.~Castonguay},
  \bibinfo{author}{A.~Jameson},
\newblock \bibinfo{title}{A {New} {Class} of {High}-{Order} {Energy} {Stable}
  {Flux} {Reconstruction} {Schemes}},
\newblock \bibinfo{journal}{Journal of Scientific Computing}
  \bibinfo{volume}{47} (\bibinfo{year}{2011}) \bibinfo{pages}{50--72}.
\bibitem[{Castonguay et~al.(2012)Castonguay, Vincent, and
  Jameson}]{castonguay2012newTRI}
\bibinfo{author}{P.~Castonguay}, \bibinfo{author}{P.~E. Vincent},
  \bibinfo{author}{A.~Jameson},
\newblock \bibinfo{title}{A new class of high-order energy stable flux
  reconstruction schemes for triangular elements},
\newblock \bibinfo{journal}{Journal of Scientific Computing}
  \bibinfo{volume}{51} (\bibinfo{year}{2012}) \bibinfo{pages}{224--256}.
\bibitem[{Cicchino and Nadarajah(2020)}]{Cicchino2020NewNorm}
\bibinfo{author}{A.~Cicchino}, \bibinfo{author}{S.~Nadarajah},
\newblock \bibinfo{title}{A new norm and stability condition for tensor product
  flux reconstruction schemes},
\newblock \bibinfo{journal}{Journal of Computational Physics}
  (\bibinfo{year}{2020}) \bibinfo{pages}{110025}.
\bibitem[{Abe et~al.(2018)Abe, Morinaka, Haga, Nonomura, Shibata, and
  Miyaji}]{abe2018stable}
\bibinfo{author}{Y.~Abe}, \bibinfo{author}{I.~Morinaka},
  \bibinfo{author}{T.~Haga}, \bibinfo{author}{T.~Nonomura},
  \bibinfo{author}{H.~Shibata}, \bibinfo{author}{K.~Miyaji},
\newblock \bibinfo{title}{Stable, non-dissipative, and conservative
  flux-reconstruction schemes in split forms},
\newblock \bibinfo{journal}{Journal of Computational Physics}
  \bibinfo{volume}{353} (\bibinfo{year}{2018}) \bibinfo{pages}{193--227}.
\bibitem[{De~Grazia et~al.(2014)De~Grazia, Mengaldo, Moxey, Vincent, and
  Sherwin}]{de2014connections}
\bibinfo{author}{D.~De~Grazia}, \bibinfo{author}{G.~Mengaldo},
  \bibinfo{author}{D.~Moxey}, \bibinfo{author}{P.~Vincent},
  \bibinfo{author}{S.~Sherwin},
\newblock \bibinfo{title}{Connections between the discontinuous {Galerkin}
  method and high-order flux reconstruction schemes},
\newblock \bibinfo{journal}{International journal for numerical methods in
  fluids} \bibinfo{volume}{75} (\bibinfo{year}{2014})
  \bibinfo{pages}{860--877}.
\bibitem[{Gassner et~al.(2016)Gassner, Winters, and Kopriva}]{gassner2016split}
\bibinfo{author}{G.~J. Gassner}, \bibinfo{author}{A.~R. Winters},
  \bibinfo{author}{D.~A. Kopriva},
\newblock \bibinfo{title}{Split form nodal discontinuous galerkin schemes with
  summation-by-parts property for the compressible euler equations},
\newblock \bibinfo{journal}{Journal of Computational Physics}
  \bibinfo{volume}{327} (\bibinfo{year}{2016}) \bibinfo{pages}{39--66}.
\bibitem[{Cicchino et~al.(2021)Cicchino, Nadarajah, and {Del Rey
  Fern\'andez}}]{CicchinoNonlinearlyStableFluxReconstruction2021}
\bibinfo{author}{A.~Cicchino}, \bibinfo{author}{S.~Nadarajah},
  \bibinfo{author}{D.~{Del Rey Fern\'andez}},
\newblock \bibinfo{title}{Nonlinearly stable flux reconstruction high-order
  methods in split form},
\newblock \bibinfo{journal}{(Submitted to the Journal of Computational
  Physics)}  (\bibinfo{year}{2021}).
\bibitem[{Jameson et~al.(2012)Jameson, Vincent, and
  Castonguay}]{jameson_non-linear_2012}
\bibinfo{author}{A.~Jameson}, \bibinfo{author}{P.~E. Vincent},
  \bibinfo{author}{P.~Castonguay},
\newblock \bibinfo{title}{On the {Non}-linear {Stability} of {Flux}
  {Reconstruction} {Schemes}},
\newblock \bibinfo{journal}{Journal of Scientific Computing}
  \bibinfo{volume}{50} (\bibinfo{year}{2012}) \bibinfo{pages}{434--445}.
\bibitem[{Castonguay et~al.(2013)Castonguay, Williams, Vincent, and
  Jameson}]{castonguay_energy_2013}
\bibinfo{author}{P.~Castonguay}, \bibinfo{author}{D.~M. Williams},
  \bibinfo{author}{P.~Vincent}, \bibinfo{author}{A.~Jameson},
\newblock \bibinfo{title}{Energy stable flux reconstruction schemes for
  {advection }–{diffusion} problems},
\newblock \bibinfo{journal}{Computer Methods in Applied Mechanics and
  Engineering} \bibinfo{volume}{267} (\bibinfo{year}{2013})
  \bibinfo{pages}{400--417}.
\bibitem[{Sheshadri and Jameson(2016)}]{sheshadri2016stability}
\bibinfo{author}{A.~Sheshadri}, \bibinfo{author}{A.~Jameson},
\newblock \bibinfo{title}{On the stability of the flux reconstruction schemes
  on quadrilateral elements for the linear advection equation},
\newblock \bibinfo{journal}{Journal of Scientific Computing}
  \bibinfo{volume}{67} (\bibinfo{year}{2016}) \bibinfo{pages}{769--790}.
\bibitem[{Allaneau and Jameson(2011)}]{allaneau_connections_2011}
\bibinfo{author}{Y.~Allaneau}, \bibinfo{author}{A.~Jameson},
\newblock \bibinfo{title}{Connections between the filtered discontinuous
  {Galerkin} method and the flux reconstruction approach to high order
  discretizations},
\newblock \bibinfo{journal}{Computer Methods in Applied Mechanics and
  Engineering} \bibinfo{volume}{200} (\bibinfo{year}{2011})
  \bibinfo{pages}{3628--3636}.
\bibitem[{Gassner et~al.(2018)Gassner, Winters, Hindenlang, and
  Kopriva}]{gassner2018br1}
\bibinfo{author}{G.~J. Gassner}, \bibinfo{author}{A.~R. Winters},
  \bibinfo{author}{F.~J. Hindenlang}, \bibinfo{author}{D.~A. Kopriva},
\newblock \bibinfo{title}{The br1 scheme is stable for the compressible
  navier--stokes equations},
\newblock \bibinfo{journal}{Journal of Scientific Computing}
  \bibinfo{volume}{77} (\bibinfo{year}{2018}) \bibinfo{pages}{154--200}.
\bibitem[{Manzanero et~al.(2019)Manzanero, Rubio, Kopriva, Ferrer, and
  Valero}]{manzanero2019entropy}
\bibinfo{author}{J.~Manzanero}, \bibinfo{author}{G.~Rubio},
  \bibinfo{author}{D.~A. Kopriva}, \bibinfo{author}{E.~Ferrer},
  \bibinfo{author}{E.~Valero},
\newblock \bibinfo{title}{Entropy-stable discontinuous {Galerkin} approximation
  with summation-by-parts property for the incompressible
  navier-stokes/cahn-hilliard system},
\newblock \bibinfo{journal}{arXiv preprint arXiv:1910.11252}
  (\bibinfo{year}{2019}).
\bibitem[{Kopriva(2006)}]{kopriva2006metric}
\bibinfo{author}{D.~A. Kopriva},
\newblock \bibinfo{title}{Metric identities and the discontinuous spectral
  element method on curvilinear meshes},
\newblock \bibinfo{journal}{Journal of Scientific Computing}
  \bibinfo{volume}{26} (\bibinfo{year}{2006}) \bibinfo{pages}{301}.
\bibitem[{Thomas and Lombard(1979)}]{thomas1979geometric}
\bibinfo{author}{P.~D. Thomas}, \bibinfo{author}{C.~K. Lombard},
\newblock \bibinfo{title}{Geometric conservation law and its application to
  flow computations on moving grids},
\newblock \bibinfo{journal}{AIAA journal} \bibinfo{volume}{17}
  (\bibinfo{year}{1979}) \bibinfo{pages}{1030--1037}.
\bibitem[{Vinokur and Yee(2002)}]{vinokur2002extension}
\bibinfo{author}{M.~Vinokur}, \bibinfo{author}{H.~Yee},
\newblock \bibinfo{title}{Extension of efficient low dissipation high order
  schemes for 3-d curvilinear moving grids},
\newblock in: \bibinfo{booktitle}{Frontiers of Computational Fluid Dynamics
  2002}, \bibinfo{publisher}{World Scientific}, \bibinfo{year}{2002}, pp.
  \bibinfo{pages}{129--164}.
\bibitem[{Kopriva et~al.(2019)Kopriva, Hindenlang, Bolemann, and
  Gassner}]{kopriva2019free}
\bibinfo{author}{D.~A. Kopriva}, \bibinfo{author}{F.~J. Hindenlang},
  \bibinfo{author}{T.~Bolemann}, \bibinfo{author}{G.~J. Gassner},
\newblock \bibinfo{title}{Free-stream preservation for curved geometrically
  non-conforming discontinuous {Galerkin} spectral elements},
\newblock \bibinfo{journal}{Journal of Scientific Computing}
  \bibinfo{volume}{79} (\bibinfo{year}{2019}) \bibinfo{pages}{1389--1408}.
\bibitem[{Botti(2012)}]{botti2012influence}
\bibinfo{author}{L.~Botti},
\newblock \bibinfo{title}{Influence of reference-to-physical frame mappings on
  approximation properties of discontinuous piecewise polynomial spaces},
\newblock \bibinfo{journal}{Journal of Scientific Computing}
  \bibinfo{volume}{52} (\bibinfo{year}{2012}) \bibinfo{pages}{675--703}.
\bibitem[{{Del Rey Fern\'andez} et~al.(2019){Del Rey Fern\'andez}, Carpenter,
  Dalcin, Fredrich, Rojas, Winters, Gassner, Zampini, and
  Parsani}]{fernandez2019entropy}
\bibinfo{author}{D.~{Del Rey Fern\'andez}}, \bibinfo{author}{M.~H. Carpenter},
  \bibinfo{author}{L.~Dalcin}, \bibinfo{author}{L.~Fredrich},
  \bibinfo{author}{D.~Rojas}, \bibinfo{author}{A.~R. Winters},
  \bibinfo{author}{G.~J. Gassner}, \bibinfo{author}{S.~Zampini},
  \bibinfo{author}{M.~Parsani},
\newblock \bibinfo{title}{Entropy stable p-nonconforming discretizations with
  the summation-by-parts property for the compressible euler equations},
\newblock \bibinfo{journal}{arXiv preprint arXiv:1909.12536}
  (\bibinfo{year}{2019}).
\bibitem[{Jameson(2010{\natexlab{a}})}]{jameson_proof_2010}
\bibinfo{author}{A.~Jameson},
\newblock \bibinfo{title}{A {Proof} of the {Stability} of the {Spectral}
  {Difference} {Method} for {All} {Orders} of {Accuracy}},
\newblock \bibinfo{journal}{Journal of Scientific Computing}
  \bibinfo{volume}{45} (\bibinfo{year}{2010}{\natexlab{a}})
  \bibinfo{pages}{348--358}.
\bibitem[{Jameson(2010{\natexlab{b}})}]{jameson2010proof}
\bibinfo{author}{A.~Jameson},
\newblock \bibinfo{title}{A proof of the stability of the spectral difference
  method for all orders of accuracy},
\newblock \bibinfo{journal}{Journal of Scientific Computing}
  \bibinfo{volume}{45} (\bibinfo{year}{2010}{\natexlab{b}})
  \bibinfo{pages}{348--358}.
\bibitem[{Huynh(2009)}]{huynh_reconstruction_2009}
\bibinfo{author}{H.~T. Huynh},
\newblock \bibinfo{title}{A {Reconstruction} {Approach} to {High}-{Order}
  {Schemes} {Including} {Discontinuous} {Galerkin} for {Diffusion}},
\newblock \bibinfo{publisher}{American Institute of Aeronautics and
  Astronautics}, \bibinfo{year}{2009}. \DOIprefix\doi{10.2514/6.2009-403}.
\bibitem[{Witherden et~al.(2016)Witherden, Vincent, and
  Jameson}]{witherden2016high}
\bibinfo{author}{F.~Witherden}, \bibinfo{author}{P.~Vincent},
  \bibinfo{author}{A.~Jameson},
\newblock \bibinfo{title}{High-order flux reconstruction schemes},
\newblock in: \bibinfo{booktitle}{Handbook of numerical analysis},
  volume~\bibinfo{volume}{17}, \bibinfo{publisher}{Elsevier},
  \bibinfo{year}{2016}, pp. \bibinfo{pages}{227--263}.
\bibitem[{Chan(2019)}]{chan2019skew}
\bibinfo{author}{J.~Chan},
\newblock \bibinfo{title}{Skew-symmetric entropy stable modal discontinuous
  {Galerkin} formulations},
\newblock \bibinfo{journal}{Journal of Scientific Computing}
  \bibinfo{volume}{81} (\bibinfo{year}{2019}) \bibinfo{pages}{459--485}.
\bibitem[{Witherden et~al.(2014)Witherden, Farrington, and
  Vincent}]{witherden2014pyfr}
\bibinfo{author}{F.~D. Witherden}, \bibinfo{author}{A.~M. Farrington},
  \bibinfo{author}{P.~E. Vincent},
\newblock \bibinfo{title}{Pyfr: An open source framework for solving
  advection--diffusion type problems on streaming architectures using the flux
  reconstruction approach},
\newblock \bibinfo{journal}{Computer Physics Communications}
  \bibinfo{volume}{185} (\bibinfo{year}{2014}) \bibinfo{pages}{3028--3040}.
\bibitem[{Witherden and Vincent(2014)}]{witherden2014analysis}
\bibinfo{author}{F.~D. Witherden}, \bibinfo{author}{P.~E. Vincent},
\newblock \bibinfo{title}{An analysis of solution point coordinates for flux
  reconstruction schemes on triangular elements},
\newblock \bibinfo{journal}{Journal of Scientific Computing}
  \bibinfo{volume}{61} (\bibinfo{year}{2014}) \bibinfo{pages}{398--423}.
\bibitem[{Abe et~al.(2015)Abe, Haga, Nonomura, and Fujii}]{abe2015freestream}
\bibinfo{author}{Y.~Abe}, \bibinfo{author}{T.~Haga},
  \bibinfo{author}{T.~Nonomura}, \bibinfo{author}{K.~Fujii},
\newblock \bibinfo{title}{On the freestream preservation of high-order
  conservative flux-reconstruction schemes},
\newblock \bibinfo{journal}{Journal of Computational Physics}
  \bibinfo{volume}{281} (\bibinfo{year}{2015}) \bibinfo{pages}{28--54}.
\bibitem[{Johnen et~al.(2013)Johnen, Remacle, and
  Geuzaine}]{johnen2013geometrical}
\bibinfo{author}{A.~Johnen}, \bibinfo{author}{J.-F. Remacle},
  \bibinfo{author}{C.~Geuzaine},
\newblock \bibinfo{title}{Geometrical validity of curvilinear finite elements},
\newblock \bibinfo{journal}{Journal of Computational Physics}
  \bibinfo{volume}{233} (\bibinfo{year}{2013}) \bibinfo{pages}{359--372}.
\bibitem[{Turner et~al.(2018)Turner, Peir{\'o}, and
  Moxey}]{turner2018curvilinear}
\bibinfo{author}{M.~Turner}, \bibinfo{author}{J.~Peir{\'o}},
  \bibinfo{author}{D.~Moxey},
\newblock \bibinfo{title}{Curvilinear mesh generation using a variational
  framework},
\newblock \bibinfo{journal}{Computer-Aided Design} \bibinfo{volume}{103}
  (\bibinfo{year}{2018}) \bibinfo{pages}{73--91}.
\bibitem[{Shi-Dong and Nadarajah(2021)}]{shi2021full}
\bibinfo{author}{D.~Shi-Dong}, \bibinfo{author}{S.~Nadarajah},
\newblock \bibinfo{title}{Full-space approach to aerodynamic shape
  optimization},
\newblock \bibinfo{journal}{Computers \& Fluids}  (\bibinfo{year}{2021})
  \bibinfo{pages}{104843}.
\bibitem[{Wu et~al.(2021)Wu, Kubatko, and Chan}]{wu2021high}
\bibinfo{author}{X.~Wu}, \bibinfo{author}{E.~J. Kubatko},
  \bibinfo{author}{J.~Chan},
\newblock \bibinfo{title}{High-order entropy stable discontinuous galerkin
  methods for the shallow water equations: curved triangular meshes and gpu
  acceleration},
\newblock \bibinfo{journal}{Computers \& Mathematics with Applications}
  \bibinfo{volume}{82} (\bibinfo{year}{2021}) \bibinfo{pages}{179--199}.
\bibitem[{Hennemann et~al.(2021)Hennemann, Rueda-Ram{\'\i}rez, Hindenlang, and
  Gassner}]{hennemann2021provably}
\bibinfo{author}{S.~Hennemann}, \bibinfo{author}{A.~M. Rueda-Ram{\'\i}rez},
  \bibinfo{author}{F.~J. Hindenlang}, \bibinfo{author}{G.~J. Gassner},
\newblock \bibinfo{title}{A provably entropy stable subcell shock capturing
  approach for high order split form {DG} for the compressible {Euler}
  equations},
\newblock \bibinfo{journal}{Journal of Computational Physics}
  \bibinfo{volume}{426} (\bibinfo{year}{2021}) \bibinfo{pages}{109935}.

\end{thebibliography}

\end{document}